\documentclass[11pt]{amsart}
\usepackage{amssymb,latexsym,graphicx,amscd}
\usepackage{lscape}
\usepackage[all]{xy}
\usepackage{color}
\usepackage{epic,eepic}
\usepackage{xspace}
\newtheorem{THM}{Theorem}
\newtheorem{Thm}{Theorem}[section]
\newtheorem{Lem}[Thm]{Lemma}
\newtheorem{Cor}[Thm]{Corollary}
\newtheorem{Prop}[Thm]{Proposition}

\newtheorem{Def}[Thm]{Definition}

\newcommand{\Z}{\mathbb{Z}}
\newcommand{\Q}{\mathbb{Q}}
\newcommand{\N}{\mathbb{N}}
\newcommand{\C}{\mathbb{C}}

\newcommand{\df}{\colon}

\newcommand{\F}{{\mathcal F}}
\newcommand{\tF}{{\widetilde{\mathcal F}}}
\newcommand{\tG}{{\widetilde{\mathcal G}}}

\newcommand{\sur}{{\mathrm{sur}}}
\newcommand{\Gr}{{\rm Gr}}
\newcommand{\U}{{\mathcal U}}
\newcommand{\G}{{\mathcal G}}

\newcommand{\cD}{{\mathcal D}}
\newcommand{\cE}{{\mathcal E}}

\newcommand{\cI}{{\mathcal I}}
\newcommand{\cM}{{\mathcal M}}
\newcommand{\cA}{{\mathcal A}}
\newcommand{\cS}{{\mathcal S}}
\newcommand{\cP}{{\mathcal P}}
\newcommand{\cV}{{\mathcal V}}
\newcommand{\cG}{{\mathcal G}}
\newcommand{\cX}{{\mathcal X}}

\newcommand{\stCC}{{\underline{\mathcal C}}}

\newcommand{\g}{\mathfrak{g}}
\newcommand{\n}{\mathfrak{n}}

\newcommand{\ov}{\overline}

\newcommand{\mm}{\mathbf{m}}
\newcommand{\dd}{{\mathbf d}}
\newcommand{\bb}{{\mathbf b}}
\newcommand{\ii}{{\mathbf i}}
\newcommand{\jj}{{\mathbf j}}
\newcommand{\ff}{{\mathbf f}}

\newcommand{\LL}{\Lambda}
\newcommand{\GG}{\Gamma}

\newcommand{\vpi}{\varpi}
\newcommand{\vph}{\varphi}
\newcommand{\vphi}{\varphi}

\newcommand{\orb}{\mathcal O}

\newcommand{\md}{\operatorname{mod}}
\newcommand{\nil}{\operatorname{nil}}
\newcommand{\rep}{\operatorname{rep}}

\newcommand{\add}{\operatorname{add}}
\newcommand{\Gen}{\operatorname{Fac}}
\newcommand{\Cogen}{\operatorname{Sub}}

\newcommand{\soc}{\operatorname{soc}}
\newcommand{\tp}{\operatorname{top}}
\newcommand{\rad}{\operatorname{rad}}
\newcommand{\wt}{\operatorname{wt}}

\newcommand{\Hom}{\operatorname{Hom}}
\newcommand{\Ext}{\operatorname{Ext}}
\newcommand{\End}{\operatorname{End}}

\newcommand{\Ima}{\operatorname{Im}}
\newcommand{\Ker}{\operatorname{Ker}}
\newcommand{\Coker}{\operatorname{Coker}}

\newcommand{\dimv}{\underline{\dim}}

\newcommand{\Span}{\operatorname{Span}}

\newcommand{\ebrace}[1]{\langle #1 \rangle}

\newcommand{\cyc}{{\operatorname{cyc}}}
\newcommand{\op}{{\rm op}}

\newcommand{\irr}{\operatorname{Irr}}
\newcommand{\srirr}{\operatorname{Irr^{sr}}}

\newcommand{\bsm}{\begin{smallmatrix}}
\newcommand{\esm}{\end{smallmatrix}}
\newcommand{\bbm}{\begin{matrix}}
\newcommand{\ebm}{\end{matrix}}

\newcommand{\GL}{\operatorname{GL}}

\newtheorem{Rem}[Thm]{Remark}
\newcommand{\ra}{\rightarrow}

\newcommand{\Rmi}{{R_-}}

\newcommand{\Rma}{{R_{\mathrm{max}}}}
\newcommand{\cDp}{{\cD_{\text{perf}}}}

\newcommand{\hvph}{\hat{\vph}}
\newcommand{\hx}{\hat{x}}
\newcommand{\hy}{\hat{y}}

\newcommand{\sx}{\underline{x}}

\newcommand{\ba}{\mathbf{a}}
\newcommand{\bd}{\mathbf{d}}
\newcommand{\be}{\mathbf{e}}
\newcommand{\bg}{\mathbf{g}}
\newcommand{\bh}{\mathbf{h}}
\newcommand{\bi}{\mathbf{i}}
\newcommand{\bt}{\mathbf{t}}

\newcommand{\sHom}{\underline{\Hom}}
\newcommand{\sEnd}{\underline{\End}}
\newcommand{\scE}{{\underline{\cE}}}
\newcommand{\scC}{{\underline{\cC}}}

\newcommand{\codim}{\operatorname{codim}}

\newcommand{\Gam}{\Gamma}
\newcommand{\Lam}{\Lambda}
\newcommand{\Sig}{\Sigma}
\newcommand{\Ome}{\Omega}

\newcommand{\sGG}{\underline{\Gamma}}

\newcommand{\ie}{i.e.\@\xspace}

\newcommand{\Cpl}{\operatorname{Null}}

\newcommand{\tGam}{\widetilde{\Gam}}
\newcommand{\tW}{\widetilde{W}}
\newcommand{\tM}{\widetilde{M}}
\newcommand{\sGam}{\underline{\Gam}}
\newcommand{\stGam}{\widetilde{\sGam}}

\newcommand{\fm}{\mathfrak{m}}
\newcommand{\epow}[1]{\langle\langle #1\rangle\rangle}
\newcommand{\CC}{\mathbb{C}}

\newcommand{\cB}{\mathcal{B}}
\newcommand{\cC}{\mathcal{C}}
\newcommand{\cF}{\mathcal{F}}
\newcommand{\sB}{\underline{B}}

\newcommand{\cU}{\mathcal{U}}
\newcommand{\cT}{\mathcal{T}}
\newcommand{\cY}{\mathcal{Y}}
\newcommand{\cZ}{\mathcal{Z}}

\newcommand{\alp}{\alpha}
\newcommand{\lam}{\lambda}

\newcommand{\bul}{\bullet}
\newcommand{\Bi}{\operatorname{Im}}

\newcommand{\Fac}{\operatorname{Fac}}
\newcommand{\pex}{{\phantom{1}}}
\newcommand{\pexx}{{\phantom{-1}}}

\numberwithin{equation}{section}

\begin{document}


\title[Generic bases for cluster algebras and the Chamber Ansatz]
{Generic bases for cluster algebras\\ and the Chamber Ansatz}

\author{Christof Gei{\ss}}
\address{Christof Gei{\ss}\newline
Instituto de Matem\'aticas\newline
Universidad Nacional Aut{\'o}noma de M{\'e}xico\newline
Ciudad Universitaria\newline
04510 M{\'e}xico D.F.\newline
M{\'e}xico}
\email{christof@math.unam.mx}

\author{Bernard Leclerc}
\address{Bernard Leclerc\newline
LMNO, Univ. de Caen\newline
CNRS, UMR 6139\newline
F-14032 Caen Cedex\newline
France}
\email{leclerc@math.unicaen.fr}

\author{Jan Schr\"oer}
\address{Jan Schr\"oer\newline
Mathematisches Institut\newline
Universit\"at Bonn\newline
Endenicher Allee 60\newline
53115 Bonn\newline
Germany}
\email{schroer@math.uni-bonn.de}



\begin{abstract}
Let $Q$ be a finite quiver without oriented cycles, and let
$\LL$ be the corresponding preprojective algebra.
Let $\g$ be the Kac-Moody Lie algebra with Cartan datum
given by $Q$, and let $W$ be its Weyl group.
With $w \in W$, there is associated a unipotent cell $N^w$ of the Kac-Moody group with
Lie algebra $\g$.
In previous work we proved that 
the coordinate ring $\C[N^w]$ of $N^w$ is a cluster algebra in
a natural way. 
A central role is played by generating functions
$\vph_X$ of Euler characteristics of certain varieties of
partial composition series of $X$,
where $X$ runs through all modules in a Frobenius subcategory
$\cC_w$ of the category of nilpotent $\LL$-modules.
The first aim of this article is to compare the function $\vph_X$
with the so-called cluster character of $X$, which is 
defined in terms of the Euler characteristics of
quiver Grassmannians.
We show that
for every $X$ in $\cC_w$, $\vph_X$ coincides,
after an appropriate change of variables, 
 with the cluster character of Fu and Keller
associated with $X$ using any cluster-tilting
object $T$ of $\cC_w$.
A crucial ingredient of the proof is the construction of an isomorphism
between varieties of partial composition series of $X$ and certain quiver
Grassmannians.
This isomorphism is obtained in a very general setup and should be of
interest in itself.
Another important tool of the proof is a representation-theoretic 
version of the Chamber Ansatz of Berenstein, Fomin and Zelevinsky,
adapted to Kac-Moody groups.
As an application, we get a new description of a generic basis of the cluster algebra
$\cA(\sGG_T)$ obtained from $\C[N^w]$ via specialization of coefficients to 1. 
Here \emph{generic} refers to the representation varieties
of a quiver potential arising from the cluster-tilting module $T$.
For the special case of coefficient-free acyclic cluster algebras this
proves a conjecture by Dupont.
\end{abstract}

\maketitle

\setcounter{tocdepth}{1}

\tableofcontents

\parskip2mm



\section{Introduction and main results}\label{section1}


\subsection{}
In the recent literature on cluster algebras, calculations of 
Euler characteristics of certain varieties related to quiver
representations play a prominent role.
In \cite{GLSRigid,GLSPartial,GLSUni2}, cluster variables of
coordinate rings of unipotent cells of algebraic groups and Kac-Moody
groups were shown to be expressible in terms of Euler characteristics
of varieties of flags of submodules of preprojective 
algebra representations. In another direction, starting with a formula of Caldero
and Chapoton \cite{CC}, the coefficients of the Laurent polynomial
expansions of cluster variables of some cluster algebras were 
described as Euler characteristics of Grassmannians of submodules 
of quiver representations.
This was first achieved for acyclic cluster algebras \cite{CK1}, 
later for cluster algebras admitting a 2-Calabi-Yau categorification
\cite{P,FK}, and more recently for general antisymmetric cluster algebras
of geometric type \cite{DWZ2}. There is a 
posterior but essentially different proof in \cite{Pl}, see also
\cite{N}. 
The first aim of this paper is to compare these two types of formulas
for the large class of cluster algebras 
which can be realized as coordinate rings of unipotent 
cells of Kac-Moody groups.

To do this, we will return to the very source of cluster algebras,
namely to the Chamber Ansatz of Berenstein, Fomin and Zelevinsky
\cite{BFZ,BZ}, which describes parametrizations of Lusztig's 
totally positive parts of unipotent subgroups and Schubert varieties. 
The second aim of this paper is to provide a new understanding of
the Chamber Ansatz formulas in terms of representations of preprojective
algebras, together with a generalization to the Kac-Moody case.
In particular the mysterious twist automorphisms of the unipotent cells
needed in these formulas turn out to be just shadows of the Auslander-Reiten
translations of the corresponding Frobenius categories of modules
over the preprojective algebras. Our treatment of the Chamber Ansatz
shows that the numerators of the twisted minors of \cite{BFZ,BZ}
form a cluster, and that the Laurent expansions with respect to these
special clusters have coefficients equal to Euler characteristics
of varieties of flags of submodules of preprojective
algebra representations. This provides the desired link between the two types of
Euler characteristics mentioned above, and it allows us to show that
the cluster characters of Fu and Keller \cite{FK} coincide after an
appropriate change of variables with the $\varphi$-functions of \cite{GLSRigid,GLSUni2}.

Finally, our third aim is to exploit these results for studying natural
bases of cluster algebras containing the cluster monomials. 
We consider the class of coefficient-free cluster algebras obtained
by specializing to 1 the coefficients of the cluster algebra structures on 
unipotent cells.
In \cite[Section~15.6]{GLSUni2} we have found such bases, consisting of appropriate 
subsets of Lusztig's dual semicanonical bases.
Here, using the above connection with Fu-Keller cluster characters, 
we give a new description of the same bases in terms of module varieties
of endomorphism algebras of cluster-tilting modules.
In the special case when the cluster algebra is acyclic, this proves 
Dupont's generic basis conjecture~\cite{D}.
In general, the elements of these bases are 
generating functions of Euler characteristics of quiver Grassmannians, 
at generic points of some particular irreducible components of the module varieties. 
These special irreducible components
can be characterized in terms of the new $E$-invariant introduced
by Derksen, Weyman and Zelevinsky \cite{DWZ2} for representations of
quivers with potential, and one may therefore conjecture that a similar description 
of a generic basis can be extended to any antisymmetric cluster algebra.

\subsection{}\label{introV}
To state our results more precisely, we need to introduce some notation.
Let $Q$ be a finite quiver with vertex set $\{1,\ldots, n\}$ and without oriented cycles.
Denote by $\LL$ the corresponding preprojective algebra.
Let $\g$ be the Kac-Moody Lie algebra with 
Cartan datum given by $Q$, and let $W$ be the Weyl group of~$\g$.
The graded dual $U(\n)_{\rm gr}^*$ of the universal enveloping algebra $U(\n)$
of the positive part $\n$ of $\g$ can be identified with the coordinate ring 
$\C[N]$ of an associated pro-unipotent pro-group $N$ with Lie algebra $\n$.

For $w\in W$, let $N^w := N \cap (B_-wB_-)$ be the corresponding unipotent cell
in $N$, where $B_-$ denotes the standard negative Borel subgroup of the Kac-Moody 
group $G$ attached to $\g$.
Here we use the same notation as in \cite{GLSUni2}.
For details on Kac-Moody groups we refer to
\cite[Sections~6 and~7.4]{Ku}.
Let $x_i(t)$ denote the one-parameter subgroup of $N$ associated
to the simple root $\alp_i$. 
For each reduced expression $\ii = (i_r,\ldots,i_1)$ of $w$, the map 
\[
\sx_\ii : (t_r,\ldots,t_2,t_1) \mapsto 
x_{i_r}(t_r) \cdots x_{i_2}(t_2)x_{i_1}(t_1) 
\]
gives a birational isomorphism from $\C^r$ to $N^w$.
In \cite{GLSUni2} we have described a cluster algebra structure on $\C[N^w]$
in terms of the representation theory of the preprojective algebra $\LL$.

For a nilpotent $\LL$-module $X$ and $\ba = (a_r,\ldots,a_1) \in \N^r$
let $\F_{\ii,\ba,X}$ be the projective variety
of flags 
$$
X_\bullet = (0 = X_r \subseteq \cdots \subseteq X_1 \subseteq X_0 = X)
$$
of submodules of $X$ such that $X_{k-1}/X_k\cong S_{i_k}^{a_k}$ 
for all $1 \le k \le r$, where $S_j$ denotes the one-dimensio\-nal
$\LL$-module supported on the vertex $j$ of $Q$.
The varieties $\F_{\bi,\ba,X}$  were first introduced by Lusztig \cite{L1}
for his Lagrangian construction of $U(\n)$. Dualizing Lusztig's construction, we can associate
with $X$ a regular function $\varphi_X \in \C[N]$ satisfying
\[
\vph_X(\sx_\ii(\bt)) = \sum_{\ba \in \N^r} \chi(\F_{\ii,\ba,X}) \bt^\ba.
\]
Here $\bt=(t_r,\ldots,t_1)\in\C^r$, $\bt^\ba := t_r^{a_r} \cdots t_2^{a_2}t_1^{a_1}$, 
and $\chi$ denotes the topological Euler characteristic.

Buan, Iyama, Reiten, and Scott \cite{BIRS} have attached to $w$
a 2-Calabi-Yau Frobenius subcategory $\cC_w$ of the category of
finite-dimensional nilpotent $\LL$-modules. (The same categories
were studied independently in \cite{GLSUni1} for special elements
$w$ called adaptable.) In \cite{GLSUni2} we showed that the $\C$-span
of 
\[
 \{\vphi_X \mid X\in\cC_w\}
\]
is a subalgebra of $\C[N]$, which becomes isomorphic to $\C[N^w]$
after localization at the multiplicative subset 
$\{\vphi_P \mid P \mbox{ is $\cC_w$-projective-injective} \}$.
Moreover, we showed that 
$\C[N^w]$ carries a 
cluster algebra structure, whose cluster variables are of the form $\vphi_X$
for indecomposable modules $X$ in $\cC_w$ without self-extension.
In Section~\ref{reminder} we explain
this in more detail.
 
The category $\cC_w$ comes with a remarkable module $V_\bi$ 
for each reduced expression $\bi$ of~$w$ (see \cite[Section~III.2]{BIRS}, 
\cite[Section~2.4]{GLSUni2}). 
The $\vphi$-functions of the indecomposable
direct summands of $V_\bi$ are some generalized minors on $N$ which
form a natural initial cluster of $\C[N^w]$.
We introduce the new module
\[
 W_\bi := I_w \oplus \Omega_w(V_\bi),
\]
where $\Omega_w=\tau_w^{-1}$ is the inverse Auslander-Reiten translation 
of $\cC_w$, and $I_w$
is the direct sum of the indecomposable $\cC_w$-projective-injectives. 
For a $\LL$-module $X$, the set $\Ext^1_\LL(W_\ii,X)$ is in a
natural way a left module over the stable endomorphism algebra
\[
\scE := \sEnd_{\cC_w}(W_\ii)^\op \cong \sEnd_{\cC_w}(V_\ii)^\op.
\]
Denote by $\Gr_\bd^\scE(\Ext^1_\LL(W_\ii,X))$ the projective variety
of $\scE$-submodules of $\Ext^1_\LL(W_\ii,X)$ with dimension vector $\bd$,
a so-called quiver Grassmannian. 
Our first main result is

\begin{THM}\label{THM1}
For $X \in \cC_w$ and all $\ba \in \N^r$, there is an isomorphism of algebraic varieties
\[
\F_{\bi,\ba,X} \cong \Gr_{d_{\bi,X}(\ba)}^\scE(\Ext^1_\LL(W_\ii,X)),
\]
where $d_{\bi,X}$ is an explicit bijection from 
$\{ \ba \mid \F_{\bi,\ba,X} \not= \varnothing \}$
to $\{ \bd \mid \Gr_\bd^\scE(\Ext^1_\LL(W_\ii,X)) \not= \varnothing \}$.
\end{THM}

It follows easily that the set $\{ \ba \mid 
\chi(\F_{\bi,\ba,X}) \not= 0 \}$
has a unique element if and only if $\Ext^1_\LL(W_\ii,X) = 0$.
Now by construction, $W_\bi$ is a cluster-tilting module of $\cC_w$, that is,
$\Ext^1_\LL(W_\bi,X) = 0$ if and only if $X$ belongs to the additive hull
$\add(W_\ii)$ of $W_\ii$. Moreover, in this case $\F_{\bi,\ba,X}$ is reduced
to a point. Hence Theorem~\ref{THM1} has the following important consequence:

\begin{THM}\label{THM2}
For $X\in\cC_w$,
the polynomial function $\bt \mapsto \vphi_X(\sx_\ii(\bt))$ is reduced 
to a single monomial $\bt^\ba$ if and only if $X\in\add(W_\ii)$. 
\end{THM}

\subsection{}
Let $W_{\bi,1},\ldots,W_{\bi,r}$ denote the indecomposable direct summands of $W_\bi$.
The $r$-tuple of regular functions $(\vphi_{W_{\bi,1}},\ldots,\vphi_{W_{\bi,r}})$ is a cluster of $\C[N^w]$,
and it follows from Theorem~\ref{THM2} that the $\vphi_{W_{\bi,k}}(\sx_\ii(\bt))$
are monomials in the variables $t_1,\ldots,t_r$. 
Inverting this monomial transformation yields expressions of the $t_k$'s
as explicit rational functions on $N^w$, a result originally called the Chamber Ansatz 
by Berenstein, Fomin and Zelevinsky \cite{BFZ} in type $A_n$, because of a convenient
description of these formulas in terms of chambers in a wiring diagram.
To present these formulas in the general Kac-Moody setting, we need more notation.
By construction, the summands $V_{\bi,k}$ of $V_\bi$ are related to the modules $W_{\bi,k}$
by short exact sequences
\[
 0\to W_{\bi,k} \to P(V_{\bi,k}) \to V_{\bi,k} \to 0
\]
where for $X\in\cC_w$, $P(X)$ denotes the projective cover in $\cC_w$.
We set 
\[
\vphi'_{V_{\bi,k}} := \frac{\vphi_{W_{\bi,k}}}{\vphi_{P(V_{\bi,k})}},
\]
a Laurent monomial in the $\vphi_{W_{\bi,k}}$ (since $\add(W_\ii)$ contains all $\cC_w$-projectives). 
As will be explained in Section~\ref{sect_twist} below, the regular functions 
$\vphi'_{V_{\bi,k}}$ on $N^w$ are the twisted generalized minors of \cite{BZ} corresponding to $\bi$
(in the Dynkin case).

Denote by $q(i,j)$ the number of edges between two vertices $i$ and $j$ 
of the underlying unoriented graph of the quiver $Q$.
For $1 \le k \le r$, put
\begin{equation}\label{ch5}
C_{\ii,k} := \frac{1}{\vph_{V_{\ii,k}}'\vph_{V_{\ii,k^-(i_k)}}'} \cdot
\prod_{j=1}^n \left(\vph_{V_{\ii,k^-(j)}}'\right)^{q(i_k,j)},
\end{equation}
where
$k^-(j) := \max\{ 0,1 \le s \le k-1 \mid i_s = j \}$ and
$V_{\ii,0}$ is by convention the zero module.

\begin{THM}\label{THM5}
For $1 \le k \le r$ and $\bt = (t_r,\ldots,t_1)$ we have 
$C_{\ii,k}(\sx_\ii(\bt))=t_k$.
Therefore, for $X \in \cC_w$ we get an equality in $\C[N^w]$:
\begin{equation}\label{Wclusterexpansion}
\vph_X = \sum_{\ba \in \N^r} \chi(\F_{\ii,\ba,X}) 
C_{\ii,r}^{a_r} \cdots C_{\ii,2}^{a_2}C_{\ii,1}^{a_1}.
\end{equation}
\end{THM}

\subsection{}\label{section1.4}
Using Theorem~\ref{THM1}, we now want to compare Equation~(\ref{Wclusterexpansion}) with
similar formulas of Fu and Keller.
To simplify our notation, we define
\begin{eqnarray*}
R &:= &\{ 1,2,\ldots,r \},\\
\Rma &:= &\{ k \in R \mid \mbox{there is no } k < s \le r \mbox{ with } i_s = i_k \},\\
\Rmi &:= &R \setminus \Rma.
\end{eqnarray*}
Let $T=T_1\oplus\cdots\oplus T_r$ be a basic cluster-tilting module in $\cC_w$,
where the numbering is chosen so that $T_k$ is $\cC_w$-projective-injective for $k\in \Rma$.
Assume that $(\vphi_{T_{1}},\ldots,\vphi_{T_{r}})$ is a cluster of $\C[N^w]$,
\ie that it can be obtained from $(\vphi_{V_{\bi,1}},\ldots,\vphi_{V_{\bi,r}})$
by a sequence of mutations. 
In this case, $T$ is called $V_\ii$-{\it reachable}.
(One conjectures that this is always the case.)
The endomorphism algebra $\cE_T := \End_\LL(T)^\op$ has  global dimension $3$, 
see \cite[Proposition~2.19]{GLSUni2}. 
Thus we may consider
$$
B^{(T)}:= (B_{k,l}^{(T)})_{k,l \in R} := \left((\dim\Hom_\LL(T_k,T_l))_{k,l \in R}\right)^{-t},
$$
the matrix of the Ringel bilinear form for $\cE_T$. 
(For a matrix $B$, we denote the inverse of its transpose by $B^{-t}$.)

For a general 2-Calabi-Yau Frobenius category $\cC$ with a cluster-tilting object, 
Fu and Keller \cite[Section~3]{FK} 
(extending previous work of Palu \cite{P}) have attached to every object of
$\cC$ a Laurent polynomial called its cluster character.
When applied to the category $\cC_w$ and the cluster-tilting object $T$, 
the formula for this cluster character 
can be written as
\begin{equation} \label{eq:FK2}
\theta^T_X := \vph_T^{(\dimv\Hom_\LL(T,X)) \cdot B^{(T)}}
\cdot
\sum_{\bd \in \N^{\Rmi}}
\chi(\Gr_\bd^{\scE_T}(\Ext_\LL^1(T,X)))\,\hvph_T^\bd\qquad (X\in\cC_w).
\end{equation}
Here we use the abbreviations
\[
\begin{array}{ccccl}
\vphi_T^\bg &:=& \prod_{k\in R_\pex} \vphi_{T_k}^{g_k} 
&& \mbox{for } \bg = (g_1,g_2,\ldots ,g_r)\in\Z^r,\\[2.5mm]
\hvph_{T,k} &:=& \prod_{l\in R_\pex} \vphi_{T_l}^{B_{l,k}^{(T)}}
&& \mbox{for } k \in \Rmi,\\[3.1mm]
\hvph_T^{\bd} &:= &\prod_{k\in\Rmi} \hvph_{T,k}^{d_k} 
&& \mbox{for } \bd = (d_k)_{k\in\Rmi}\in\N^{\Rmi}.
\end{array}
\]
By \cite[Theorem 4.3]{FK} and \cite[Theorem 3.3]{GLSUni2}, the cluster variables of $\C[N^w]$
are of the form $\theta^T_X$ for indecomposable rigid modules $X$
of $\cC_w$, and (\ref{eq:FK2}) gives therefore a 
representation-theoretic description of their cluster expansions with respect to
the cluster $(\vphi_{T_{1}},\ldots,\vphi_{T_{r}})$. 
However, for an arbitrary $X\in\cC_w$ not
much is known about the function~$\theta^T_X$.
For instance it is \emph{a priori} only a rational function on $N^w$. 
Using Theorem~\ref{THM1} and the Chamber Ansatz Theorem~\ref{THM5}, we
prove our next main result:

\begin{THM} \label{THM3}
For every $X \in \cC_w$ we have
$$
\theta^T_X = \vphi_X.
$$ 
In particular, $\theta^T_X$ is a regular function on $N^w$
for every $X \in \cC_w$, 
that is, the image of the cluster character $X \mapsto \theta^T_X$
is in the cluster algebra $\C[N^w]$.
\end{THM}

\subsection{}\label{section1.5}
In the last part of this paper, we deduce from 
Theorem~\ref{THM3} a new description of a generic basis for the coefficient-free
cluster algebra obtained from $\C[N^w]$ by specializing to 1 the functions
$\vphi_P$ for all $\cC_w$-projective-injectives $P$.
(This algebra can be seen as the coordinate ring 
of the subvariety $N \cap (N_-wN_-)$ of $N^w$, but we will 
not use it.) In \cite[Section~15.6]{GLSUni2} we have already described 
such a basis in terms of generic modules over the preprojective algebra
$\LL$. Here we want to express it in terms of generic modules
over the stable endomorphism algebra $\scE_T$ of the cluster-tilting
module $T$. 

The quiver $\sGG_T$ of $\scE_T$ has the set $\Rmi$ as vertices,
with $k \in \Rmi$
corresponding to $T_k$, 
and it has
$[B_{l,k}^{(T)}]_+$  arrows from $k$ to $l$, 
where we write for short $[z]_+ = \max(z,0)$.
We consider the cluster algebra $\cA(\sGG_T) \subset \C((x_k)_{k \in \Rmi})$ 
with initial seed $((x_k)_{k \in \Rmi},\,\sGG_T)$.
We have a unique ring homomorphism $\Pi_T\df \C[N^w] \to \C((x_k)_{k \in \Rmi})$
such that $\Pi_T(\vphi_{T_k})=x_k$ for $k\in \Rmi$,
and $\Pi_T(\vphi_{T_k})=1$ for $k\in \Rma$.
The homomorphism $\Pi_T$ restricts to an
epimorphism 
$\C[N^w]\to \cA(\sGG_T)$, which we also denote by $\Pi_T$.

Following Palu \cite{P}, for an $\scE_T$-module $Y$ we put
\begin{equation}
\psi_Y := x^{\bg_Y} \cdot \sum_{\bd \in \N^{\Rmi}}
\chi(\Gr^{\scE_T}_{\bd}(Y))\,\hx_T^{\bd},
\end{equation}
where
$$
\bg_Y := (g_k)_{k \in \Rmi} := \left(\dim \Ext^1_{\scE_T}(S_k,Y) - 
\dim \Hom_{\scE_T}(S_k,Y)\right)_{k \in \Rmi}
$$
and
$$
x^{\bg_Y} := \prod_{k \in \Rmi} x_k^{g_k}, 
\qquad
\hx_{T,k}  := \prod_{l \in \Rmi} x_l^{B^{(T)}_{l,k}}, 
\qquad
\hx_T^\bd := \prod_{k \in \Rmi} \hx_{T,k}^{d_k}.
$$ 
(Here $S_k$, $k \in \Rmi$ are the simple $\scE_T$-modules.)
In fact, if $Y = \Ext_\LL^1(T,X)$ for some $X\in\cC_w$, in view of Theorem~\ref{THM3} we have $\psi_Y=\Pi_T(\vphi_X)$.

For $\bd \in \N^{\Rmi}$ let $\md(\scE_T,\bd)$ be the affine variety of representations
of $\scE_T$ with dimension vector $\bd$. 
It will be convenient to consider
$\md(\scE_T,\bd)$ with the \emph{right} action of 
$$
\GL_\bd := \prod_{k \in \Rmi} \GL_{\bd(k)}(\C)
$$ 
by conjugation.
For each irreducible
component $\cZ$ of $\md(\scE_T,\bd)$ 
there is a dense open subset $\cU \subseteq \cZ$
such that for all $U,U' \in \cU$ we have 
$\psi_U = \psi_{U'}$.
Define $\psi_\cZ := \psi_U$,
where $U \in \cU$.
An irreducible component
$\cZ$ of $\md(\scE_T,\bd)$ is called \emph{strongly reduced} if there is a 
dense open subset
$\cU \subseteq \cZ$ such that
$$
\codim_\cZ(U.\GL_\bd) = \dim \Hom_{\scE_T}\left(\tau_{\scE_T}^{-1}(U), U\right)
$$
for all $U \in \cU$,
where $\tau_{\scE_T}$ denotes the Auslander-Reiten translation of $\md(\scE_T)$.
It follows from Voigt's Lemma 
\cite[Proposition~1.1]{G} that strongly reduced components are (scheme-theoretically) 
generically reduced, hence the name.
But contrary to what the terminology might suggest,
being strongly reduced is not a property of the scheme $\cZ$
equipped with its $\GL_\bd$-action, since the definition uses additionally
the representation theory of the algebra $\scE_T$.
Note also that $\scE_T$ is given by a quiver with 
potential~\cite{BIRSm}, and that
$$
\dim \Hom_{\scE_T}(\tau_{\scE_T}^{-1}(U),U) = E^{\mathrm{inj}}(U)
$$ 
is the E-invariant defined in \cite{DWZ2}.

Let $\irr(\md(\scE_T,\bd))$ be the set of irreducible components of $\md(\scE_T,\bd)$,
and set 
$$
\irr(\scE_T) := \bigcup_{\bd \in \N^{\Rmi}} \irr(\md(\scE_T,\bd)).
$$
Let $\srirr(\scE_T)$ denote the set of all strongly reduced irreducible 
components in $\irr(\scE_T)$.
For $\cZ \in \irr(\scE_T,\bd)$ define
$\Cpl(\cZ) := \{ \mm \in \N^\Rmi \mid \mm(k) = 0 \mbox{ if } \bd(k) \not = 0 \}$.

Finally, let us denote by $\widetilde{\cS}_w^*$ the dual semicanonical basis of
$\C[N^w]$ constructed in \cite{GLSUni2}.
We can now state

\begin{THM}\label{THM4}
The set
$$
\cG_w^T := 
\left\{ x^\mm \cdot \psi_\cZ \mid
\cZ \in \srirr(\scE_T),\ \mm \in \Cpl(\cZ) \right\}
$$
is a basis of the cluster algebra $\cA(\sGG_T)$.
It is equal to the image of the dual semicanonical basis $\widetilde{\cS}_w^*$
under $\Pi_T\df \C[N^w] \to \cA(\sGG_T)$.
\end{THM}

Each finite-dimensional path algebra
is isomorphic to $\scE_T$ for some appropriate $\LL$, $w$ and 
$T$, see~\cite[Section~16]{GLSUni2}. 
In this case, $\md(\scE_T,\bd)$ is an (irreducible) affine
space for all $\bd$, and it is easy to see that $\md(\scE_T,\bd)$ is strongly reduced. 
Thus Theorem~\ref{THM4} implies Dupont's conjecture \cite[Conjecture~6.1]{D}. 
On the other hand, even if $\scE_T$ is not hereditary
but mutation equivalent to an acyclic quiver, it is quite easy to find
examples of irreducible components of varieties $\md(\scE_T,\bd)$
which are not strongly reduced.

Since Theorem~\ref{THM4} gives a description of the generic basis $\cG_w^T$ 
of $\cA(\sGG_T)$ entirely in terms of the varieties of representations 
of the algebra $\scE_T$, it is natural
in view of \cite{DWZ2} to ask if the first statement of Theorem~\ref{THM4} 
generalizes to other classes of cluster algebras.

\subsection{}\label{sect_twist}
The paper closes with our categorical interpretation of the twist automorphisms
of the unipotent cells, introduced by Berenstein, Fomin and Zelevinsky in 
connection with the Chamber Ansatz.
For $x\in N^w$, the intersection $N \cap (B_-wx^T)$ consists of a unique
element, which, following \cite{BFZ,BZ,GLSUni2}, we denote by $\eta_w(x)$. 
(The anti-automorphism $g \mapsto g^T$ of the Kac-Moody
group is defined in \cite[Section~7.1]{GLSUni2}.
For more details on $\eta_w$ we refer to 
\cite[Section~8]{GLSUni2}.)
The map $\eta_w$ is in fact a regular automorphism of $N^w$, and we denote
by $(\eta^*_w)^{-1}$ the $\C$-algebra automorphism of $\C[N^w]$, defined by
\[
((\eta^*_w)^{-1}f)(x) = f(\eta_w^{-1}(x))\qquad (f\in \C[N^w]).
\]

\begin{THM}\label{THM6}
For every $X\in\cC_w$, we have
\[
(\eta^*_w)^{-1}(\vphi_X) = \frac{\vphi_{\Omega_w(X)}}{\vphi_{P(X)}}. 
\]
Moreover, $\eta^*_w$ preserves the dual semicanonical basis $\widetilde{\cS}_w^*$ of $\C[N^w]$
and permutes its elements.
\end{THM}

Thus, the regular functions $\vphi'_{V_{\bi,k}}$ occurring in Theorem~\ref{THM5}
are obtained by twisting the generalized minors $\vphi_{V_{\bi,k}}$
with $\eta_w^{-1}$, 
in agreement with \cite{BFZ,BZ} in the Dynkin case. 
We believe that Theorem~\ref{THM6} provides a conceptual explanation
of the existence of the automorphism $\eta_w$, and of its
compatibility with total positivity \cite[Proposition~5.3]{BZ}.

\subsection{}
The article is organized as follows:
In Section~\ref{reminder} we give a short reminder
on cluster algebras and some previous results.
In Section~\ref{isosection}
we construct
isomorphisms between flag varieties and quiver Grassmannians
in a very general setup.
The isomorphisms stated in Theorem~\ref{THM1} turn out to be
special cases.
Section~\ref{ssec:CombCp} contains the
proofs of Theorems~\ref{THM1} and~\ref{THM2} 
and of the Chamber Ansatz Theorem~\ref{THM5} 
together with some illustrating examples.
The proof of the cluster character identities stated in
Theorem~\ref{THM3} and a detailed example are in
Sections~\ref{sec:PfExp} and~\ref{section5}.
The proof of
Theorem~\ref{THM4} is in Sections~\ref{section6} and~\ref{section7}.
Finally, Section~\ref{section8} contains the proof of
Theorem~\ref{THM6}.

\subsection{Notation}\label{notation}
Throughout, we work over the field  $\C$ of complex numbers.
For a $\C$-algebra $A$ let $\md(A)$ be the category of
finite-dimensional left $A$-modules.
By an $A$-\emph{module} we always mean a module in $\md(A)$,
unless stated otherwise.
Often we do not distinguish between a module and its isomorphism 
class.
Let $D := \Hom_\C(-,\C)$ be the usual duality functor. 

For a quiver $Q$ let $\rep(Q)$ be the category of finite-dimensional
representations of $Q$ over $\C$.
It is well known that
we can identify $\rep(Q)$ and $\md(\C Q)$.

By a {\it subcategory} we always mean a full subcategory.
For an $A$-module $M$ let $\add(M)$ be the subcategory of all
$A$-modules
which are isomorphic to finite direct sums of direct summands of $M$.
A subcategory $\U$ of $\md(A)$ 
is an {\it additive subcategory} if any finite direct
sum of modules in $\U$ is again in $\U$.
By $\Gen(M)$ 
(resp. $\Cogen(M)$) 
we denote the subcategory of
all $A$-modules $X$ such that there exists some $t \ge 1$ and some
epimorphism $M^t \to X$ (resp. monomorphism $X \to M^t$).

For an $A$-module $M$ let $\Sigma(M)$ be the number of isomorphism classes
of indecomposable direct summands of $M$.
An $A$-module is called 
{\it basic} 
if it can be written as a direct sum
of pairwise non-isomorphic indecomposable modules.
An $A$-module $M$ is called {\it rigid} if $\Ext_A^1(M,M) = 0$.

For an $A$-module $M$ and a simple $A$-module $S$ let
$[M:S]$ be the Jordan-H\"older multiplicity of $S$
in a composition series of $M$.
Let $\dimv(M) := \dimv_A(M) := ([M:S])_S$ be the {\it dimension vector}
of $M$, where $S$ runs through all isomorphism classes of
simple $A$-modules.

For a set $U$ we denote its cardinality by $|U|$.
If $f\df X \to Y$ and $g\df Y \to Z$ are maps, then the composition
is denoted by $gf = g \circ f\df X \to Z$.

If $U$ is a subset of a $\C$-vector space $V$, then let
$\Span_\C\ebrace{U}$ be the subspace of $V$ generated by
$U$.

Let $\N = \{0,1,2,\ldots\}$
be the natural numbers, including $0$, and let $\Z$ be the ring of integers.
For a domain $R$ let
$R(X_1,\ldots,X_r)$, $R[X_1,\ldots,X_r]$ and $R[X_1^{\pm 1},\ldots,X_r^{\pm 1}]$
be  the field of rational functions, the polynomial ring, and the
ring of Laurent polynomials
in the variables $X_1,\ldots,X_r$ with coefficients in $R$, respectively.


\section{Reminder on cluster algebras}\label{reminder}


\subsection{}
Let $\cF := \Q(X_1,\ldots,X_r)$ be the field of rational
functions in $r$ variables.
We fix a subset $F \subseteq \{ 1,\ldots,r \}$.

A {\it seed} in $\cF$ is a pair $(x,\GG)$,
where $\GG = (\GG_0,\GG_1,s,t)$ is a finite quiver without loops and without 
2-cycles with set of vertices 
$\GG_0 = \{ 1,\ldots, r\}$, and $x = (x_1,\ldots,x_r)$ with $x_1,\ldots,x_r$
algebraically independent elements in $\cF$.
The vertices in $\{ 1,\ldots,r \} \setminus F$ are called 
\emph{mutable}, and the ones in $F$ are \emph{frozen}.

Given a seed $(x,\GG)$ in $\cF$ 
and a mutable vertex $k$ of $\GG$,
we define the {\it mutation} of $(x,\GG)$
at $k$ as
$$
\mu_k(x,\GG) := (x',\GG').
$$ 
The quiver
$\GG'$ is obtained from $\GG$ by applying the
Fomin-Zelevinsky quiver mutation at $k$, which is
defined as follows:
For $1 \le i,j \le r$ let 
$$
\gamma_{ij} := |\text{number of arrows $j \to i$ in $\GG$}| -
|\text{number of arrows $i \to j$ in $\GG$}|.
$$
(Recall that there are no 2-cycles in $\GG$. So at least one of
the numbers on the right-hand side is 0.)
By definition also $\GG'$ has no loops and no 2-cycles, and the corresponding numbers $\gamma_{ij}'$ for $\GG'$ are
$$
\gamma_{ij}' := 
\begin{cases}
-\gamma_{ij} & \text{if $i=k$ or $j=k$},\\
\gamma_{ij} + 
\dfrac{|\gamma_{ik}|\gamma_{kj} + \gamma_{ik}|\gamma_{kj}|}{2} & \text{otherwise}.
\end{cases}
$$
Finally, $x' = (x_1',\ldots,x_r')$ is defined by
$$
x_s' :=
\begin{cases}
x_k^{-1}\left(\prod_{k \to i}x_i + \prod_{j \to k}x_j\right) &
\text{if $s=k$},\\
x_s & \text{otherwise}
\end{cases}
$$
where the products are taken over all arrows of $\GG$
which start, respectively end, in $k$.
Set $\mu_{(x,\GG)}(x_k) := x_k'$.
It is easy to check that $(x',\GG')$ 
is again a seed in $\cF$ and that
$\mu_k^2(x,\GG) = (x,\GG)$.

Two seeds $(x,\GG)$ and $(y,\Sigma)$ are 
{\it mutation equivalent}
if there is a sequence $(k_1,\ldots,k_t)$
with $k_i \in \{ 1,\ldots, r\} \setminus F$ for all $i$ such that
$$
\mu_{k_t} \dots \mu_{k_2}\mu_{k_1}(x,\GG) = (y,\Sigma).
$$
In this case, we write
$(y,\Sigma) \sim (x,\GG)$.

For a seed
$(x,\GG)$ in $\cF$ 
let 
$$
\cX_{(x,\GG)}
:= \bigcup_{(y,\Sigma) \sim (x,\GG)} \{ y_1,\ldots,y_r \}
$$ 
where the union is over all seeds $(y,\Sigma)$ 
with
$(y,\Sigma) \sim (x,\GG)$.
By definition, the {\it cluster algebra} $\cA(x,\GG)$ 
associated to $(x,\GG)$ is the subalgebra of $\cF$ generated
by $\cX_{(x,\GG)}$.

We call $(y,\Sigma)$ a {\it seed in} $\cA(x,\GG)$
if $(y,\Sigma) \sim (x,\GG)$.
In this case, $y$ is a
{\it cluster} in $\cA(x,\GG)$, the elements
$y_1,\ldots,y_r$ are \emph{cluster variables} and 
$y_1^{m_1} \cdots y_r^{m_r}$ with $m_i \ge 0$ for
all $i$ are \emph{cluster monomials} in $\cA(x,\GG)$.

For any seed of the form $(y,\GG)$ in $\cF$ we obtain 
an isomorphism $\cA(x,\GG) \to \cA(y,\GG)$
given by $x_i \mapsto y_i$ for all $1 \le i \le r$.
So one sometimes writes just $\cA(\GG)$ instead of
$\cA(x,\GG)$.

Note that for any cluster $y$ in $\cA(x,\GG)$ we
have $y_i = x_i$ for all $i \in F$.
These cluster variables are also called \emph{coefficients} 
of $\cA(x,\GG)$.
Localizing $\cA(x,\GG)$ at $\prod_{i \in F} x_i$ 
yields an algebra $\cA(x,\GG,F^\pm)$, which we also
call a \emph{cluster algebra}.

There are algebra
epimorphisms
$$
\cA(x,\GG) \to \cA(\underline{x},\underline{\GG})
\text{\;\;\; and \;\;\;}
\cA(x,\GG,F^\pm) \to \cA(\underline{x},\underline{\GG})
$$
defined by 
$$
x_i \to
\begin{cases}
1 & \text{ if $i \in F$},\\
x_i & \text{otherwise},
\end{cases}
$$
where
$\cA(\underline{x},\underline{\GG}) \subseteq
\Q((x_i)_{i \in \{1,\ldots,r\} \setminus F})$
is again a cluster algebra with
$\underline{x} := (x_i)_{i \in \{1,\ldots,r\} \setminus F}$,
and the quiver
$\underline{\GG}$ is obtained from $\GG$ by deleting all
vertices in $F$ and all arrows starting or ending in one of
the vertices in $F$.
We say that the cluster algebra 
$\cA(\underline{x},\underline{\GG})$ is obtained from
$\cA(x,\GG)$ by \emph{specialization of coefficients} to 1,
and the two epimorphisms defined above are called
\emph{specialization morphisms}.
Clearly, the specialization morphisms induce a surjective
map
$\cX_{(x,\GG)} \setminus \{x_i \mid i \in F\} \to \cX_{(\underline{x},\underline{\GG})}$.

Using the identification $\C[N^w] \equiv \cA(\GG_T)$, the
epimorphism $\Pi_T$ defined in Section~\ref{section1.5},
can be seen as a specialization morphism.
Thus the cluster algebra $\cA(\sGG_T)$ is obtained from
$\C[N^w]$ by specialization of coefficients to 1.

\subsection{Cluster algebra structures for coordinate rings of unipotent cells}

In a series of papers \cite{GLSUni0,GLSRigid,GLSUni2} we constructed a map
$$
\vph\df \nil(\LL) \to \C[N]
$$
which maps a nilpotent $\LL$-module $X$
to a function $\vph_X \in \C[N]$. 
This map satisfies the following properties:
\begin{itemize}

\item[(i)]
For all $X,Y \in \nil(\LL)$ 
we have 
$$
\vph_X\vph_Y = \vph_{X \oplus Y}.
$$

\item[(ii)]
Let $X,Y \in \nil(\LL)$ with $\dim \Ext_\LL^1(X,Y) = 
\dim \Ext_\LL^1(Y,X) = 1$, and let
$$
0 \to X \to E' \to Y \to 0
\text{\;\;\; and \;\;\;}
0 \to Y \to E'' \to X \to 0
$$
be non-split short exact sequences.
Then we have
$$
\vph_X\vph_Y = \vph_{E'} + \vph_{E''}.
$$

\item[(iii)]
Restriction yields a map
$$
\vph\df \cC_w \to \C[N^w].
$$
(Again we identified $\C[N^w]$ with the localization of the
$\C$-span of $\{\vphi_X \mid X\in\cC_w\}$
at
$\{\vphi_P \mid P \mbox{ is $\cC_w$-projective-injective} \}$.)

\item[(iv)]
Let $\ii = (i_r,\ldots,i_1)$ be a reduced expression
of $w$,
and let $\GG := \GG_\ii$ and $F := \Rma$.
(The definitions of $\GG_\ii$ and $\Rma$ can be found
in
Sections~\ref{ssec:QuivEndV} and~\ref{section1.4}, respectively.)
Then there is an algebra isomorphism
$$
\eta_\ii\df \cA(x,\GG,F^\pm) \to \C[N^w]
$$
with
$\eta_\ii(x_k) = \vph_{V_{\ii,k}}$ for all
$1 \le k \le r$.

\end{itemize}
Using the isomorphism $\eta_\ii$ one can now speak
of \emph{cluster variables} and \emph{cluster monomials} 
in $\C[N^w]$.
For example, an $r$-tuple $(\vph_{T_1},\ldots,\vph_{T_r})$
is a cluster in $\C[N^w]$ if and only if there is a seed
$(y,\Sigma)$ in $\cA(x,\GG,F^\pm)$ with
$\eta_\ii(y_i) = \vph_{T_i}$ for all $i$.
In this case, let $T := T_1 \oplus \cdots \oplus T_r$.
The vertices of the quiver $\GG_T$ of the endomorphism
algebra $\End_\LL(T)^\op$ are naturally parametrized by $1,\ldots,r$ 
and the following hold:

\begin{itemize}

\item[(v)]
With the exception of arrows between 
coefficients
$c,d \in F$, the quivers
$\Sigma$ and $\GG_T$ coincide.
The seed $(y,\Sigma)$ in $\cA(x,\GG)$
is already determined by $y$.

\item[(vi)]
The module $T$ is a basic cluster-tilting module in $\cC_w$.
For any mutable vertex $k$
there is a unique indecomposable $T_k' \in \cC_w$ with
$T_k' \not\cong T_k$ such that
$$
\mu_k(T) := T_k' \oplus T/T_k
$$ 
is a basic cluster-tilting
module in $\cC_w$. 
For
$y_k' := \mu_{(y,\Sigma)}(y_k)$ we have
$$
\eta_\ii(y_k') = \vph_{T_k'}.
$$
We say that
$(\vph_{T_1},\ldots,\vph_{T_k'},\ldots,\vph_{T_r})$ is
obtained from
$(\vph_{T_1},\ldots,\vph_{T_k},\ldots,\vph_{T_r})$ by
\emph{mutation in direction} $k$.
We also say that $\mu_k(T)$ is obtained from $T$ by
\emph{mutation in direction} $k$.

\item[(vii)]
We have $\dim \Ext_\LL^1(T_k,T_k') = 
\dim \Ext_\LL^1(T_k',T_k) = 1$, and there are
short exact sequences
$$
0 \to T_k \to \bigoplus_{j \to k}T_j \to T_k' \to 0
\text{\;\;\; and \;\;\;}
0 \to T_k' \to \bigoplus_{k \to i}T_i \to T_k \to 0,
$$
where we sum over all arrows in $\GG_T$ ending
and starting in $k$, respectively.
Furthermore, the 
identity
$$
y_ky_k' = \prod_{k \to i}y_i + \prod_{j \to k}y_j
$$
in $\cA(x,\GG)$ corresponds to the identity 
$$
\vph_{T_k}\vph_{T_k'} = 
\prod_{k \to i}\vph_{T_i} + \prod_{j \to k}\vph_{T_j}
$$
in $\C[N^w]$.
For $i,j \in R$, the number of
arrows $k \to i$ in $\GG_T$ equals $[B_{i,k}^{(T)}]_+$
and the number of arrows $j \to k$  
is $[-B_{j,k}^{(T)}]_+$.

\item[(viii)]
The cluster monomials in $\C[N^w]$ are 
$$
\vph_{T_1}^{m_1} \cdots \vph_{T_r}^{m_r}
$$
where $m_i \ge 0$ for all $i$, and 
$T := T_1 \oplus \cdots \oplus T_r$ runs
through the set of $V_\ii$-reachable cluster-tilting
modules
in $\cC_w$.

\item[(ix)]
All cluster monomials in $\C[N^w]$ belong to the
dual semicanonical basis of $\C[N^w]$.

\end{itemize}


\section{Partial flag varieties and quiver Grassmannians}
\label{isosection}


\subsection{Basic algebras and nilpotent modules}\label{algebraclass}
Let $\GG = (\GG_0,\GG_1,s,t)$ 
be a finite quiver with set of vertices $\GG_0 = \{ 1,\ldots,n \}$,
and set of arrows $\GG_1$.
For an arrow $a\df i \to j$ in $\GG$ let $s(a) := i$ and $t(a) := j$
be its start vertex and terminal vertex, respectively. 

A \emph{path} of \emph{length} $m$ in $\GG$ is an 
$m$-tuple $p = (a_1,\ldots,a_m)$ of arrows in $\GG$
such that $s(a_i) = t(a_{i+1})$ for all $1 \le i \le m-1$.
We define $s(p) := s(a_m)$ and $t(p) := t(a_1)$.
Additionally, for each vertex $i \in \GG_0$ there is a path
$e_i$ of length 0 with $s(e_i) = t(e_i) = i$.
An arrow $a$ in $\GG$ is a \emph{loop} if $s(a) = t(a)$.
A path $p = (a_1,a_2)$ is a 2-\emph{cycle} if $s(p) = t(p)$.

The \emph{path algebra} $\C\GG$ of $\GG$ 
has the paths in $\GG$ as a $\C$-basis, and the multiplication
of two paths $p$ and $q$ 
is defined by
$$
pq :=
\begin{cases}
(a_1,\ldots,a_m,b_1,\ldots,b_l) & \text{if
$s(p) = t(q)$, $p = (a_1,\ldots,a_m)$ and
$q = (b_1,\ldots,b_l)$},\\
p & \text{if $q = e_{s(p)}$},\\
q & \text{if $p = e_{t(q)}$},\\
0 & \text{if $s(p) \not= t(q)$}.
\end{cases}
$$
Extending this rule linearly turns $\C \GG$ into an associative
$\C$-algebra with unit element.

For $m \ge 0$ let $\C\GG_{\ge m}$ be the ideal in $\C\GG$ generated by all
paths of length $m$.
An algebra $A$ is called \emph{basic} if 
$A = \C\GG/J$, where $J$ is an ideal in $\C\GG$ with $J \subseteq \C\GG_{\ge 2}$.
For the rest of this section, we assume that $A = \C\GG/J$ is a basic 
algebra.

Let $S_1,\ldots,S_n$ be the 1-dimensional $A$-modules associated to
the vertices of $\GG$.
(If $A$ is finite-dimensional, then $S_1,\ldots,S_n$ are all simple $A$-modules
up to isomorphism.)
We focus on $A$-modules having only $S_1,\ldots,S_n$ as composition factors.
These modules are called \emph{nilpotent}.
The category of all nilpotent $A$-modules is denoted by $\nil(A)$.
(If $A$ is finite-dimensional, then $\nil(\LL) = \md(A)$.)
Let $\widehat{I}_1,\ldots,\widehat{I}_n$ be the injective envelopes of $S_1,\ldots,S_n$,
respectively.
(The modules $\widehat{I}_j$ 
are in general infinite-dimensional $A$-modules.)

Let $J_i$ be the maximal ideal of $A$ 
spanned by all residue classes $\ov{p} := p+J$ of paths, 
where $p$ runs through all
paths except $e_i$.
Thus $A/J_i$ is 1-dimensional and
(as an $A$-module) isomorphic to $S_i$.
(In the following, we sometimes do not distinguish between a 
path $p$ in $\C\GG$ and 
its residue class $\ov{p}$.)

Each (not necessarily finite-dimensional) $A$-module $X$ 
can be interpreted as a representation
$X = (X(i),X(a))_{i \in \GG_0,a \in \GG_1}$ of the quiver $\GG$,
where
the vector space $X(i)$ is defined by $e_iX$, and the linear
map $X(a)\df X(s(a)) \to X(t(a))$ is defined by $x \mapsto ax$.
Recall that a \emph{subrepresentation} of $X$ is given by
$U = (U(i))_{i \in \GG_0}$, where $U(i)$ is a subspace of $X(i)$
for all $i$, and for all $a \in \GG_1$ we have 
$X(a)(U(s(a))) \subseteq U(t(a))$.
When passing from modules to representations, the submodules
obviously correspond to the subrepresentations.
The \emph{dimension vector} of a representation $X = (X(i),X(a))_{i \in \GG_0,a \in \GG_1}$
is by definition $\dimv_\GG(X) := (\dim X(i))_{i \in Q_0}$.

\begin{Def}
{\rm
For a dimension vector $\bd$, let $\Gr_\bd^A(X)$
be the projective variety of subrepresentations $Y$ of $X$ with
$\dimv_\GG(Y) = \bd$.
Such a variety is called a \emph{quiver Grassmannian}.
}
\end{Def}

If $X$ is nilpotent, then $\dim X(i) = [M:S_i]$ for all $i \in Q_0$.
We study Grassmannians $\Gr_\bd^A(X)$ only for nilpotent $A$-modules
$X$, so there is no danger of confusing the two types of dimension vectors 
$\dimv_\GG(-)$ and $\dimv_A(-)$ associated to $X$ and its submodules.

\subsection{Refined socle and top series}
For an arbitrary (not necessarily finite-dimensional)
$A$-module $X$ and a simple $A$-module $S$, let
$\soc_S(X)$ be the sum of all submodules $U$ of $X$ with
$U \cong S$.
(If there is no such $U$, then $\soc_S(X) = 0$.)
Similarly, let $\tp_S(X) = X/V$, where $V$ is the intersection of
all submodules $U$ of $X$ such that $X/U \cong S$.
(If there is no such $U$, then $V = X$ and $\tp_S(X) = 0$.)
Define $\rad_S(X) := V$.

Let us interpret $X$ as a representation 
$X = (X(i),X(a))_{i \in \GG_0,a \in \GG_1}$ of $\GG$, and 
let $1 \le j \le n$.
Then $\soc_{S_j}(X)$ can be seen as a subrepresentation
$(X'(i))_{i \in \GG_0}$ of $X$, where 
$$
X'(i) =
\begin{cases}
0 & \text{if $i \not= j$},\\
\bigcap_{a \in \GG_1,s(a)=j} \Ker(X(a)) & \text{if $i=j$}.
\end{cases}
$$
Similarly,
$\rad_{S_j}(X)$ can be seen as a subrepresentation
$(X'(i))_{i \in \GG_0}$ of $X$, where 
$$
X'(i) =
\begin{cases}
X(i) & \text{if $i \not= j$},\\
\sum_{a \in \GG_1,t(a)=j} \Ima(X(a)) & \text{if $i=j$}.
\end{cases}
$$
It follows that $\soc_{S_j}(X)$ and $\tp_{S_j}(X)$ are isomorphic
to (possibly infinite) direct sums of copies of $S_j$.

Now fix some sequence $\ii = (i_r,\ldots,i_1)$ with $1 \le i_k \le n$ 
for all $k$.
There exists a unique chain
$$
(0 = X_r \subseteq \cdots \subseteq X_1 \subseteq X_0 \subseteq X)
$$
of submodules $X_k$ of $X$ such that
$X_{k-1}/X_k = \soc_{S_{i_k}}(X/X_k)$ for all $1 \le k \le r$.
We define $\soc_\ii(X) := X_0$,
$$
X_k^+ := X_k^{\ii,+} := X_k
$$ 
for all $0 \le k \le r$, and
$X_\bullet^+ := X_\bullet^{\ii,+} :=
(X_r^+ \subseteq \cdots \subseteq X_1^+ \subseteq X_0^+)$.
If $\soc_\ii(X) = X$, then we call this chain the 
\emph{refined socle series of type} $\ii$ of $X$.
Similarly, 
there exists a unique chain
$$
(0 \subseteq X_r \subseteq \cdots \subseteq X_1 \subseteq X_0 = X)
$$
of submodules $X_k$ of $X$ 
such that $X_{k-1}/X_k = \tp_{S_{i_k}}(X_{k-1})$ for all
$1 \le k \le r$.
Set $\tp_\ii(X) := X/X_r$, $\rad_\ii(X) := X_r$, and 
$$
X_k^- := X_k^{\ii,-} := X_k
$$ 
for all $0 \le k \le r$.
Define
$X_\bullet^- := X_\bullet^{\ii,-} :=
(X_r^- \subseteq \cdots \subseteq X_1^- \subseteq X_0^-)$.
If $\rad_\ii(X) = 0$, then 
$X_\bullet^-$ is called the \emph{refined top series of type} $\ii$ of $X$.

The following lemma is straightforward:

\begin{Lem}\label{homsoc}
For arbitrary (not necessarily finite-dimensional) $A$-modules
$X$ and $Y$ and every $A$-module homomorphism 
$f\df X \to Y$ the following
hold:
\begin{itemize}

\item[(i)]
$f(\soc_\ii(X)) \subseteq \soc_\ii(Y)$ and
$f(\rad_\ii(X)) \subseteq \rad_\ii(Y)$.

\item[(ii)]
If $f$ is a monomorphism (resp. epimorphism), then
the induced maps
$$
X/\soc_\ii(X) \to Y/\soc_\ii(Y)
\text{\;\;\; and \;\;\;}
\rad_\ii(X) \to \rad_\ii(Y)
$$
are both monomorphisms (resp. epimorphisms).

\item[(iii)]
If $\soc_\ii(Y) = Y$, then $f(\rad_\ii(X)) = 0$.

\end{itemize}
\end{Lem}

For $1 \le k,s \le r$ define
$$
J_{k,s} := 
\begin{cases}
J_{i_k}J_{i_{k-1}} \cdots J_{i_s} & \text{if $k \ge s$},\\
A & \text{otherwise}.
\end{cases}
$$
Also the next lemma is easy to show:

\begin{Lem}\label{Jrad}
For an arbitrary (not necessarily finite-dimensional) $A$-module
$X$ and $1 \le k \le r$ we have
$J_{k,1}X = X_k^- = \rad_{(i_k,\ldots,i_1)}(X)$.
\end{Lem}

\begin{Cor}\label{Jrad2}
The algebra $A/J_{k,1}$ is finite-dimensional for all $1 \le k \le r$.
\end{Cor}

\begin{proof}
Use Lemma~\ref{Jrad} and the fact that the quiver $\GG$ of $A$ is
finite.
\end{proof}

Let $\cD_\ii$ be the category of all $A$-modules
$X$ in $\md(A)$ such that $\soc_\ii(X) = X$.

\begin{Lem}
For an $A$-module $X$ the following are equivalent:
\begin{itemize}

\item[(i)]
$X \in \cD_\ii$.

\item[(ii)]
$\soc_\ii(X) = X$.

\item[(iii)]
$\rad_\ii(X) = 0$.

\end{itemize}
\end{Lem}

\begin{proof}
By definition, (i) and (ii) are equivalent.
The equivalence of (ii) and (iii) follows
by an obvious induction on the length $r$ of the sequence $\ii$.
\end{proof}

Let $A_\ii := A/J_{r,1}$.
We identify the category $\md(A_\ii)$ of finite-dimensional $A_\ii$-modules
with the category of all $X$ in $\nil(A)$ such that $J_{r,1}X = 0$.
Under this identification we obviously get the following:

\begin{Lem}\label{D1}
We have
$\cD_\ii = \md(A_\ii)$.
\end{Lem}

\subsection{Partial composition series}

\begin{Def}
{\rm 
For $X \in \cD_\ii$ and $\ba = (a_r,\ldots,a_1)$ with $a_j \ge 0$ let
$\F_{\ii,\ba,X}$ be the (possibly empty) set of chains
$
X_\bullet = 
(0 = X_r \subseteq \cdots \subseteq X_1 \subseteq X_0 = X)
$
of submodules $X_k$ of $X$ such that
$X_{k-1}/X_k \cong S_{i_k}^{a_k}$ for all $1 \le k \le r$.
We call $(0 = X_r \subseteq \cdots \subseteq X_1 \subseteq X_0 = X)$
a \emph{partial composition series} of \emph{type} $\ii$ of $X$.
}
\end{Def}

Clearly, $\F_{\ii,\ba,X}$ is a projective variety.
The \emph{weight} of $X_\bullet \in \F_{\ii,\ba,X}$ is defined by
$$
\wt(X_\bullet) := (a_r,\ldots,a_2,a_1).
$$
If $X_\bullet = X_\bullet^-$ (resp. $X = X_\bullet^+$), we define
$\ba^-(X) := \wt(X_\bullet^-)$ and
$\ba_k^-(X) := a_k$ (resp. $\ba^+(X) := \wt(X_\bullet^+)$ and  
$\ba_k^+(X) := a_k$) for all $1 \le k \le r$.

\begin{Lem}\label{lemma3}
For $X \in \cD_\ii$ and
$(X_r \subseteq \cdots \subseteq X_1 \subseteq X_0) \in \F_{\ii,\ba,X}$
we have 
$$
X_k^- \subseteq X_k \subseteq X_k^+
$$
for all $1 \le k \le r$.
\end{Lem}

\begin{proof}
For $1 \le k \le r$ we show that
$X_k \subseteq X_k^+$ by decreasing induction
on $k$.
Clearly, we have $X_r \subseteq X_r^+$. 
(By definition, $X_r = X_r^+ = 0$.)
Next, assume that $X_s \subseteq X_s^+$ for some $1 \le s \le r$.
Thus, there is an epimorphism $\pi\df X/X_s \to X/X_s^+$.
We have $X_{s-1}/X_s \subseteq \soc_{S_{i_s}}(X/X_s)$, and
by definition $X_{s-1}^+/X_s^+ = \soc_{S_{i_s}}(X/X_s^+)$.
This implies that $\pi(X_{s-1}/X_s) \subseteq X_{s-1}^+/X_s^+$.
In other words, $x + X_s^+ \in X_{s-1}^+/X_s^+$ for all
$x \in X_{s-1}$.
For
each such $x$ there exists some $y \in X_{s-1}^+$ with
$x + X_s^+ = y + X_s^+$.
This implies that $x-y$ is in $X_s^+$.
Since $X_s^+ \subseteq X_{s-1}^+$ we get $x \in X_{s-1}^+$.
Thus we have proved that $X_k \subseteq X_k^+$ for all $1 \le k \le r$.
Similarly, one shows by induction on $k$ that 
$X_k^- \subseteq X_k$ for all $1 \le k \le r$.
\end{proof}

The next lemma follows from the uniqueness of refined
socle and top series.

\begin{Lem}\label{lemma3.1}
Let $X \in \cD_\ii$.
If $\ba$ is equal to $\wt(X_\bullet^-)$ or
$\wt(X_\bullet^+)$, then
$\chi\left(\F_{\ii,\ba,X}\right) = 1$.
\end{Lem}

\begin{Cor}
For every $X \in \cD_\ii$ there exists some $\ba$ such that
$\chi\left(\F_{\ii,\ba,X}\right) \not= 0$.
\end{Cor}

\subsection{The modules $V_\ii$} \label{Vii}
Let $A = \C\GG/J$ be a basic algebra,
and let $\ii = (i_r,\ldots,i_1)$ with $1 \le i_k \le n$ for all $k$.
Without loss of generality we assume that for each $1 \le j \le n$
there exists some $k$ with $i_k = j$.
For $1 \le k \le r$ and $1 \le j \le n$ let
\begin{align*}
k^- &:= \max\{ 0,1 \le s \le k-1 \mid i_s = i_k \},\\
k^+ &:= \min\{ k+1 \le s \le r,r+1 \mid i_s = i_k \},\\
k_{\rm max} &:= \max\{ 1 \le s \le r \mid i_s = i_k \},\\
k_{\rm min} &:= \min\{ 1 \le s \le r \mid i_s = i_k \},\\
k_j &:= \max\{ 1 \le s \le r \mid i_s = j \}.
\end{align*}
For $1 \le k \le r$ define
$$
V_k := V_{\ii,k} := \soc_{(i_k,\ldots,i_1)}(\widehat{I}_{i_k})
$$ 
and
$V_\ii := V_1 \oplus \cdots \oplus V_r$.
We also set $V_0 := 0$.
For every $1 \le j \le n$ let 
$I_{\ii,j} := V_{k_j}$ and
$I_\ii := I_{\ii,1} \oplus \cdots \oplus I_{\ii,n}$.
The modules in $\add(I_\ii)$ are called $\ii$-\emph{injective}.

\begin{Lem}\label{lemma4}
For $1 \le k \le r$ we have
$V_k \cong D(e_{i_k}(A/J_{k,1}))$.
In particular, $V_k$ is an indecomposable
injective $A/J_{k,1}$-module. 
\end{Lem}

\begin{proof}
Clearly, $J_{k,1}V_k = 0$.
Thus $V_k$ is an $A/J_{k,1}$-module.
We have $\soc_{S_{i_k}}(V_k) \cong S_{i_k}$.
Thus $V_k$ can be embedded into the indecomposable injective 
$A/J_{k,1}$-module $D(e_{i_k}(A/J_{k,1}))$.
We have
$\soc_{S_{i_k}}(D(e_{i_k}(A/J_{k,1}))) \cong S_{i_k}$.
Therefore $D(e_{i_k}(A/J_{k,1}))$ can be embedded into $\widehat{I}_{i_k}$.
Thus we get two monomorphisms
$$
V_k \xrightarrow{\iota_1} D(e_{i_k}(A/J_{k,1})) 
\xrightarrow{\iota_2} \widehat{I}_{i_k}.
$$
Since 
$\soc_{(i_k,\ldots,i_1)}(D(e_{i_k}(A/J_{k,1}))) = D(e_{i_k}(A/J_{k,1}))$,
we can apply Lemma~\ref{homsoc}(i) and get
$
\iota_2(D(e_{i_k}(A/J_{k,1}))) \subseteq 
\soc_{(i_k,\ldots,i_1)}(\widehat{I}_{i_k})$.
Since  $D(e_{i_k}(A/J_{k,1}))$ is finite-dimensional by 
Corollary~\ref{Jrad2}, this implies that
$V_k \cong D(e_{i_k}(A/J_{k,1}))$.
\end{proof}

\begin{Cor}
$V_\ii \in \cD_\ii$.
\end{Cor}

\begin{proof}
We have $\soc_\ii(V_\ii) = V_\ii$, and $V_\ii$ is finite-dimensional
by Corollary~\ref{Jrad2} and Lemma~\ref{lemma4}.
\end{proof}

\begin{Lem}\label{D2}
An $A_\ii$-module $X$ is injective if and only if
$X \in \add(I_\ii)$.
\end{Lem}

\begin{proof}
One easily checks that $e_jJ_{r,1} = e_jJ_{k_j,1}$.
This implies
$D(e_j(A/J_{r,1})) = D(e_j(A/J_{k_j,1}))$.
But $D(e_j(A/J_{k_j,1})) = I_{\ii,j}$ by Lemma~\ref{lemma4}.
Thus the modules in $\add(I_\ii)$ are the injective
$A_\ii$-modules.
\end{proof}

\begin{Lem}\label{lemma5}
For every $1 \le k \le r$ there is a 
monomorphism $V_{k^-} \to V_k$.
\end{Lem}

\begin{proof}
We have $J_{k,1} \subseteq J_{k^-,1}$.
Thus there is a short exact sequence
$$
0 \to J_{k^-,1}/J_{k,1} \to A/J_{k,1} \to A/J_{k^-,1} \to 0. 
$$
Applying $e_{i_k} \cdot$ and then the duality $D$ yields
a short exact sequence
$$
0 \to D(e_{i_k}(A/J_{k^-,1})) \to D(e_{i_k}(A/J_{k,1})) \to 
D(e_{i_k}(J_{k^-,1}/J_{k,1})) \to 0.
$$
Now the result follows from Lemma~\ref{lemma4}.
\end{proof}

The following lemma is well known and easy to prove:

\begin{Lem}\label{lemma8}
For any $A$-module $X$ and any idempotent $e$ in $A$ the following hold:
\begin{itemize}

\item[(i)]
There is an isomorphism
of $(eAe)^\op$-modules
$$
D(eX) \cong \Hom_A(X,D(eA))
$$
defined by $\eta \mapsto f_\eta := 
\left[ x \mapsto (ea \mapsto \eta(eax)) \right]$.

\item[(ii)]
Assume that $X$ is finite-dimensional.
Then
there is an isomorphism
of $eAe$-modules
$$
eX \cong D{\Hom}_A(X,D(eA))
$$
defined by $ex \mapsto 
\left[ f \mapsto f(x)(e) \right]$.

\end{itemize}
\end{Lem}

The vector space $D{\Hom}_A(X,D(eA))$ is an $\End_A(D(eA))^\op$-module
in an obvious way, and we have
$eAe \cong \End_A(D(eA))^\op$.
Under the isomorphisms 
$eX \cong D{\Hom}_A(X,D(eA))$ and $eAe \cong \End_A(D(eA))^\op$,
the action of $\End_A(D(eA))^\op$ on $D{\Hom}_A(X,D(eA))$
turns into the action $eae \cdot ex := eaex$ of 
$eAe$ on $eX$.

\begin{Lem}\label{lemma7}
For any $A$-module $X$ we have
$$
\Hom_A(X,V_k) = \Hom_A(X/X_k^-,V_k) = 
\Hom_{A/J_{k,1}}(X/X_k^-,V_k).
$$
\end{Lem}

\begin{proof}
We have $\soc_{(i_k,\ldots,i_1)}(V_k) = V_k$, and
$\rad_{(i_k,\ldots,i_1)}(X) = X_k^-$.
By Lemma~\ref{homsoc}(iii) this implies 
$f(X_k^-) = 0$ for every $f \in \Hom_A(X,V_k)$.
This yields the identification
$\Hom_A(X,V_k) = \Hom_A(X/X_k^-,V_k)$.
Now $X/X_k^-$ and $V_k$ are annihilated by $J_{k,1}$.
Thus $X/X_k^-$ and $V_k$ are $A/J_{k,1}$-modules.
This implies
$\Hom_A(X/X_k^-,V_k) = \Hom_{A/J_{k,1}}(X/X_k^-,V_k)$.
\end{proof}

\begin{Cor}\label{cor8.2}
For any finite-dimensional $A$-module $X$ we have
$$
D{\Hom}_A(X,V_k) \cong e_{i_k}(X/X_k^-).
$$
\end{Cor}

\begin{proof}
The $A$-modules $X/X_k^-$ and $V_k$ can be regarded as an $A/J_{k,1}$-module,
since both are annihilated by $J_{k,1}$, and
$V_k$ is injective as an $A/J_{k,1}$-module.
Now we apply Lemma~\ref{lemma8}.
\end{proof}

\subsection{Balanced modules}
An $A$-module $X$ is called $\ii$-\emph{balanced} 
if $X \in \cD_\ii$ and $X_\bullet^- = X_\bullet^+$.
Thus, $X$ is $\ii$-balanced if and only if $X_k^- = X_k^+$ for all
$0 \le k \le r$.

\begin{Prop}
Let $X \in \cD_\ii$.
Then the following are equivalent:
\begin{itemize}

\item[(i)]
$X$ is $\ii$-balanced.

\item[(ii)]
There is a unique $\bb$ such that
$\F_{\ii,\bb,X} \not= \varnothing$.

\item[(iii)]
There is a unique $\bb$ such that
$\chi\left(\F_{\ii,\bb,X}\right) \not= 0$.

\end{itemize}
\end{Prop}

\begin{proof}
\parindent0mm
${\rm (i)} \implies {\rm (ii)}$:
Since $X \in \cD_\ii$, we know that $\soc_\ii(X) = X$ and
$\rad_\ii(X) = 0$.
This implies that
$\F_{\ii,\wt(X_\bullet^-),X}$ and $\F_{\ii,\wt(X_\bullet^+),X}$
are both non-empty.
Set $\bb := \wt(X_\bullet^+)$.
Since $X$ is $\ii$-balanced, we have $X_k^- = X_k^+$ for all $k$.
In other words, $\bb = \wt(X_\bullet^-) = \wt(X_\bullet^+)$.
The uniqueness of $\bb$ follows now from Lemma~\ref{lemma3}.

\parindent0mm
${\rm (ii)} \implies {\rm (iii)}$:
This follows directly from Lemma~\ref{lemma3.1}.

\parindent0mm
${\rm (iii)} \implies {\rm (i)}$:
Since $X \in \cD_\ii$, Lemma~\ref{lemma3.1} implies 
$\chi(\F_{\ii,\wt(X_\bullet^-),X}) =
\chi(\F_{\ii,\wt(X_\bullet^+),X}) = 1$.
Since we assume $\bb$ to be unique, we get 
$\wt(X_\bullet^-) = \wt(X_\bullet^+)$.
Now (i) follows from Lemma~\ref{lemma3}.
\end{proof}

\begin{Lem}\label{st1}
Let $X$ and $Y$ be $A$-modules.
Then the following hold:
\begin{itemize}

\item[(i)]
If $X$ and $Y$ are $\ii$-balanced, then $X \oplus Y$ is $\ii$-balanced. 

\item[(ii)]
If $X$ is $\ii$-balanced, then each direct summand of $X$ is $\ii$-balanced.

\end{itemize}
\end{Lem}

\begin{proof}
One easily checks that
for every direct sum decomposition $M = M_1 \oplus M_2$ of an $A$-module
$M$ and every sequence $\jj = (j_t,\ldots,j_1)$ with $1 \le j_s \le n$
for all $s$, we have
$\soc_\jj(M) = \soc_\jj(M_1) \oplus \soc_\jj(M_2)$ and
$\rad_\jj(M) = \rad_\jj(M_1) \oplus \rad_\jj(M_2)$.
This implies both (i) and (ii).
\end{proof}

We say that the pair $(A,\ii)$ is \emph{balanced}, if for each
$1 \le k \le r$ the $A$-module 
$V_k = V_{\ii,k}$ 
is $(i_k,\ldots,i_1)$-balanced.
The following lemma follows directly from the definitions:

\begin{Lem}\label{st2}
Assume that $(A,\ii)$ is balanced.
For $1 \le k \le r$ and $0 \le s < k$ we have
$$
\rad_{(i_s,\ldots,i_1)}(V_{\ii,k}) =
(V_{\ii,k})_s^{\ii,-} =
(V_{\ii,k})_s^{(i_k,\ldots,i_1),-} =
(V_{\ii,k})_s^{(i_k,\ldots,i_1),+}
= \soc_{(i_k,\ldots,i_{s+1})}(V_{\ii,k}).
$$
\end{Lem}

\begin{Lem}\label{st3}
Assume that $(A,\ii)$ is balanced.
Then the modules $V_{\ii,1},\ldots,V_{\ii,r}$ are pairwise non-isomorphic.
\end{Lem}

\begin{proof}
Assume $V_{\ii,k} \cong V_{\ii,s}$ with $k > s$.
By definition $V_{\ii,k} = \soc_{(i_k,\ldots,i_s,\ldots,i_1)}(\widehat{I}_{i_k})$
and
$V_{\ii,s} = \soc_{(i_s,\ldots,i_1)}(\widehat{I}_{i_s})$.
Clearly, $\rad_{(i_s,\ldots,i_1)}(V_{\ii,s}) = 0$.
Since $V_{\ii,k} \cong V_{\ii,s}$ we also get
$\rad_{(i_s,\ldots,i_1)}(V_{\ii,k}) = 0$.
But $V_{\ii,k}$ is $(i_k,\ldots,i_1)$-balanced.
By Lemma~\ref{st2} this implies
$\soc_{(i_k,\ldots,i_{s+1})}(V_{\ii,k}) = 
\rad_{(i_s,\ldots,i_1)}(V_{\ii,k}) = 0$.
But we have $\soc_{S_{i_k}}(V_{\ii,k}) \cong S_{i_k}$.
This implies
$\soc_{(i_k,\ldots,i_{s+1})}(V_{\ii,k}) \not= 0$, a contradiction.
\end{proof}

\begin{Prop}\label{st4}
Assume that $(A,\ii)$ is balanced.
For $1 \le k,s \le r$ we have
$$
\Hom_A(V_k,V_s) \cong e_{i_k}\left(J_{k,s+1}/J_{k,1}\right)e_{i_s}.
$$
\end{Prop}

\begin{proof}
Recall that
$V_k = D(e_{i_k}(A/J_{k,1}))$ and $V_s = D(e_{i_s}(A/J_{s,1}))$.
By Lemma~\ref{lemma7} we have
$\Hom_A(V_k,V_s) = \Hom_A(V_k/(V_k)_s^-,V_s)$.
We have
$$
(V_k)_s^- = J_{s,1}V_k = D(e_{i_k}(A/J_{k,s+1})).
$$
For the second equality we used that $V_k$ is $(i_k,\ldots,i_1)$-balanced.
Note that $(V_k)_s^- = 0$ if $k \le s$.
We get
$$
\Hom_A(V_k/(V_k)_s^-,V_s) \cong D(e_{i_s}(V_k/(V_k)_s^-)) =
D\left(e_{i_s}(D(e_{i_k}(A/J_{k,1}))/D(e_{i_k}(A/J_{k,s+1})))\right).
$$
For the first isomorphism we used Lemma~\ref{lemma8}.
Now we first apply $e_{i_k} \cdot$ and then the duality $D$ to
the short exact sequence
$$
0 \to J_{k,s+1}/J_{k,1} \to A/J_{k,1} \to A/J_{k,s+1} \to 0,
$$
and we obtain
$$
D(e_{i_k}(A/J_{k,1}))/D(e_{i_k}(A/J_{k,s+1})) \cong 
D(e_{i_k}(J_{k,s+1}/J_{k,1})).
$$
Now 
$D(e_{i_s}D(e_{i_k}(J_{k,s+1}/J_{k,1}))) =
D(D(e_{i_k}(J_{k,s+1}/J_{k,1})e_{i_s})) \cong
e_{i_k}(J_{k,s+1}/J_{k,1})e_{i_s}$ 
implies
$\Hom_A(V_k,V_s) \cong e_{i_k}(J_{k,s+1}/J_{k,1})e_{i_s}$.
\end{proof}

Using Lemma~\ref{lemma8},
the isomorphism 
$e_{i_k}J_{k,s+1}/J_{k,1}e_{i_s} \to
\Hom_A(V_k,V_s)$
can be described more precisely:
Let $e_{i_k}\ov{b}e_{i_s} \in e_{i_k}(J_{k,s+1}/J_{k,1})e_{i_s}$.
Then $e_{i_k}\ov{b}e_{i_s}$ is mapped to the homomorphism
$V_k \to V_s$,
which maps a linear form 
$\eta\df e_{i_k}(A/J_{k,1}) \to \C$ in $D(e_{i_k}(A/J_{k,1}))$
to the linear form
$\psi \in D(e_{i_s}(A/J_{s,1}))$ defined by
$$
\psi(e_{i_s}\ov{a}) := \eta(e_{i_k}\ov{b}e_{i_s}\ov{a}).
$$

For $A$-modules $X$ and $Y$
let $\cI_\ii(X,Y)$ be the subspace of $\Hom_A(X,Y)$ consisting
of the morphisms factoring through a module in $\add(I_\ii)$.
Define 
$$
\ov{\Hom}_A(X,Y) := \Hom_A(X,Y)/\cI_\ii(X,Y).
$$

\begin{Lem}\label{st6}
Assume that $(A,\ii)$ is balanced.
Then for each $X \in \cD_\ii$ and $1 \le k \le r$ 
we have 
$$
\cI_\ii(X,V_k) = \Hom_A(X/X_k^+,V_k).
$$
\end{Lem}

\begin{proof}
There is a short exact sequence
$$
0 \to X_k^+ \to X \to X/X_k^+ \to 0.
$$
Applying the functor $\Hom_A(-,V_k)$ we
can identify $\Hom_A(X/X_k^+,V_k)$ with a subspace of $\Hom_A(X,V_k)$.
Suppose that $f\df X \to V_k$ is a homomorphism.

Assume first that $f = h \circ g$ with $g\df X \to I$ and 
$I \in \add(I_\ii)$.
It follows from Lemma~\ref{D1} and Lemma~\ref{D2}
that we can assume without loss of generality that
$g$ is a monomorphism.
By Lemma~\ref{homsoc}(i) we know that $g(X_k^+) \subseteq I_k^+$.
By definition $I_k^- = \rad_{(i_k,\ldots,i_1)}(I)$ and 
$\soc_{(i_k,\ldots,i_1)}(V_k) = V_k$.
Thus Lemma~\ref{homsoc}(iii) implies $h(I_k^-) = 0$.
Since $(A,\ii)$ is balanced, we get $I_k^- = I_k^+$.
This shows that $f(X_k^+) = 0$.
In other words, $f \in \Hom_A(X/X_k^+,V_k)$.
So we proved that $\cI_\ii(X,V_k) \subseteq  \Hom_A(X/X_k^+,V_k)$.

To show the other inclusion, let $f\df X \to V_k$ be a homomorphism
with $f(X_k^+) = 0$.
Thus there is a factorization $f = h_1 \circ g_1$, where
$g_1\df X \to X/X_k^+$ is the projection.
Let $u_1\df X \to I$ be a monomorphism with $I \in \add(I_\ii)$,
and let $u_2\df I \to I/I_k^+$ be the projection.
By Lemma~\ref{homsoc}(ii) we get a monomorphism
$g_2\df X/X_k^+ \to I/I_k^+$ such that $u_2 \circ u_1 = g_2 \circ g_1$.
Now $X/X_k^+$ and $I/I_k^+$ are $A/J_{k,1}$-modules, $V_k$
is an injective $A/J_{k,1}$-module, and $g_2$ is a monomorphism.
Thus there exists a homomorphism $u_3\df I/I_k^+ \to V_k$ such that
$u_3 \circ g_2 = h_1$.
The following commutative diagram illustrates the situation:
$$
\xymatrix{
& X \ar[dl]_{u_1}\ar[d]^{g_1}\ar[r]^f & V_k \\
I \ar[d]_{u_2} & X/X_k^+ \ar[dl]^{g_2}\ar[ur]_{h_1}\\
I/I_k^+ \ar@/_3pc/[uurr]_{u_3}
}
$$
It follows that 
$$
f = h_1 \circ g_1 = u_3 \circ g_2 \circ g_1
= u_3 \circ u_2 \circ u_1.
$$
Thus we have proved that $\Hom_A(X/X_k^+,V_k) \subseteq \cI_\ii(X,V_k)$.
Note that for the proof of this inclusion we did not use the assumption
that $(A,\ii)$ is balanced.
\end{proof}

\begin{Prop}\label{st7}
Assume $(A,\ii)$ is balanced, and let $X \in \cD_\ii$.
For $1 \le k \le r$ we have
$$
D\ov{\Hom}_A(X,V_k) \cong e_{i_k}(X_k^+/X_k^-).
$$
\end{Prop}

\begin{proof}
There is a short exact sequence
$$
\eta\df \;\;\; 
0 \to X_k^+/X_k^- \to X/X_k^- \to X/X_k^+ \to 0.
$$
As noted in Lemma~\ref{lemma7} 
we have $\Hom_A(X,V_k) = \Hom_A(X/X_k^-,V_k)$,
and by Lemma~\ref{st6} we know that $\cI_\ii(X,V_k) = \Hom_A(X/X_k^+,V_k)$.
Note that $X/X_k^+$ and $X_k^+/X_k^-$ are both annihilated by $J_{k,1}$.
Thus they are $A/J_{k,1}$-modules, and $V_k = D(e_{i_k}(A/J_{k,1}))$
is an injective $A/J_{k,1}$-module.
Now we apply $\Hom_A(-,V_k)$ to
$\eta$ and obtain 
$\ov{\Hom}_A(X,V_k) \cong \Hom_A(X_k^+/X_k^-,V_k)$.
By Lemma~\ref{lemma8} we get
$$
\Hom_A(X_k^+/X_k^-,V_k) = \Hom_A(X_k^+/X_k^-, D(e_{i_k}(A/J_{k,1})))
\cong D(e_{i_k}(X_k^+/X_k^-)).
$$
Thus we have proved that
$D\ov{\Hom}_A(X,V_k) \cong e_{i_k}(X_k^+/X_k^-)$.
\end{proof}

\subsection{The quiver of $\cE_\ii$} \label{ssec:QuivEndV}
Again, let $A = \C\GG/J$ be a basic algebra, and let us fix
some sequence $\ii = (i_r,\ldots,i_1)$.
Define $\cE_\ii := \End_A(V_\ii)^\op$.
Since we work over an algebraically closed field, Lemma~\ref{st3} 
and a result by Gabriel 
(see for example \cite[Theorem~3.5.4 combined with Theorem~3.6.6]{DK}) 
imply
that $\cE_\ii$ is a finite-dimensional basic algebra.
We want to determine the quiver $\GG_{\cE_\ii}$ of
$\cE_\ii$.
The vertices of $\GG_{\cE_\ii}$ correspond to the
indecomposable direct summands $V_1,\ldots,V_r$ of $V_\ii$.

Define a quiver $\GG_\ii$ as follows:
The set of vertices of $\GG_\ii$ is just $\{ 1,2,\ldots,r \}$.
For each pair $(k,s)$ with
$1 \le s,k \le r$ and $k^+ \ge s^+ \ge k > s$ and
each arrow $a\df i_s \to i_k$ in the quiver $\GG$ of $A$,
there is an arrow $\gamma_a^{k,s}\df s \to k$ in $\GG_\ii$.
These are called the \emph{ordinary arrows} of $\GG_\ii$.
Furthermore, for each $1 \le k \le r$ there is an arrow
$\gamma_k\df k \to k^-$ provided $k^- > 0$.
These are the \emph{horizontal arrows} of $\GG_\ii$.

\begin{Prop}\label{endoquiver}
Assume that $(A,\ii)$ is balanced.
Then there is a quiver isomorphism $\GG_\ii \to \GG_{\cE_\ii}$
with $k \mapsto V_k$ for all $1 \le k \le r$.
\end{Prop}

\begin{proof}
One can almost copy the proof of \cite[Theorem III.4.1]{BIRS}. 
One only has to replace
the ideals $I_j$ used in \cite{BIRS} 
by our ideals $J_j$.
(We have $I_j = J_j$ if and only if $\GG$ has no loop at the vertex $j$.)
Furthermore, everything has to be dualized.
\end{proof}

In Proposition~\ref{endoquiver} we identify the vertex of $\GG_{\cE_\ii}$
corresponding to $V_k$ with the vertex $k$ of $\GG_\ii$.
Some examples can be found in Section~\ref{examples}.

\subsection{The $\cE_\ii$-module $D\ov{\Hom}_A(X,V_\ii)$}
Using Lemma~\ref{lemma8} together with 
Propositions~\ref{st4} and \ref{st7}, 
we arrive at the following conclusion:
Assume $(A,\ii)$ is balanced, and let $X \in \cD_\ii$.
Using the identifications
\begin{align*}
\Hom_A(V_k,V_s) &= e_{i_k}(J_{k,s+1}/J_{k,1})e_{i_s}, \\
\Hom_A(V_s,V_k) &= e_{i_s}(A/J_{s,1})e_{i_k},\\
D\ov{\Hom}_A(X,V_k) &= e_{i_k}(X_k^+/X_k^-),
\end{align*}
the algebra 
$\cE_\ii$ acts
on 
$Y := D\ov{\Hom}_A(X,V_\ii)$
as follows:
Assume $1 \le s \le k \le r$.
$$
\xymatrix{
e_{i_k}(X_k^+/X_k^-) \ar@/_1.8pc/[rrr]_{e_{i_s}(A/J_{s,1})e_{i_k} \cdot} &&&
e_{i_s}(X_s^+/X_s^-)
\ar@/_1.8pc/[lll]_{e_{i_k}(J_{k,s+1}/J_{k,1})e_{i_s} \cdot}\\
}
$$
For $e_{i_k}\ov{b}e_{i_s} \in e_{i_k}(J_{k,s+1}/J_{k,1})e_{i_s}$
and $\ov{x_s} \in e_{i_s}(X_s^+/X_s^-) = e_{i_s}X_s^+/e_{i_s}X_s^-$
we have 
$$
e_{i_k}\ov{b}e_{i_s} \cdot \ov{x_s} = e_{i_k}\ov{bx_s},
$$
and for
$e_{i_s}\ov{b}e_{i_k} \in e_{i_s}(A/J_{s,1})e_{i_k}$
and $\ov{x_k} \in e_{i_k}(X_k^+/X_k^-) = e_{i_k}X_k^+/e_{i_k}X_k^-$
we have 
$$
e_{i_s}\ov{b}e_{i_k} \cdot \ov{x_k} = e_{i_s}\ov{bx_k}.
$$
We consider $Y$ as a representation 
$Y = (Y(k),Y(\gamma))_{k,\gamma}$ of the quiver $\GG_\ii$ of $\cE_\ii$.
To describe $Y$, we just need to know how the maps $Y(\gamma)$
act on the vector spaces $Y(k) = e_{i_k}(X_k^+/X_k^-)$, where $1 \le k \le r$.
Again using the description of $\GG_{\cE_\ii}$ based on
\cite[Theorem III.4.1]{BIRS} we obtain the following result:
First, assume $\gamma_k\df k \to k^-$ is a horizontal arrow
of $\GG_\ii$.
Then $Y(\gamma_k)$ acts as left multiplication with $e_{i_k}$:
$$
\xymatrix{
e_{i_k}(X_k^+/X_k^-) \ar[rr]^{e_{i_k} \cdot} &&
e_{i_k}(X_{k^-}^+/X_{k^-}^-)
}
$$
Next, let 
$\gamma_a^{k,s}\df s \to k$ be an ordinary arrow of $\GG_\ii$.
Then $Y(\gamma_a^{k,s})$ acts as left multiplication with $a$: 
$$
\xymatrix{
e_{i_k}(X_k^+/X_k^-) && \ar[ll]_{a \cdot}
e_{i_s}(X_s^+/X_s^-)
}
$$

\begin{Rem}
{\rm
For $X \in \cD_\ii$ the following hold:
\begin{itemize}

\item[(i)]
$\cI_\ii(X,V_\ii)$ is a submodule of the 
$\End_A(V_\ii)$-module $\Hom_A(X,V_\ii)$.
This implies that $D\ov{\Hom}_A(X,V_\ii)$ is a submodule
of the $\cE_\ii$-module $D{\Hom}_A(X,V_\ii)$.
Clearly, $D\ov{\Hom}_A(X,V_\ii)$ is also a module over the
algebra $\ov{B}_\ii := (\ov{\End}_A(V_\ii))^\op$.

\item[(ii)]
For $X \in \cD_\ii$ we have
$$
\Hom_A(X,V_\ii) = \Hom_{A_\ii}(X,V_\ii).
$$
Since $\add(I_\ii)$ are the injective $A_\ii$-modules,
we can apply the Auslander-Reiten formula to obtain an isomorphism
of $\ov{B}_\ii$-modules
$$
D\ov{\Hom}_{A_\ii}(X,V_\ii) \cong \Ext_{A_\ii}^1(\tau_{A_\ii}^{-1}(V_\ii),X),
$$
where $\tau_{A_\ii}$ denotes the Auslander-Reiten translation of 
the finite-dimensional algebra $A_\ii$.

\end{itemize}
}
\end{Rem}

\subsection{An isomorphism between partial flag varieties 
and quiver Grassmannians}\label{section2.9}
In this section we prove that the varieties $\F_{\ii,\ba,X}$
of partial composition series
of modules $X \in \cD_\ii$
are isomorphic to certain quiver Grassmannians 
$\G_{\ii,\ba,X}$.
In the proof we first construct a (rather trivial) isomorphism
between partial flag varieties $\tF_{\ii,\ba,X}$
of graded vector spaces and the
image $\tG_{\ii,\ba,X}$ of the usual embedding of
$\tF_{\ii,\ba,X}$ into
a product of classical subspace Grassmannians.
Then we show that the restriction to the
subvarieties
$\F_{\ii,\ba,X} \subseteq \tF_{\ii,\ba,X}$
and
$\G_{\ii,\ba,X} \subseteq \tG_{\ii,\ba,X}$ yields
an isomorphism
$\F_{\ii,\ba,X} \to \G_{\ii,\ba,X}$.

Let $X \in \cD_\ii$ for some $\ii = (i_r,\ldots,i_1)$.
We define a map
$d_{\ii,X}\df \N^r \to \Z^r$
by $(a_r,\ldots,a_1) \mapsto (f_1,\ldots,f_r)$,
where
$$
f_k := (a_k^- - a_k) + (a_{k^-}^- - a_{k^-}) +
\cdots + (a_{k_{\rm min}}^- - a_{k_{\rm min}})
$$
for all $1 \le k \le r$, and $(a_r^-,\ldots,a_1^-) := \ba^-(X)$.
In the following theorem, if $d_{\ii,X}(\ba) \notin \N^r$, then
$\Gr_{d_{\ii,X}(\ba)}^{\cE_\ii}(Y)$ is by definition the empty set.

\begin{Thm}\label{variso}
Assume that $(A,\ii)$ is balanced, and let $X \in \cD_\ii$.
Then for each $\ba \in \N^r$
there exists an isomorphism of algebraic varieties
$$
\F\df \F_{\ii,\ba,X} \to \Gr_{d_{\ii,X}(\ba)}^{\cE_\ii}(Y),
$$
where
$Y$ is the $\cE_\ii$-module $D\ov{\Hom}_A(X,V_\ii)$.
Furthermore,
the map $\ba \mapsto d_{\ii,X}(\ba)$ yields a bijection
$\left\{ \ba \in \N^r \mid \F_{\ii,\ba,X} \not= \varnothing \right\} \to
\left\{ \ff \in \N^r \mid \Gr_{\ff}^{\cE_\ii}(Y) 
\not= \varnothing \right\}$.
\end{Thm}

Our proof of Theorem~\ref{variso} will show that 
$\dimv_{\cE_\ii}(Y) = d_{\ii,X}(\ba^+(X))$.
Furthermore, if $\F_{\ii,\ba,X} \not= \varnothing$, and
$X_\bullet = (X_r \subseteq \cdots \subseteq X_1 \subseteq X_0) \in
\F_{\ii,\ba,X}$, then
$f_k = \dim(e_{i_k}(X_k/X_k^-))$ for all $1 \le k \le r$.
Note that $f_k = 0$ if $k^+ = r+1$.

\subsection{Proof of Theorem \ref{variso}}

\subsubsection{}
Assume that $(A,\ii)$ is balanced.
For the rest of this section, besides $\ii$, we also
fix some $\ba = (a_r,\ldots,a_1) \in \N^r$ and some $X \in \cD_\ii$.
With the same notation as in Theorem~\ref{variso}, we define
$$
\G_{\ii,\ba,X} := \Gr_\ff^{\cE_\ii}(Y),
$$
where $\ff := d_{\ii,X}(\ba)$.

We consider $X$ as a representation 
$X = (X(j),X(a))_{j \in \GG_0,a \in \GG_1}$ of the quiver $\GG$ of $A$,
and the $\cE_\ii$-module $Y$ is considered as
a representation $Y = (Y(k),Y(\gamma))_{k,\gamma}$ 
of the quiver $\GG_\ii$ of $\cE_\ii$. 
Given $X_\bullet = (0 = X_r \subseteq \cdots \subseteq X_1 \subseteq X_0 = X)$
in $\F_{\ii,\ba,X}$ we consider each $X_k$ as a subrepresentation of 
$X$.
Thus we have $X_k = (X_k(j))_{j \in \GG_0}$ such that
$$
X(a)(X_k(s(a))) \subseteq X_k(t(a))
$$
for all arrows $a$ of $\GG$.

Our aim is the construction of
two mutually inverse isomorphisms of varieties
$$
\xymatrix{
{\F_{\ii,\ba,X}} \ar@<0.3pc>[rr]^<<<<<<<<<<{\F} &&
{\G_{\ii,\ba,X}} \ar@<0.3pc>[ll]^<<<<<<<<<{\G}.
}
$$

\subsubsection{}
Recall that $\GG_0 = \{1,\ldots,n\}$.
We need to work with the category of $\GG_0$-graded vector spaces.
Its objects are just tuples $W = (W(j))_{j \in \GG_0}$ of 
$\C$-vector spaces $W(j)$.
Set $e_jW := W(j)$ for all $j \in \GG_0$.
The morphisms are defined in the obvious way.
The {\it degree} of $W$ is 
$\dimv(W) := (\dim(W(j)))_{j \in \GG_0}$.
Let ${\bf e}_1,\ldots,{\bf e}_n$ denote the canonical coordinate vectors
of $\Z^n$.
(Thus the $j$th entry of ${\bf e}_j$ is 1, and all other entries
are 0.)
Each representation $X = (X(j),X(a))_{j \in \GG_0,a \in \GG_1}$ 
of $\GG$ yields a $\GG_0$-graded
vector space ${\rm gr}(X) := (X(j))_{j \in \GG_0}$. 

Let $\tF_{\ii,\ba,X}$
be the projective variety of chains
$$
X_\bullet = (0 = X_r \subseteq \cdots \subseteq X_1 \subseteq X_0 = 
{\rm gr}(X))
$$
of $\GG_0$-graded subspaces of ${\rm gr}(X)$ such that
${\rm gr}(X_k^-) \subseteq X_k \subseteq {\rm gr}(X_k^+)$ and
$\dimv(X_{k-1}/X_k) = a_k{\bf e}_{i_k}$ for all $1 \le k \le r$.

For a vector space $L$ let $\Gr_d(L)$ be the projective 
variety of $d$-dimensional
subspaces of $L$.
Clearly, the variety of $(f_k+\dim(e_{i_k}X_k^-))$-dimensional subspaces $U_k$
of $e_{i_k}X$ such that $e_{i_k}X_k^- \subseteq U_k \subseteq e_{i_k}X_k^+$
is isomorphic to $\Gr_{f_k}(e_{i_k}(X_k^+/X_k^-))$.
The isomorphism is given by 
$$
U_k \mapsto \ov{U}_k := U_k/e_{i_k}X_k^-.
$$
Let $\tG_{\ii,\ba,X}$ be the projective variety 
formed by the $r$-tuples
$\ov{U} := (\ov{U}_k)_{1 \le k \le r}$ in
$$
\prod_{k=1}^r \Gr_{f_k}(e_{i_k}(X_k^+/X_k^-))
$$
such that $U_k \subseteq U_{k^-}$ for all $1 \le k \le r$.

We construct two morphisms
$$
\xymatrix{
{\tF_{\ii,\ba,X}} \ar@<0.3pc>[rr]^<<<<<<<<<<{\tF} &&
{\tG_{\ii,\ba,X}} \ar@<0.3pc>[ll]^<<<<<<<<<{\tG}
}
$$
as follows:
First, we define $\tF$.
Let 
$X_\bullet = (0 = X_r \subseteq \cdots \subseteq X_1 \subseteq X_0 
= {\rm gr}(X))$
be in $\tF_{\ii,\ba,X}$.
To each $X_k$ we assign the subspace
$$
\ov{U}_k := e_{i_k}(X_k/X_k^-)
$$ 
of $e_{i_k}(X_k^+/X_k^-)$.
Set $\ov{U} := (\ov{U}_k)_{1 \le k \le r}$.
Then $\tF(X_\bullet) := \ov{U}$ defines a morphism of varieties
$$
\tF\df \tF_{\ii,\ba,X} \to \tG_{\ii,\ba,X}.
$$
Second, we define the morphism $\tG$.
Let $\ov{U} = (\ov{U}_k)_{1 \le k \le r}$ be in $\tG_{\ii,\ba,X}$.
We define a chain
$$
X_\bullet := (0 = X_r \subseteq \cdots 
\subseteq X_1 \subseteq X_0 = {\rm gr}(X))
$$
of $\GG_0$-graded vector spaces
as follows:
For $j \in \GG_0$ set $X_k := (X_k(j))_{j \in \GG_0}$, where
$$
X_k(j) := U_p
$$ 
and
$p := \min\{ k \le s \le r,r+1 \mid i_s = j \}$.
(Here we set $U_{r+1} := 0$.)
Then $\tG(\ov{U}) := X_\bullet$ defines a morphism of varieties
$$
\tG\df \tG_{\ii,\ba,X} \to \tF_{\ii,\ba,X}.
$$
Just using the definitions of $\tF$ and $\tG$ we obtain the following:

\begin{Lem}
The morphisms $\tF$ and $\tG$ are isomorphisms of 
algebraic varieties,
and we have 
$\tG \circ \tF = {\rm id}_{\tF_{\ii,\ba,X}}$ and 
$\tF \circ \tG = {\rm id}_{\tG_{\ii,\ba,X}}$.
\end{Lem}

\subsubsection{}
The following lemma is needed in order to ensure that
the quiver Grassmannian $\G_{\ii,\ba,X}$ is a subvariety of 
$\tG_{\ii,\ba,X}$:

\begin{Lem}\label{U}
Let $\ov{U} = (\ov{U}_k)_{1 \le k \le r}$ be a submodule of
the $\cE_\ii$-module $Y$. 
Then we have
$U_k \subseteq U_{k^-}$ for all $1 \le k \le r$.
\end{Lem}

\begin{proof}
Let $\gamma_k\df k \to k^-$ be a horizontal arrow
of $\GG_\ii$.
We know that $Y(\gamma_k)$ acts on $Y$ as follows:
$$
\xymatrix{
e_{i_k}(X_k^+/X_k^-) \ar[rr]^{e_{i_k} \cdot} &&
e_{i_k}(X_{k^-}^+/X_{k^-}^-)
}
$$
In other words,
$Y(\gamma_k)(x_k+e_{i_k}X_k^-) = x_k+e_{i_k}X_{k^-}^-$ for all 
$x_k \in e_{i_k}X_k^+$.
Since $\ov{U}$ is a submodule of $Y$, we know that
$u_k+e_{i_k}X_{k^-}^-$
is contained in $U_{k^-}/e_{i_k}X_{k^-}^-$ for all $u_k \in U_k$.
This implies $u_k \in U_{k^-}$ for all $u_k \in U_k$.
Thus $U_k \subseteq U_{k^-}$.
\end{proof}

\begin{Lem}
The following hold:
\begin{itemize}

\item[(i)]
$\F_{\ii,\ba,X}$ is a Zariski closed subset of $\tF_{\ii,\ba,X}$.

\item[(ii)]
Under the identification
$$
Y = \bigoplus_{k=1}^r e_{i_k}(X_k^+/X_k^-)
$$
the variety
$\G_{\ii,\ba,X}$ is a Zariski 
closed subset of $\tG_{\ii,\ba,X}$.

\end{itemize}
\end{Lem}

\begin{proof}
For $X_\bullet \in \tF_{\ii,\ba,X}$, the condition that
all $X_k$ are submodules is closed. This implies (i).
Similarly, for $\ov{U} \in \tG_{\ii,\ba,X}$, the condition
that $\ov{U}$ is a subrepresentation is closed. 
Now (ii) follows 
directly from Lemma~\ref{U}.
\end{proof}

We claim that $\tF$ and $\tG$ restrict to isomorphisms
$\F\df \F_{\ii,\ba,X} \to \G_{\ii,\ba,X}$
and
$\G\df \G_{\ii,\ba,X} \to \F_{\ii,\ba,X}$.
Thus, we have to show the following:
\begin{itemize}

\item[(a)]
If $X_\bullet \in \F_{\ii,\ba,X}$, then $\tF(X_\bullet) \in \G_{\ii,\ba,X}$.

\item[(b)]
If $\ov{U} \in \G_{\ii,\ba,X}$, then $\tG(\ov{U}) \in \F_{\ii,\ba,X}$.

\end{itemize}
Note that (a) and (b) imply Theorem~\ref{variso}.

\subsubsection{Proof of (a)}
Let $X_\bullet = (0 = X_r \subseteq \cdots \subseteq X_1 \subseteq X_0 = X)$
be in $\F_{\ii,\ba,X}$.
Define $\ov{U} := (\ov{U}_k)_{1 \le k \le r}$, where
$\ov{U}_k := e_{i_k}(X_k/X_k^-)$.
Thus $\tF(X_\bullet) = \ov{U}$.
We have to show that $\ov{U}$ is a subrepresentation of
$Y$.

For each 
horizontal arrow $\gamma_k\df k \to k^-$ of $\GG_\ii$ we have
$Y(\gamma_k)(\ov{U}_k) \subseteq \ov{U}_{k^-}$.
This holds, since $X_k \subseteq X_{k^-}$ and therefore
$e_{i_k}X_k \subseteq e_{i_k}X_{k^-}$.
Next, let $\gamma_a^{k,s}\df s \to k$ be an ordinary arrow of $\GG_\ii$.
It follows that $k > s$.
We have
$$
X_k \subseteq X_{k-1} \subseteq \cdots \subseteq X_{s+1} \subseteq X_s.
$$
By definition of $\F_{\ii,\ba,X}$ we have
$X_{t-1}/X_t \cong S_{i_t}^{a_t}$  for all $1 \le t \le r$.
By the definition of $\GG_\ii$ we know that $i_t \not= i_s$ for all
$k \ge t \ge s+1$.
This implies $e_{i_s}X_s = e_{i_s}X_k$.
It follows that 
$$
aX_s = ae_{i_s}X_s = ae_{i_s}X_k \subseteq e_{i_k}X_k.
$$
(To get the inclusion $ae_{i_s}X_k \subseteq e_{i_k}X_k$ we used our assumption
that $X_k$ is an $A$-module.)
This implies $Y(\gamma_a^{k,s})(\ov{U}_s) \subseteq \ov{U}_k$.
Thus we proved that $\ov{U} \in \G_{\ii,\ba,X}$.

\subsubsection{Proof of (b)}
Let $\ov{U} = (\ov{U}_k)_{1 \le k \le r}$ be a subrepresentation
of the $\cE_\ii$-module $Y$.
Thus we have $\ov{U}_k = U_k/e_{i_k}X_k^-$.
Let $X_\bullet := \tG(\ov{U})$.
Recall that
$$
X_\bullet = (0 = X_r \subseteq \cdots \subseteq 
X_1 \subseteq X_0 = {\rm gr}(X))
$$
is defined as follows:
For $1 \le k \le r$ and $j \in \GG_0$ we have
$X_k(j) = U_p$, 
where 
$p = {\rm min}\{ k \le s \le r,r+1 \mid i_s = j\}$.
(We set $U_{r+1} := 0$.)
Clearly, for all $0 \le s \le r-1$ we have
$X_{s+1} \subseteq X_s$, since by Lemma~\ref{U} we know that 
$U_{s^+} \subseteq U_s$.
It remains to show that
$X_s$ is a subrepresentation of $X$ for all $1 \le s \le r$.

By induction, we can assume that $X_r,\ldots,X_{s+1}$ are
subrepresentations of $X$.
So we only have to investigate how $A$ acts on the subspace $U_s$
of $e_{i_s}X_s$.
Obviously, $e_jX_t = X_t(j)$ for all $1 \le t \le r$ and 
$j \in \GG_0$.
Next,
assume that $\gamma_a^{k,s}\df s \to k$ is an ordinary arrow of
$\GG_\ii$.
We know that $Y(\gamma_a^{k,s})$ acts on $\ov{U}_s$ as follows:
For all $u_s \in U_s$ we have 
$$
Y(\gamma_a^{k,s})(u_s + e_{i_s}X_s^-) = (au_s) + e_{i_k}X_k^-.
$$
Since by our assumption, $\ov{U}$ is a subrepresentation of $Y$,
we get that $(au_s) + e_{i_k}X_k^-$ is contained in 
$\ov{U}_k = U_k/e_{i_k}X_k^-$
for all $u_s \in U_s$.
Thus $au_s \in U_k$ for all $u_s \in U_s$.
Now it follows from the definition of $X_s$ and Lemma~\ref{U} that 
$U_k \subseteq e_{i_k}X_s = X_s(i_k)$.
Thus $X_s$ is a subrepresentation of $X$.
It follows that 
$X_\bullet \in \F_{\ii,\ba,X}$.
This finishes the proof of Theorem~\ref{variso}.

\subsection{Examples}\label{examples}

\subsubsection{}
Let $\GG$ be the quiver 
with just one vertex $1$ and
arrows $a$ and $b$.
Set $A := \C\GG/J$, where $J$ is generated by 
$\{ ab, ba \}$.
For $\ii = (i_4,\ldots,i_1) = (1,1,1,1)$
the modules $V_k = V_{\ii,k}$ look as follows:
\begin{align*}
V_1 &= {\bsm 1 \esm} &
V_2 &= {\bsm 1&&1\\&1 \esm} &
V_3 &= {\bsm 1&&&&1\\&1&&1\\&&1 \esm} &
V_4 &= {\bsm 1&&&&&&1\\&1&&&&1\\&&1&&1\\&&&1 \esm}
\end{align*}
Obviously, $V_k$ is $(i_k,\ldots,i_1)$-balanced.
The quiver $\GG_\ii$ of $\cE_\ii$ looks as follows:
$$
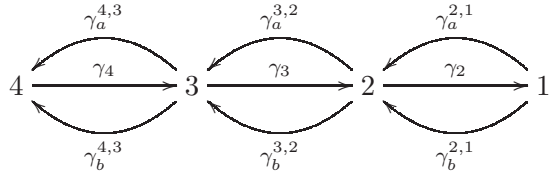
\xymatrix{
4 \ar[rr]^{\gamma_4} && 
3 \ar[rr]^{\gamma_3} \ar@/^1.5pc/[ll]^{\gamma_b^{4,3}} 
\ar@/_1.5pc/[ll]_{\gamma_a^{4,3}} &&
2 \ar[rr]^{\gamma_2} \ar@/^1.5pc/[ll]^{\gamma_b^{3,2}} 
\ar@/_1.5pc/[ll]_{\gamma_a^{3,2}} &&
1 \ar@/^1.5pc/[ll]^{\gamma_b^{2,1}} \ar@/_1.5pc/[ll]_{\gamma_a^{2,1}}
}
$$
Let $X$ be the
$A$-module
$$
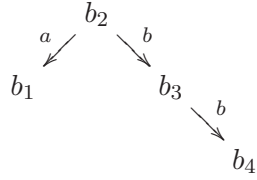
\xymatrix@-1.0pc{
& b_2 \ar[dl]_a \ar[dr]^b \\
b_1 && b_3 \ar[dr]^b \\
&&& b_4
}
$$
(Here $\{ b_1,\ldots,b_4 \}$ is a basis of $X$, and the arrows show how
the generators $a$ and $b$ of $A$ act on this basis.)
As a representation of $\GG$, we have $X = (\C^4,X(a),X(b))$,
where 
$$
X(a) = \left(\bbm 0&1&0&0\\0&0&0&0\\0&0&0&0\\0&0&0&0\ebm\right)
\text{\;\;\; and \;\;\;}
X(b) = \left(\bbm 0&0&0&0\\0&0&0&0\\0&1&0&0\\0&0&1&0\ebm\right).
$$
In the following we just write $\ebrace{\cdots}$ instead of
$\Span_\C\ebrace{\cdots}$.
The chains $X_\bullet^+$ and $X_\bullet^-$
look as follows:
\begin{align*}
X_\bullet^+ 
&= (0 
\subseteq \langle b_1,b_4 \rangle 
\subseteq \langle b_1,b_3,b_4 \rangle 
\subseteq \langle b_1,b_2,b_3,b_4 \rangle 
\subseteq \langle b_1,b_2,b_3,b_4 \rangle),\\
X_\bullet^- 
&= (0 
\subseteq 0 
\subseteq \langle b_4 \rangle 
\subseteq \langle b_1,b_3,b_4 \rangle 
\subseteq \langle b_1,b_2,b_3,b_4 \rangle).
\end{align*}
As a representation of $\GG_\ii$, the $\cE_\ii$-module 
$Y = D\ov{\Hom}_A(X,V_\ii)$ looks as follows:
$$
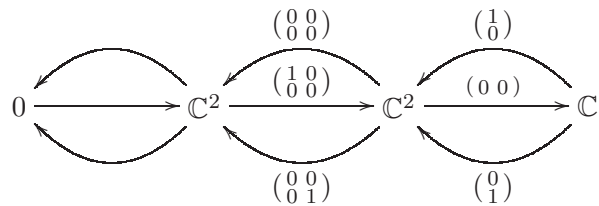
\xymatrix{
0 \ar[rr] && 
\C^2 \ar[rr]^{\left(\bsm 1&0\\0&0\esm\right)} \ar@/^1.8pc/[ll]
\ar@/_1.8pc/[ll] &&
\C^2 \ar[rr]^{\left(\bsm 0&0 \esm\right)} 
\ar@/^1.8pc/[ll]^{\left(\bsm 0&0\\0&1\esm\right)} 
\ar@/_1.8pc/[ll]_{\left(\bsm 0&0\\0&0\esm\right)} &&
\C \ar@/^1.8pc/[ll]^{\left(\bsm 0\\1\esm\right)} 
\ar@/_1.8pc/[ll]_{\left(\bsm 1\\0 \esm\right)}
}
$$
More precisely, the four vector spaces
in the above quiver representation are (from left to right) 
\begin{align*}
0 & \equiv e_1(X_4^+/X_4^-),\\
\C^2 &\equiv e_1(X_3^+/X_3^-) = 
\langle b_1,b_4\rangle
\text{ with basis } (b_1,b_4),\\
\C^2 &\equiv e_1(X_2^+/X_2^-) = 
\langle b_1,b_3,b_4\rangle/\langle b_4\rangle
\text{ with basis } (\overline{b_1},\overline{b_3}),\\
\C &\equiv e_1(X_1^+/X_1^-) = 
\langle b_1,b_2,b_3,b_4\rangle/\langle b_1,b_3,b_4\rangle
\text{ with basis } (\overline{b_2}).
\end{align*}
(Here $\overline{b_i}$ denotes the corresponding residue
class of $b_i$.)
One easily checks that the elements in $\F_{\ii,(1,1,1,1),X}$ are
\begin{align*}
f_\lambda &:=
(0 \subset \langle b_4\rangle \subset \langle b_1+\lambda b_3,b_4\rangle 
\subset \langle b_1,b_3,b_4\rangle \subset X),\\
f_\infty &:=
(0 \subset \langle b_4\rangle \subset \langle b_3,b_4\rangle 
\subset \langle b_1,b_3,b_4\rangle \subset X),\\
g_\lambda &:=
(0 \subset \langle b_1+\lambda b_4\rangle \subset 
\langle b_1,b_4\rangle 
\subset \langle b_1,b_3,b_4\rangle \subset X)
\end{align*}
where $\lambda \in \C$.
It follows 
that the Euler characteristic of 
$\F_{\ii,(1,1,1,1),X}$ is 3.
In this example, the 
isomorphism
$\F_{\ii,(1,1,1,1),X} \to \Gr_{(0,1,1,0)}^{\cE_\ii}(Y)$
from Theorem~\ref{variso} 
looks as follows:
\begin{align*}
f_\lambda &\mapsto 
(0,\langle b_4\rangle,\langle \overline{b_1} + \lambda
\overline{b_3}\rangle,0),\\
f_\infty &\mapsto (0,\langle b_4\rangle,
\langle \overline{b_3}\rangle,0),\\
g_\lambda &\mapsto 
(0,\langle b_1+\lambda b_4\rangle,
\langle \overline{b_1}\rangle,0).
\end{align*}

\subsubsection{Springer fibres}
Let $\GG$ be the quiver 
with just one vertex $1$ and
one arrow $a$.
Set $A := \C\GG/J$, where $J$ is generated by 
$a^m$ for some $m \ge 2$.
Let 
$\ii = (i_m,\ldots,i_2,i_1) = (1,\ldots,1,1)$.
For $1 \le k \le m$ the
module $V_k = V_{\ii,k}$ is uniserial of length $k$, and
$V_k$ is $(i_k,\ldots,i_1)$-balanced.
We have $\add(V_\ii) = \nil(A) = \md(A)$. 
The quiver $\GG_\ii$ of $\cE_\ii$ looks as follows:
$$
\xymatrix{
m \ar@/^0.8pc/[rr]^{\gamma_m} && 
\cdots \ar@/^0.8pc/[rr]^{\gamma_4}\ar@/^0.8pc/[ll]^{\gamma_a^{m,m-1}} &&
3\ar@/^0.8pc/[rr]^{\gamma_3}\ar@/^0.8pc/[ll]^{\gamma_a^{4,3}} &&
2 \ar@/^0.8pc/[rr]^{\gamma_2}\ar@/^0.8pc/[ll]^{\gamma_a^{3,2}} && 
1 \ar@/^0.8pc/[ll]^{\gamma_a^{2,1}}
}
$$
Let $\lambda = (\lambda_t,\ldots,\lambda_1)$ be a partition of
$m$, \ie the $\lambda_j$ are integers such that
$\lambda_t \ge \cdots \ge \lambda_1 \ge 1$ and 
$\lambda_t + \cdots + \lambda_1 = m$.
Define $V_\lambda := V_{\lambda_t} \oplus \cdots \oplus V_{\lambda_1}$.
This yields a bijection between the set of partitions of $m$
and the set of isomorphism classes of $m$-dimensional $A$-modules.
For  $\ba := (a_m,\ldots,a_2,a_1) := (1,\ldots,1,1)$
the varieties $\F_\lambda := \F_{\ii,\ba,V_\lambda}$
are just the classical Springer fibres of Dynkin type ${\rm A}_{m-1}$.

For example, let $m = 7$ and $\lambda = (3,2,2)$.
Then $V_\lambda = V_3 \oplus V_2 \oplus V_2$.
Set $Y := D\ov{\Hom}_A(V_\lambda,V_\ii)$.
By Theorem~\ref{variso} we get 
$\F_\lambda \cong \Gr_\ff^{\cE_\ii}(Y)$,
where $\ff = (f_1,\ldots,f_7) = (2,4,4,3,2,1,0)$
and
$\dimv_{\cE_\ii}(Y) = (h_1,\ldots,h_7) = (3,6,7,7,6,3,0)$.
It is an easy exercise to write 
$Y$ explicitly as a representation
of $\GG_\ii$.

\subsubsection{Balanced modules over preprojective algebras}
\label{balpre}
Let $\LL$ be the preprojective algebra associated to
a finite connected acyclic quiver $Q$. 
Recall that
$\LL = \C\overline{Q}/(c)$, 
where $\overline{Q}$ is the  
{\it double quiver } obtained
by adding to each arrow $a\df i \to j$ in $Q$ an arrow
$a^*\df j \to i$ pointing in the opposite direction, and $(c)$ is
the ideal generated by the element
$$
c = \sum_{a \in Q_1} (a^*a - aa^*).
$$
Let $\ii = (i_r,\ldots,i_1)$ be a reduced
expression for some element $w$ of the Weyl group $W$ of $Q$.
In this situation, the module $V_\ii$ defined in Section~\ref{Vii}
coincides with the cluster-tilting module $V_\ii$ of $\cC_w$
mentioned in 
Section~\ref{introV} (see \cite{GLSUni2}).
The following result is then a direct consequence of
\cite[Proposition~9.6]{GLSUni2}.

\begin{Thm}\label{preprobalanced}
$(\LL,\ii)$ is balanced.
\end{Thm}

Let $(-,-)$ denote the usual $W$-invariant bilinear form
on $\mathfrak{h}^*$, the dual of the Cartan subalgebra of the symmetric 
Kac-Moody Lie algebra $\g$ associated to $Q$.
For $i \in Q_0$ let $\alp_i$ and $\varpi_i$ be the corresponding 
simple root and fundamental weight, respectively.
By \cite[Proposition~9.6]{GLSUni2}, for 
$V_k = V_{\ii,k}$ we have
\begin{equation}\label{eq:a-V_k}
\ba_l^-(V_k) =
\begin{cases}
-(s_{i_l}s_{i_{l+1}} \cdots s_{i_k}(\varpi_{i_k}),\alp_{i_l}) 
& \text{if  $1 \le l \le k$},\\
0 & \text{otherwise}.
\end{cases}
\end{equation}
%


\section{Categorification of the Chamber Ansatz}\label{ssec:CombCp}


In this section we prove 
Theorems~\ref{THM1},~\ref{THM2} and~\ref{THM5}.
The proofs of Theorems~\ref{THM1} and~\ref{THM2} 
follow rather easily from Theorem~\ref{variso}.
As a
main ingredient for the proof of Theorem~\ref{THM5} we 
describe
the twisted $\vph$-functions 
$\vph_{V_{\bi,k}}'$ in Proposition~\ref{ch4}.
These can be seen as module-theoretic versions of
the twisted generalized minors introduced by
Berenstein and Zelevinsky \cite{BZ} (in the Dynkin case).

\subsection{Proof of Theorems~\ref{THM1} and~\ref{THM2}}
We know from Theorem~\ref{preprobalanced} that $(\LL,\ii)$ is balanced.
Let $X \in \cC_w$, and let
$$
\cE := \cE_\ii := \End_\LL(V_\ii)^\op
\quad\text{ and }\quad
\scE := \scE_\ii := \sEnd_{\cC_w}(V_\ii)^\op.
$$
As before, let $W_\ii := I_w \oplus \Omega_w(V_\ii)$.

Recall that a cluster-tilting module $T \in \cC_w$ is
called {\it $V_\ii$-reachable} if one can obtain $T$
via a finite sequence of mutations starting with
the initial cluster-tilting module $V_\ii$.
By \cite[Proposition~13.4]{GLSUni2}, the module $W_\ii$
is $V_\ii$-reachable.

Since $\stCC_w$ is a triangulated category with shift functor
$\Omega_w^{-1}$, we get an $\scE$-module isomorphism
$$
D\ov{\Hom}_\LL(X,V_\ii) \cong D{\Ext}_\LL^1(X,\Omega_w(V_\ii)).
$$
The module $I_w$ is $\cC_w$-projective-injective, thus
$\Ext_\LL^1(X,\Omega_w(V_\ii)) \cong \Ext_\LL^1(X,W_\ii)$.
The category $\stCC_w$ is a 2-Calabi-Yau category, see 
\cite[Proposition~III.2.3]{BIRS} and also
\cite{GLSUni1} for a special case.
Thus there is an $\scE$-module isomorphism
$D{\Ext}_\LL^1(X,W_\ii) \cong \Ext_\LL^1(W_\ii,X)$.
Combining these isomorphisms, we have an
$\scE$-module isomorphism
$$
D\ov{\Hom}_\LL(X,V_\ii) \cong \Ext_\LL^1(W_\ii,X).
$$
We can regard
$D\ov{\Hom}_\LL(X,V_\ii)$ and $\Ext_\LL^1(W_\ii,X)$ as modules
over $\cE$ and $\End_\LL(W_\ii)^\op$, which are
annihilated by the ideals $\cI_\LL(V_\ii,V_\ii)$ and
$\cI_\LL(W_\ii,W_\ii)$, respectively.
Since $\Omega_w\df \stCC_w \to \stCC_w$ is an equivalence,
we have an isomorphism of stable endomorphism algebras 
$$
\scE \cong \sEnd_{\cC_w}(W_\ii)^\op.
$$
Recall 
that $W_\ii$ is a cluster-tilting
module in $\cC_w$.
In particular, we have $\Ext_\LL^1(W_\ii,X) = 0$ for some
$X \in \cC_w$ if and only if $X \in \add(W_\ii)$.

By Theorem~\ref{variso} the varieties
$\F_{\ii,\ba,X}$ and $\Gr_{d_{\ii,X}(\ba)}^\cE(Y)$
are isomorphic for all $\ba \in \N^r$, where
$Y := D\ov{\Hom}_\LL(X,V_\ii)$.
Furthermore, the map
$\ba \mapsto d_{\ii,X}(\ba)$ yields a bijection
$$
\left\{ \ba \in \N^r \mid \F_{\ii,\ba,X} \not= \varnothing \right\} \to
\cU := \left\{ \ff \in \N^r \mid \Gr_{\ff}^\cE(Y) 
\not= \varnothing \right\}.
$$
(If $\Gr_\ff^\cE(Y) \not= \varnothing$, then $f_k = 0$ for all $k \in \Rma$, where $\Rma$ is defined as in 
Section~\ref{section1.4}.
Thus we can identify $\Gr_\ff^\cE(Y)$
and $\Gr_\bd^\scE(Y)$, where $\bd := (f_k)_{k \in \Rmi}$.
Being a bit sloppy, we often just write $\Gr_\ff^\scE(Y)$ instead of $\Gr_\bd^\scE(Y)$.)
Clearly, 
$\cU$
contains always the elements $\ff = \dimv(Y)$ and the 0-dimension vector
$\ff = (0,\ldots,0)$.
(In both cases, $\Gr_{\ff}^\cE(Y)$ is a single point.)
Thus there is a unique $\ba \in \N^r$ with
$\F_{\ii,\ba,X} \not= \varnothing$ if and only if $Y = 0$.
This finishes the proof of Theorems~\ref{THM1} and~\ref{THM2}.

\subsection{Example}
Let $Q$ be a quiver with underlying graph 
$\xymatrix@-0.5pc{1 \ar@{-}[r] & 2 \ar@{-}[r] & 3 \ar@{-}[r] & 4}$ and let
$\ii := (i_{10},\ldots,i_1) := (1,3,2,4,1,3,2,4,1,3)$, 
which is a reduced expression of the longest element in the Weyl group $W_Q$.
It follows that $\cC_w = \nil(\LL) = \md(\LL)$.
Let $V_\ii = V_1 \oplus \cdots \oplus V_{10}$ and 
$W_\ii = I_w \oplus \Omega_w(V_\ii) = W_1 \oplus \cdots \oplus W_{10}$.
We first display the indecomposable direct summands which are not
$\cC_w$-projective-injective:
\begin{align*}
V_1 &= {\bsm 3 \esm} &
V_2 &= {\bsm 1 \esm} &
V_3 &= {\bsm 3\\&4 \esm} &
V_4 &= {\bsm 1&&3\\&2 \esm} &
V_5 &= {\bsm 1&&3\\&2&&4\\&&3\esm} &
V_6 &= {\bsm &&3\\&2\\1 \esm} 
\\[10pt]
W_1 &= {\bsm &2&&4\\1&&3\\&2 \esm} &
W_2 &= {\bsm 2\\&3\\&&4 \esm} &
W_3 &= {\bsm &2\\1&&3\\&2 \esm} &
W_4 &= {\bsm &2&&4\\1&&33\\&2&&4\esm} &
W_5 &= {\bsm &2\\1&&3\\&2&&4 \esm}  &
W_6 &= {\bsm &&4\\&3\\2 \esm} 
\end{align*}
Finally, the indecomposable $\cC_w$-projective-injectives look as follows:
\begin{align*}
V_7 &= W_7 = {\bsm 1\\&2\\&&3\\&&&4 \esm} &
V_8 &= W_8 = {\bsm &&3\\&2&&4\\1&&3\\&2 \esm} &
V_9 &= W_9 = {\bsm &2\\1&&3\\&2&&4\\&&3 \esm} &
V_{10} &= W_{10} = {\bsm &&&4\\&&3\\&2\\1 \esm}
\end{align*}
Using our language of $\vph$-functions, the functions 
$Z_a$ and $T_J$ appearing in
\cite[Example~3.2.2]{BFZ} can be written as follows:
\begin{align*}
Z_1 &= \vph_{V_7} &
Z_2 &= \vph_{V_9} &
Z_3 &= \vph_{V_8} &
Z_4 &= \vph_{V_{10}} \\
T_2 &= \vph_{W_2} &
T_4 &= \vph_{W_6} &
T_{124} &= \vph_{W_1} &
T_{1245} &= \vph_{W_3} &
T_{245} &= \vph_{W_5} & 
T_{24} &= \vph_{W_4}
\end{align*}
and the equality
$T_{24} = \Delta^{1345}(x)\Delta^{25}(x) - \Delta^{2345}(x)\Delta^{15}(x)$
translates to
$\vph_{W_4} = \vph_{W_2}\vph_X - \vph_{W_7}\vph_Y$
where 
$$
X := {\bsm &&&4\\1&&3\\&2\esm} 
\quad\text{ and }\quad
Y := {\bsm &&4\\&3\\2\esm}.
$$

\subsection{Proof of Theorem~\ref{THM5}}
As before, let $\ii = (i_r,\ldots,i_2,i_1)$ be a reduced
expression of $w$.
Define $\jj := (i_r,\ldots,i_2)$.
(This is a reduced expression of $v := ws_{i_1}$.)
By Theorem~\ref{preprobalanced} we have
$$
\rad_{S_{i_1}}(V_{\ii,k}) \cong V_{\jj,k-1}.
$$
This yields the following result:

\begin{Lem}\label{ch1}
If $P$ is $\cC_w$-projective-injective, then
$\rad_{S_{i_1}}(P)$ is $\cC_v$-projective-injective.
\end{Lem}

\begin{Cor}\label{ch2}
If $X \in \cC_w$, then $\rad_{S_{i_1}}(X) \in \cC_v$.
\end{Cor}

\begin{proof}
There is an epimorphism $P \to X$, where $P$ is
$\cC_w$-projective-injective.
This yields an epimorphism
$\rad_{S_{i_1}}(P) \to \rad_{S_{i_1}}(X)$, see Lemma~\ref{homsoc}(ii).
Now apply Lemma~\ref{ch1}.
\end{proof}

For $1 \le k \le r$ with $k^+ \not= r+1$ we have a short exact
sequence
$$
\eta\df
0 \to W_{\ii,k} \to P(V_{\ii,k}) \to V_{\ii,k} \to 0
$$
in $\cC_w$.
The module $S_{i_1}$ is a direct summand of the rigid module
$V_\ii$.
Thus applying $\Hom_\LL(-,S_{i_1})$ to $\eta$ yields
$$
\tp_{S_{i_1}}(W_{\ii,k}) \oplus \tp_{S_{i_1}}(V_{\ii,k})
\cong \tp_{S_{i_1}}(P(V_{\ii,k})).
$$
Thus by restriction we obtain a short exact sequence
$$
0 \to \rad_{S_{i_1}}(W_{\ii,k}) \to
\rad_{S_{i_1}}(P(V_{\ii,k})) \to \rad_{S_{i_1}}(V_{\ii,k}) \to 0.
$$
By Lemma~\ref{ch1}, the module
$\rad_{S_{i_1}}(P(V_{\ii,k}))$ is $\cC_v$-projective-injective,
and by Corollary~\ref{ch2} we know that 
$\rad_{S_{i_1}}(W_{\ii,k}) \in \cC_v$.
Since
$$
\rad_{S_{i_1}}(V_{\ii,k}) \cong V_{\jj,k-1},
$$
we get
$\rad_{S_{i_1}}(W_{\ii,k}) \cong P \oplus W_{\jj,k-1}$ for some
$\cC_v$-projective-injective module $P$.
(Here we just use the basic properties of the syzygy functor
$\Omega_v$, see \cite{H}.)
Thus we have proved the following:

\begin{Lem}\label{ch3}
If we apply $\rad_{S_{i_1}}(-)$ to $\eta$, we get a short exact sequence
$$
0 \to P \oplus W_{\jj,k-1} \to P \oplus P(V_{\jj,k-1}) \to V_{\jj,k-1}
\to 0
$$
where $P$ is $\cC_v$-projective-injective.
\end{Lem}

For $1 \le l \le k \le r$ define
$$
b_\ii(l,k) := -(s_{i_l}s_{i_{l+1}} \cdots s_{i_k}(\vpi_{i_k}),\alpha_{i_l}).
$$
(For $l > k$ we define $b(l,k) := 0$.)
Note that if $l > 1$, then $b_\ii(l,k) = b_\jj(l-1,k-1)$.

\begin{Prop}\label{ch4}
For $1 \le k \le r$ we have
$$
\vph_{V_{\ii,k}}'(\sx_\bi(\bt)) = \prod_{l=1}^k t_l^{-b_\ii(l,k)} 
 = \bt^{-\ba^-(V_k)}.
$$
\end{Prop}

\begin{proof}
Let $1 \le k \le r$.
If $k^+ = r+1$, then the statement follows directly from
\cite[Proposition~9.6]{GLSUni2}.
Thus assume $k^+ \le r$.
By induction we get
$$
\vph_{V_{\jj,k-1}}'(\sx_\jj(t_r,\ldots,t_2)) = \prod_{l=2}^k t_l^{-b_\jj(l-1,k-1)} 
= \prod_{l=2}^k t_l^{-b_\ii(l,k)}. 
$$
Now Lemma~\ref{ch3} together with Theorems~\ref{THM2} and
\cite[Proposition~9.6]{GLSUni2} yield the result.
\end{proof}

The following statement is a direct consequence by Proposition~\ref{ch4}.

\begin{Cor}\label{corch4}
For $1 \le k \le r$ we have
$$
\ba^+(W_{\ii,k}) - \ba^+(P(V_{\ii,k})) = -\ba^-(V_{\ii,k})
= -( 0,\ldots,0, b_\ii(k,k),\ldots,b_\ii(2,k),b_\ii(1,k)).
$$
\end{Cor}

Now we can finish the proof of
Theorem~\ref{THM5}.
For $1 \le k \le r$ we have to show that $t_k = C_{\ii,k}(\sx_\ii(\bt))$,
where
$$
C_{\ii,k} := \frac{1}{\vph_{V_{\ii,k}}'\vph_{V_{\ii,k^-(i_k)}}'} \cdot
\prod_{j=1}^n \left(\vph_{V_{\ii,k^-(j)}}'\right)^{q(i_k,j)}
$$
and $k^-(j) := \max\{ 0,1 \le s \le k-1 \mid i_s = j \}$.
We know from Proposition~\ref{ch4} that
\begin{equation}\label{ch6}
\vph_{V_{\ii,k}}'(\sx_\ii(\bt)) = \prod_{l=1}^k t_l^{-b_\ii(l,k)}.
\end{equation}
We insert (\ref{ch5}) in the right-hand side of 
equation~(\ref{ch6}) and obtain
$$
\prod_{l=1}^k C_{\ii,l}(\sx_\ii(\bt))^{-b_\ii(l,k)}.
$$
To prove Theorem~\ref{THM5}, we need to show that
for all $1 \le k \le r$ we have
\begin{equation}\label{ch7}
\vph_{V_{\ii,k}}' = \prod_{l=1}^k C_{\ii,l}^{-b_\ii(l,k)}.
\end{equation}
This is done in exactly the same way as in~\cite[Section~4]{BZ}.
Namely, one first shows that the exponent of 
$\vph_{V_{\ii,k}}'$ on the right-hand side of equation~(\ref{ch7}) is
equal to $b(k,k) = 1$.
Then one shows that for $1 \le s < k$ the exponent of 
$\vph_{V_{\ii,s}}'$ on the right-hand side of (\ref{ch7}) is
equal to
$$
\zeta(s) := b(s,k) + b(s^+,k) - \sum_{s^+ > m > s} q(i_m,i_s)b(m,k).
$$
A straightforward calculation shows that $\zeta(s) = 0$ 
for all $1 \le s < k$.
Finally, by~\cite[Corollary~15.7]{GLSUni2}
we know that a function
$\vph_X \in \C[N]$ with $X \in \cC_w$ is already
uniquely determined by its values on $\Ima(\sx_\ii)$.
This finishes the proof of Theorem~\ref{THM5}.


\begin{Rem}
{\rm
Recall that the module $W_\ii$ is $V_\ii$-reachable.
This shows that $(\vph_{W_{\ii,1}},\ldots,\vph_{W_{\ii,r}})$
is a cluster of the cluster structure on $\C[N^w]$ defined by the
initial seed 
$((\vph_{V_{\ii,1}},\ldots,\vph_{V_{\ii,r}}),\GG_\ii)$.
By Theorem~\ref{THM5}, the cluster $(\vph_{W_{\ii,1}},\ldots,\vph_{W_{\ii,r}})$
gives a total positivity criterion for $N^w$ in the sense of \cite{BFZ,BZ}.
Therefore, every $V_\ii$-reachable cluster-tilting module of $\cC_w$  
also provides a total positivity criterion.  
}
\end{Rem}

\subsection{Example}\label{chamberex}
Let $Q$ be a quiver with underlying graph 
$\xymatrix@-0.5pc{1 \ar@{-}[r] & 2 \ar@{-}[r] & 3}$ and let
$\ii := (i_6,\ldots,i_1) := (1,2,1,3,2,1)$, which is a reduced expression of the longest
element in the Weyl group $W_Q$.
It follows that $\cC_w = \nil(\LL) = \md(\LL)$.
The modules $V_\ii = V_1 \oplus \cdots \oplus V_6$ and
$W_\ii = W_1 \oplus \cdots \oplus W_6$ look as follows:
\begin{align*}
V_1 &= {\bsm 1 \esm} &
V_2 &= {\bsm 1\\&2 \esm} &
V_3 &= {\bsm 1\\&2\\&&3\esm} &
V_4 &= {\bsm &2\\1 \esm} &
V_5 &= {\bsm &2\\1&&3\\&2\esm} &
V_6 &= {\bsm &&3\\&2\\1 \esm} \\[10pt]
W_1 &= {\bsm 2\\&3 \esm} &
W_2 &= {\bsm 3 \esm} &
W_3 &= V_3 &
W_4 &= {\bsm &3\\2 \esm} &
W_5 &= V_5 &
W_6 &= V_6
\end{align*}
Besides the modules $V_k$ and $W_k$ there are only three other indecomposable
$\LL$-modules:
\begin{align*}
L_1 &= {\bsm 1&&3\\&2 \esm} &
L_2 &= {\bsm 2 \esm} &
L_4 &= {\bsm &2\\1&&3\esm}
\end{align*}
(The reason for naming the third module $L_4$ and not $L_3$ will become clear
in Section~\ref{genericex}.)
Here we used the same conventions for displaying $\LL$-modules as explained
in \cite{GLSUni2}.
The $\vph$-functions of the indecomposable $\LL$-modules are the following:
\begin{align*}
\vph_{V_1}(\sx_\ii(\bt)) &= t_6 + t_4 + t_1 &
\vph_{W_1}(\sx_\ii(\bt)) &= t_3t_2\\
\vph_{V_2}(\sx_\ii(\bt)) &= t_5t_4 + t_5t_1 + t_2t_1 &
\vph_{W_2}(\sx_\ii(\bt)) &= t_3\\
\vph_{V_3}(\sx_\ii(\bt)) &= t_3t_2t_1 &
\vph_{W_4}(\sx_\ii(\bt)) &= t_5t_3\\
\vph_{V_4}(\sx_\ii(\bt)) &= t_6t_5 + t_6t_2 + t_4t_2 &
\vph_{L_1}(\sx_\ii(\bt)) &= t_5t_4t_3 + t_5t_3t_1\\
\vph_{V_5}(\sx_\ii(\bt)) &= t_5t_4t_3t_2 &
\vph_{L_2}(\sx_\ii(\bt)) &= t_5 + t_2\\
\vph_{V_6}(\sx_\ii(\bt)) &= t_6t_5t_3 &
\vph_{L_4}(\sx_\ii(\bt)) &= t_6t_3t_2 + t_4t_3t_2
\end{align*}
The modules $P(V_k)$ are the following:
\begin{align*}
P(V_1) &= V_3 &
P(V_2) &= V_3 &
P(V_3) &= V_3 &
P(V_4) &= V_5 &
P(V_5) &= V_5 &
P(V_6) &= V_6
\end{align*}
Thus, we obtain the twisted minors $\vph_{V_k}' = \vph_{\Omega_w(V_k)}\vph_{P(V_k)}^{-1}$:
\begin{align*}
\vph_{V_1}'(\sx_\ii(\bt)) &= \frac{t_3t_2}{t_3t_2t_1} &
\vph_{V_2}'(\sx_\ii(\bt)) &= \frac{t_3}{t_3t_2t_1} &
\vph_{V_3}'(\sx_\ii(\bt)) &= \frac{1}{t_3t_2t_1} \\[10pt]
\vph_{V_4}'(\sx_\ii(\bt)) &= \frac{t_5t_3}{t_5t_4t_3t_2} &
\vph_{V_5}'(\sx_\ii(\bt)) &= \frac{1}{t_5t_4t_3t_2} &
\vph_{V_6}'(\sx_\ii(\bt)) &= \frac{1}{t_6t_5t_3}
\end{align*}
Finally, we compute the maps $C_{\ii,k}$:
\begin{align*}
C_{\ii,1} &= \frac{1}{\vph_{V_1}'} &
C_{\ii,2} &= \frac{1}{\vph_{V_2}'} \cdot \vph_{V_1}' &
C_{\ii,3} &=  \frac{1}{\vph_{V_3}'} \cdot \vph_{V_2}' \\[10pt]
C_{\ii,4} &= \frac{1}{\vph_{V_4}' \vph_{V_1}'} \cdot \vph_{V_2}'  &
C_{\ii,5} &=  \frac{1}{\vph_{V_5}'\vph_{V_2}'} \cdot \vph_{V_4}'\vph_{V_3}' &
C_{\ii,6} &=  \frac{1}{\vph_{V_6}'\vph_{V_4}'} \cdot \vph_{V_5}'
\end{align*}
%


\section{Monomials of twisted minors} \label{sec:PfExp}


As defined before, 
let $V := V_\ii = V_1 \oplus \cdots \oplus V_r$ and
$W := W_\ii = W_1 \oplus \cdots \oplus W_r$.
For a cluster-tilting module $T$ in $\cC_w$ and any $X \in \cC_w$, we consider
$\Hom_\LL(T,X)$ as a module over $\cE_T := \End_\LL(T)^\op$.
The Ext-group $\Ext_\LL^1(T,X)$ is a module over $\cE_T$ and over
$\scE_T := \sEnd_{\cC_w}(T)^\op$.
By $\dimv \Ext_\LL^1(T,X)$ we mean the dimension vector of $\Ext_\LL^1(T,X)$
as an $\cE_T$-module.

From now on,
for the reduced expression $\ii = (i_r,\ldots,i_1)$ we assume without loss of
generality that for each $1 \le j \le n$ there is some $k$ with $i_k = j$.
We can also assume that for at least one such $j$ there are indices $k \not= s$ 
with $i_k = i_s = j$. 
(Otherwise all direct summands of $T$ are $\cC_w$-projective-injective,
\ie $\Rmi = \varnothing$.)

\subsection{}
Apart from the definitions for Theorem~\ref{THM3} the following will be
useful:
\begin{align*} 
\vph'_k &:= \vph_{V_k}' = \vph_{\Omega_w(V_k)}^\pex\vph^{-1}_{P(V_k)} 
&\text{ for } & k \in R, \text{ in particular, } \\[5pt]
\vph'_l &= \vph^{-1}_l &\text{ for } &l \in \Rma,\\[5pt]
\hvph'_k &:= \prod_{l\in R_\pex} (\vph_l')^{B^{(V)}_{l,k}} &\text{ for } & k \in \Rmi,\\[5pt]
(\vph'_\bul)^\bg &:=\prod_{k\in R_\pex} (\vph'_k)^{g_k}
&\text{ for } & \bg =(g_1,\ldots, g_r)\in\Z^r,\\[5pt]
(\hvph'_\bul)^\bd &:=\prod_{k\in\Rmi}(\hvph'_k)^{d_k} 
&\text{ for } &\bd =(d_l)_{l\in\Rmi}\in\N^{\Rmi}.
\end{align*}
Recall from \cite[Proposition~9.1]{GLSUni2} 
that the functions $\vph_{V_k}$ can be seen as
generalized minors.
The functions $\vph_{V_k}' $ are
the \emph{twisted generalized minors}. 
The following proposition describes  some special monomials in the functions
$\vph_{V_k}'$.
These results play a crucial role in the proof of Theorem~\ref{THM3}.

\begin{Prop}\label{prp:exp}
For $\bt = (t_r,\ldots,t_1) \in (\C^*)^r$ and $k \in \Rmi$ we have
\begin{equation} \label{exp1}
\hvph_{W,k}(\sx_\bi(\bt))=\hvph'_k(\sx_\bi(\bt))=
t^{\phantom{-1}}_{k^+} t_{k^{\phantom{+}}}^{-1}.
\end{equation}
Moreover, for $X \in \cC_w$ we have
\begin{align}
\label{exp2}
t^{\ba^-(X)} &= \vph_W^{(\dimv\Hom_\Lam(W,X)) \cdot B^{(W)}}(\sx_\bi(\bt))\\
\label{exp3}
&=(\vph'_\bul)^{(\dimv\Ext_\Lam^1(V,X) - \dimv\Hom_\Lam(V,X)) \cdot 
B^{(V)}}(\sx_\bi(\bt)).
\end{align} 
\end{Prop}

\begin{Rem}\label{rem:exp}
{\rm
It seems to be in general quite cumbersome to calculate the ingredients
$\ba^-(W_k) = \ba^+(W_k)$ and $B^{(W)}$ of equation~(\ref{exp2}). 
In contrast, by Corollary~\ref{corch4}
we have 
$$
\vph'_k(\sx_\bi(\bt)) = \bt^{-\ba^-(V_k)}.
$$ 
Similarly, $B^{(V)}= (\dim\Hom_\Lam(V_k,V_l)_{1\leq k,l\leq r})^{-t}$ can be
determined by our results in \cite{GLSUni2}.
Moreover, the dimension vector $\dimv\Hom_\Lam(V,X)$ depends linearly on the
multiplicities in the $\add(M_\ii)$-filtration of $X$, 
so that equation~\eqref{exp3} appears to be much
more convenient for practical purposes, see~\cite[Section~11]{GLSUni2}.
}
\end{Rem}

The rest of this section is dedicated to the proof of Proposition~\ref{prp:exp}.

\subsection{Proof of Equation~\eqref{exp1}}
For $k\in\Rmi$ we have by definition  (see equation~\eqref{ch5})
\begin{align*}
C_{\ii,k^+} &=\frac{1}{\vph'_{k^+}\vph'_k}\quad
\prod_{j\in Q_0} (\vph'_{(k^+)^-(j)})^{q(i_k,j)} \\[10pt]
&= \frac{1}{\vph'_{k^+}\vph'_k}\hspace{-1em}
\prod_{\substack{j\in Q_0\\k^+>(k^+)^-(j)>k}}\hspace{-2em}
(\vph'_{(k^+)^-(j)})^{q(i_k,j)}
\prod_{\substack{j\in Q_0\\k>(k^+)^-(j)}}\hspace{-1em}
(\vph'_{(k^+)^-(j)})^{q(i_k,j)}, \\[10pt]
C_{\ii,k} &=\frac{1}{\vph'_{k}\vph'_{k^-}}\quad
\prod_{j\in Q_0} (\vph'_{k^-(j)})^{q(i_k,j)}\\[10pt]
&=\frac{1}{\vph'_k\vph'_{k^-}}\hspace{-1em}
\prod_{\substack{j\in Q_0\\{k^+>(k^-(j))^+>k}}}\hspace{-2em}
(\vph'_{k^-(j)})^{q(i_k,j)}
\prod_{\substack{j\in Q_0\\(k^-(j))^+>k^+}}\hspace{-1em}
(\vph'_{k^-(j)})^{q(i_k,j)}.\\
\end{align*}
For $(k,j)\in\Rmi\times Q_0$ we have $k>(k^+)^-(j)$ if and only
if $(k^-(j))^+>k^+$, and in this case $(k^+)^-(j)=k^-(j)$. Thus
\[
\prod_{\substack{j\in Q_0\\k>(k^+)^-(j)}}(\vph'_{(k^+)^-(j)})^{q(i_k,j)}=
\prod_{\substack{j\in Q_0\\{(k^-(j))^+>k^+}}}(\vph'_{k^-(j)})^{q(i_k,j)}.
\]
Using Theorem~\ref{THM5} we conclude that
\begin{align*}
t_{k^+}^\pex t_k^{-1} &= (C_{\ii,k^+}^\pex C_{\ii,k}^{-1})(x_\ii(\bt))\\[10pt]
&=\left(\vph'_{k^-} (\vph'_{k^+})^{-1}\hspace{-2em}
\prod_{\substack{j\in Q_0\\((k^+)^-(j))^+>k^+>(k^+)^-(j)>k}}\hspace{-4em}
(\vph'_{(k^+)^-(j)})^{q(i_k,j)}
\prod_{\substack{j\in Q_0\\k^+>( k^-(j))^+>k>k^-(j)}}\hspace{-3em}
(\vph'_{k^-(j)})^{-q(i_k,j)}\right)(x_\ii(\bt))\\[10pt]
&={\hvph'_{k}(x_\ii(\bt))}
\end{align*}
where the last  equality follows from the description of the quiver
of $\End_\Lam(V)^\op$ in Section~\ref{ssec:QuivEndV} 
and the definition of $\hvph_k'$. 
This shows the second equality  of equation~\eqref{exp1}.

For the first equality of~\eqref{exp1} 
we compare mutations of $V$ and $W$ in direction $k$. 
Thus, for $k \in \Rmi$ we consider the short exact sequences
\[
0\ra V_k \ra \bigoplus_{l\in R} V_l^{[-B^{(V)}_{l,k}]_+}\ra V'_k\ra 0 
\qquad\text{ and }\qquad
0\ra V'_k \ra \bigoplus_{l\in R} V_l^{[B^{(V)}_{l,k}]_+}\ra V_k\ra 0
\]
as well as similar sequences for $W$.  
Since the stable endomorphism rings of 
$V$ and $W$ are isomorphic, we have $B^{(V)}_{k,l} =B^{(W)}_{k,l}$ for
all $k,l\in\Rmi$. 
Thus, if  we write 
\begin{align*}
\bar{V}^{(k)}_+ &= \bigoplus_{l\in\Rmi} V_l^{[-B^{(V)}_{l,k}]_+}  &
 P^{(k)}_+ &= \bigoplus_{l\in\Rma} V_l^{[-B^{(V)}_{l,k}]_+} \\ 
\bar{V}^{(k)}_- &= \bigoplus_{l\in\Rmi} V_l^{[B^{(V)}_{l,k}]_+} &
P^{(k)}_- &=\bigoplus_{l\in\Rma} V_l^{[B^{(V)}_{l,k}]_+}\\
\bar{W}^{(k)}_+ &= \bigoplus_{l\in\Rmi} W_l^{[-B^{(W)}_{l,k}]_+}  &
 Q^{(k)}_+ &= \bigoplus_{l\in\Rma} W_l^{[-B^{(W)}_{l,k}]_+} \\
\bar{W}^{(k)}_- &= \bigoplus_{l\in\Rmi} W_l^{[B^{(W)}_{l,k}]_+} &
Q^{(k)}_- &=\bigoplus_{l\in\Rma} W_l^{[B^{(W)}_{l,k}]_+}\\
\end{align*}
we obtain by the Snake Lemma two commutative diagrams with exact rows and
columns:
\[
\xymatrix{
        &0\ar[d]         &0\ar[d]                                     &0\ar[d]\\
0\ar[r] & W_k\ar[r]\ar[d]&\bar{W}_+^{(k)}\oplus Q^{(k)}_+\ar[r]\ar[d] &W'_k\ar[r]\ar[d] & 0\\
0\ar[r] &P(V_k)\ar[r]\ar[d]&P(\bar{V}_+^{(k)})\oplus P^{(k)}_+\oplus Q^{(k)}_+\ar[r]\ar[d]& P(V'_k)\ar[r]\ar[d]&0\\
0\ar[r] & V_k\ar[r]\ar[d]&\bar{V}_+^{(k)}\oplus P^{(k)}_+\ar[r]\ar[d] &V'_k\ar[r]\ar[d] & 0\\
        &  0             &                    0                        &     0    
}
\]
and
\[
\xymatrix{
        &0\ar[d]         &0\ar[d]                                     &0\ar[d]\\
0\ar[r] & W'_k\ar[r]\ar[d]&\bar{W}_-^{(k)}\oplus Q^{(k)}_-\ar[r]\ar[d] &W_k\ar[r]\ar[d] & 0\\
0\ar[r] &P(V'_k)\ar[r]\ar[d]&P(\bar{V}_-^{(k)})\oplus P^{(k)}_-\oplus Q^{(k)}_-\ar[r]\ar[d]& P(V_k)\ar[r]\ar[d]&0\\
0\ar[r] & V'_k\ar[r]\ar[d]&\bar{V}_-^{(k)}\oplus P^{(k)}_-\ar[r]\ar[d] &V_k\ar[r]\ar[d] & 0\\
        &  0             &                    0                        &     0    
}
\]
Since in both diagrams the end term of the respective middle row  is   
projective, both of them split, \ie
\[
P(\bar{V}_+^{(k)})\oplus P^{(k)}_+\oplus Q^{(k)}_+\cong P(V'_k)\oplus P(V_k)
\cong P(\bar{V}_-^{(k)})\oplus P^{(k)}_-\oplus Q^{(k)}_-.
\]
It follows that
\begin{equation} \label{eqn:PI1}
\vph_{Q^{(k)}_-}\vph_{Q^{(k)}_+}^{-1} = 
\vph_{P(\bar{V}^{(k)}_+)}\vph_{P^{(k)}_+}
\left(\vph_{P(\bar{V}^{(k)}_-)}\vph_{P^{(k)}_-}\right)^{-1}.
\end{equation}
On the other hand, since $B^{(V)}_{k,l}=B^{(W)}_{k,l}$ for $k,l\in\Rmi$ we can
write, with the above definitions,
\[
\hvph'_k =\frac{\vph_{\bar{W}^k_-}}{\vph_{P(\bar{V}^k_-)}}\cdot
              \frac{\vph_{P(\bar{V}^k_+)}}{\vph_{\bar{W}^k_+}}\cdot
              \frac{\vph_{P^{(k)}_+}}{\vph_{P^{(k)}_-}}\quad\text{ and }\quad
\hvph_{W,k} = \vph_{\bar{W}^k_-} \cdot \frac{1}{\vph_{\bar{W}^k_+}}\cdot
\frac{\vph_{Q^{(k)}_-}}{\vph_{Q^{(k)}_+}}.
\]
Thus $\hvph'_k=\hvph_{W,k}$ by equation~\eqref{eqn:PI1}.

\subsection{Proof of Equation~\eqref{exp2}}
Since $W$ is a cluster-tilting module, for each $X\in\cC_w$ we have a
short exact sequence
\[
0 \ra  \bigoplus_{k\in R} W_k^{g''_k} \ra 
{\bigoplus_{k\in R}}
W_k^{g'_k}\xrightarrow{p} X \ra 0
\]
for certain $\bg'=(g'_1,\ldots,g'_r)\in\N^r$ and 
$\bg''=(g''_1,\ldots,g''_r)\in\N^r$. Note, that we can assume
$g''_k=0$ for $k\in\Rma$ since $W_k$ is $\cC_w$-projective-injective 
in this case. 
Write $W''$ resp. $W'$ for the
first two terms of this sequence. Since $W$ is rigid, the sequence 
remains exact under $\Hom_\Lam(W,-)$. We conclude
that \[
(\dimv\Hom_\Lam(W,X)) \cdot B^{(W)}=\bg'-\bg''.
\]
We know from Theorem~\ref{THM2} that there is a matrix
$A \in \N^{r\times r}$ such that for $1 \le k \le r$
we have
$$
\vph_{W_k}(\sx_\bi(\bt))=\prod_{l=1}^r t_l^{A_{k,l}}.
$$
By Theorem~\ref{THM1}, 
the modules $W'$ resp. $W''$ have a unique 
partial composition series $W'_\bul=(W')^-_\bul$ resp. 
$W''_\bul=(W'')^-_\bul$ of type 
$\bi$ with 
$$
\ba' := \wt(W'_\bul) = \bg'\cdot A
\qquad\text{ resp. }\qquad 
\ba'' := \wt(W''_\bul) = \bg''\cdot A.
$$
It follows that $W''_k= W''\cap W'_k$ for $1 \le k \le r$.
With  $X_k:=p(W'_k)$ we obtain for all $k\in R$ a commutative
diagram with exact rows 
\begin{equation} \label{eq:Snake0}
\xymatrix{
0 \ar[r] & W''_{k-1} \ar[r] & W'_{k-1} \ar[r] & X_{k-1}\ar[r] & 0\\
0 \ar[r] & W''_k \ar[u]\ar[r] & W'_k \ar[u]\ar[r] & X_k\ar[u]\ar[r] & 0
}
\end{equation}
and the vertical maps being the natural inclusions. 
By the Snake Lemma we obtain a short exact sequence
\[
0\ra S_{i_k}^{a''_k}\ra S_{i_k}^{a'_k}\ra X_k/X_{k-1} \ra 0
\]
and conclude that $X_k/X_{k-1}\cong S_{i_k}^{a'_k-a''_k}$.
Thus, $X_\bul$ is a partial composition
series of type $\bi$ for $X$ with $\wt(X_\bul) = (\bg'-\bg'')\cdot A$.
Applying $\Hom_\Lam(-, S_{i_k})$ to the bottom row of the 
above diagram
shows that 
$\tp_{S_{i_k}}(X_k) = 0$ for all $k\in R$, since
$W_\bul' = (W')^-_\bul$. 
This shows that $X_\bul$ is the refined 
top series of type $\bi$ of $X$. 
All together we now have
\[
\vph_W^{(\dimv\Hom_\Lam(W,X)) \cdot B^{(W)}}(\sx_\ii(\bt)) = 
\vph_W^{\bg'-\bg''}(\sx_\ii(\bt))
=\bt^{(\bg'-\bg'')A}=\bt^{\wt(X_\bul)}=\bt^{\ba^-(X)},
\]
which is our claim.

\subsection{Proof of Equation~\eqref{exp3}}
For a $\cC_w$-projective-injective module $X \in \cC_w$ the claim is clear by
the definition of $\vph'_k$ for $k \in \Rma$. 
So we can assume that
$X$ has no non-zero $\cC_w$-projective-injective summands. 
Then we have
a short exact sequence 
\[
0\ra X\ra P(\Omega_w^{-1}(X))\ra \Omega_w^{-1}(X) \ra 0
\]
in $\cC_w$ with $\Omega_w^{-1}(X)$ having no non-zero 
$\cC_w$-projective-injective summands. 
Now we apply $\Hom_\Lam(V,-)$ and obtain
\[
\dimv\Hom_\Lam(V,\Omega_w^{-1}(X)) - \dimv\Hom_\Lam(V,P(\Omega_w^{-1}(X))) =
\dimv\Ext^1_\Lam(V,X) - \dimv\Hom_\LL(V,X).
\]
Thus,  we have to show that 
\begin{equation}\label{eq:toshow45}
(\vph'_\bul)^{(\dimv\Hom_\Lam(V,\Omega_w^{-1}(X))-\dimv\Hom_\Lam(V,P(\Omega_w^{-1}(X)))) \cdot 
B^{(V)}} =
\vph_W^{(\dimv\Hom_\Lam(W,X))\cdot B^{(W)}}.
\end{equation}
Using that $\Omega_w^{-1}$ is an autoequivalence of
the stable category $\scC_w$ we obtain again by the Snake Lemma a commutative
diagram with exact rows and columns:
\[
\xymatrix{
       & 0\ar[d]        & 0 \ar[d]                    & 0\ar[d]\\
0\ar[r]& W''\ar[d]\ar[r]& \bar{W}'\oplus P\ar[d]\ar[r]& X\ar[d]\ar[r]& 0\\
0\ar[r]&P(\Ome_w^{-1}(W''))\ar[d]\ar[r]&P(\Ome_w^{-1}(\bar{W}'))\oplus P\oplus Q\ar[d]\ar[r]&P(\Ome_w^{-1}(X))\ar[d]\ar[r]&0\\
0\ar[r]&\Ome_w^{-1}(W'')\ar[d]\ar[r]&\Ome_w^{-1}(\bar{W}')\oplus Q\ar[d]\ar[r]&\Ome_w^{-1}(X)\ar[d]\ar[r]& 0\\
       &    0           & 0                          & 0\\
}
\]
where $\bar{W}'$ and $W''$ have no non-zero $\cC_w$-projective-injective summands, 
$P$ is $\cC_w$-projective-injective,  and $\bar{W}'\oplus P=W'_X\in\add(W)$. 
Similarly, we have $\Ome_w^{-1}(\bar{W}')$
without non-zero $\cC_w$-projective-injective summands, $Q$ is 
$\cC_w$-projective-injective,  
and 
$$
\Omega_w^{-1}(\bar{W}') \oplus Q = V'_{\Omega_w^{-1}(X)} \in \add(V).
$$ 
From this it
is already clear that the  components corresponding to $\Rmi$ 
of the three vectors
\begin{align*}
& (\dimv\Hom_\Lam(W,X)) \cdot B^{(W)},\\ 
& (\dimv\Hom_\Lam(V,\Omega_w^{-1}(X))) \cdot B^{(V)},\\
& (\dimv\Hom_\Lam(V,\Omega_w^{-1}(X))-\dimv\Hom_\Lam(V,P(\Omega_w^{-1}(X))))
\cdot B^{(V)}
\end{align*}
coincide. 
Since $P(\Omega_w^{-1}(X))$ is $\cC_w$-projective-injective, the middle row of
our diagram splits. Thus we have
\[
P(\Omega_w^{-1}(X)) \oplus P(\Omega_w^{-1}(W'')) \cong 
P(\Omega_w^{-1}(\bar{W}')) \oplus P\oplus Q.
\]
Finally, from the above diagram we conclude that
\[
\vph_W^{(\dimv\Hom_\Lam(W,X))\cdot B^{(W)}}=\vph_{\bar{W}'}\vph_{W''}^{-1}\vph_{P}
\]
and
\begin{multline*}
(\vph'_\bul)^{(\dimv\Hom_\Lam(V,\Omega_w^{-1}(X))-\dimv\Hom_\Lam(V,P(\Omega_w^{-1}(X)))) \cdot 
B^{(V)}} \\
= \frac{\vph_{\bar{W}'}}{\vph_{P(\Ome^{-1}_w(\bar{W}'))}}\cdot
          \frac{\vph_{P(\Ome^{-1}_w(W''))}}{\vph_{W''}}\cdot
          \frac{\vph_{P(\Ome^{-1}_w(X))}}{\vph_Q}.
\end{multline*}
Thus~\eqref{eq:toshow45} follows from the above isomorphism of 
$\cC_w$-projective-injectives.


\section{Cluster character identities}\label{section5}


\subsection{Quivers with potential and mutations}
We review some material from~\cite[Section~4]{DWZ2},
which in turn is a review of~\cite{DWZ1}.
Let $\cP(\Gam,W) := \C\epow{\GG}/J(W)$ be the Jacobian algebra associated to a quiver
$\Gam = (\GG_0,\GG_1,s,t)$ and a potential $W\in\fm_\cyc\subset\CC\epow{\Gam}$.
For $k\in\Gam_0$ we set
\begin{align*}
\Gam^{(-,k)} &:=\{b\in\Gam_1\mid s(b)=k\},\\
\Gam^{(k,+)} &:=\{a\in\Gam_1\mid t(a)=k\},\\ 
\Gam^{(2,k)} &:=\Gam^{(-,k)}\times\Gam^{(k,+)}.
\end{align*}
For a reduced quiver potential $(\Gam, W)$ and $\Gam$ without 2-cycles
at the vertex $k\in\Gam_0$ let $\mu_k(\GG,W)$ be
the mutation of $(\GG,W)$ in direction $k$, as defined in~\cite{DWZ1}. 
For convenience, we briefly recall the construction. 
First, a possibly non-reduced quiver potential 
$\tilde{\mu}_k(\Gam,W):=(\tGam,\tW)$ is defined as follows:
The quiver $\tGam$ is obtained from $\Gam$ by inserting for each pair of arrows 
$(b,a)\in\Gam^{(2,k)}$  a new arrow $[ba]$ from $s(a)$ to $t(b)$
and replacing each arrow $c$ with $s(c)=k$ or $t(c)=k$ by a new arrow $c^*$
in the opposite direction. Then $\tW:= [W]+\Delta$, where
\[
\Delta := -\sum_{(b,a)\in\Gam^{(2,k)}}[ba]a^*b^*
\]
and $[W]$ is obtained by substituting each occurrence of a path $ba$ with
$(b,a)\in\Gam^{(2,k)}$ by the arrow $[ba]$
(after some rotation, if necessary). Our definition of
$\Delta$ deviates from the original one by a sign. However the resulting
quiver potential is right equivalent to the original one and more convenient for our 
purpose. Finally, $\mu_k(\Gam,W)$ is by definition the reduced part of $(\tGam,\tW)$. 

For a representation $M$ of $\cP(\Gam,W)$ and $k\in\Gam_1$ we need the
following notation:
\begin{align*}
M^+(k) &:= \bigoplus_{a\in\Gam^{(k,+)}} M(s(a)),\\
M^-(k) &:= \bigoplus_{b\in\Gam^{(-,k)}} M(t(b)),\\
M(\alpha_k) &:= (M(a))_{a\in\Gam^{(k,+)}}\colon M^+(k)\ra M(k),\\
M(\beta_k) &:= (M(b))_{a\in\Gam^{(-,k)}}\colon 
 M(k)\ra M^-(k),\\
M(\gamma_k) &:=  (M(\partial_{b,a}W))_{(b,a)\in\Gam^{(2,k)}}\colon M^-(k)\ra M^+(k).
\end{align*}
For an indecomposable representation $M$ of $\cP(\Gam,W)$, which is not the
simple representation $S_k$, the ``premutation'' $\tilde{\mu}_k(M) := \tM$ is a representation
of $\cP(\tGam,\tW)$ which can be described as follows~\cite{DWZ1}:
\begin{itemize}

\item 
$\tM(j) := M(j)$ for all $j\in\Gam_0\setminus\{k\}$ and $\tM(a) := M(a)$ for all 
arrows $a\in\tGam_1\cap\Gam_1$.

\item 
$\tM([ba]) := M(b)M(a)$ for all pairs of arrows 
$(b,a)\in\Gamma^{(2,k)}$.

\item 
It remains to define the maps
\begin{align*}
\tM(\alpha_k)&\colon \tM^+(k) \ra \tM(k),\\
\tM(\beta_k)&\colon \tM(k)\ra\tM^-(k)
\end{align*}
where  $\tM^+(k) := M^-(k)$ and $\tM^-(k) := M^+(k)$.

\end{itemize}

\begin{Rem} \label{Rem:MutMod}
It is an elementary exercise to verify that $\tM$ is up to isomorphism
uniquely determined by the following properties of those maps:
\begin{align}
\Ker(\tM(\alpha_k)) &=\Bi(M(\beta_k)), \label{eq:MutMod1}\\
\Bi(\tM(\beta_k))  &=\Ker(M(\alpha_k)),\label{eq:MutMod2}\\
\Ker(\tM(\beta_k))&\subseteq \Bi(\tM(\alpha_k)),
\label{eq:MutMod3}\\
\tM(\beta_k)\tM(\alpha_k) &=M(\gamma_k).\label{eq:MutMod4}
\end{align}
A concrete choice of a triple $(\tM(k),\tM(\alpha_k),\tM(\beta_k))$ 
with properties~\eqref{eq:MutMod1}--\eqref{eq:MutMod4} can be found 
in~\cite[p.765]{DWZ2}.
\end{Rem}

Next, we need to extract some material from~\cite{BIRSm}.
Let $T=T_1 \oplus \cdots \oplus T_r$ be a 
$V_\ii$-reachable cluster-tilting module in $\cC_w$.
We consider the quiver $\Gam_T$ of the 
endomorphism algebra $\cE_T := \End_\Lam(T)^\op$ and
the quiver $\sGam_T$ of the corresponding stable endomorphism algebra
$\scE_T := \sEnd_{\cC_w}(T)^\op$. 

We have $(\Gam_T)_0= R = \{1,2,\ldots,r\}$ with the vertex
$i$ corresponding to the direct summand $T_i$ of $T$. We identify $\sGam_T$ with
the full subquiver of $\Gam_T$ with vertices $\Rmi$ corresponding to the non-$\cC_w$-projective-injective indecomposable direct summands of $T$. 
Thus we get  a surjective algebra homomorphism
\[
\psi\colon \CC\epow{\sGam_T}\ra
\scE_T
\]
such that $\psi(c)\in\sHom_{\cC_w}(T_{t(c)},T_{s(c)})$ for all $c\in(\sGam_T)_1$.

We fix some $k \in \Rmi$.
Then there are short 
exact sequences
\begin{equation}\label{eq:MutSES1}
0\ra T_k\xrightarrow{\alpha_k}\bigoplus_{a\in\Gam_T^{(k,+)}} T_{s(a)}
\xrightarrow{\beta'_k} T_k'\ra 0
\end{equation}
and
\begin{equation}\label{eq:MutSES2}
0\ra T_k'\xrightarrow{\alpha'_k}\bigoplus_{b\in\Gam_T^{(-,k)}}T_{t(b)}
\xrightarrow{\beta_k} T_k\ra 0
\end{equation}
such that $\mu_k(T):= T_k' \oplus T/T_k$ is also a basic cluster-tilting
module in $\cC_w$.
It is convenient to label the components of the above maps as follows:
\begin{alignat*}{2}
\alpha_k &=(\alpha_{a,k})_{a\in\Gam_T^{(k,+)}},\qquad&
\beta'_k &=(\beta'_{k,a})_{a\in\Gam_T^{(k,+)}},\\
\alpha'_k&=(\alpha'_{b,k})_{b\in\Gam_T^{(-,k)}},\qquad&
\beta_k  &=(\beta_{k,b})_{b\in\Gam_T^{(-,k)}}.
\end{alignat*}

\begin{Lem} \label{Lem:MutEnd}
With the above notation the following hold:
\begin{itemize}

\item[(a)]
One can choose $\psi$ such that
$\Ker(\psi)=J(W_T)$  for some potential $W_T\in\fm_\cyc\subset\CC\epow{\sGam_T}$
and
\begin{alignat*}{2}
\psi(a)&=\alpha_{a,k}\in\sHom_{\cC_w}(T_k,T_{s(a)}) &&
\text{\quad for all }
a\in\sGam_T^{(k,+)},\\
\psi(b)&=\beta_{k,b}\in \sHom_{\cC_w}(T_{t(b)},T_k) &&
\text{\quad for all }
b\in\sGam_T^{(-,k)},\\
\psi(\partial_{b,a}W) &= \alpha'_{b,k}\beta'_{k,a} &&
\text{\quad for all }
(b,a)\in\sGam_T^{(2,k)}.
\end{alignat*}

\item[(b)]
We have a surjective algebra homomorphism
\[
\psi'\colon \CC\epow{\stGam_T} \ra 
\scE_{\mu_k(T)} 
\]
such that
\begin{alignat*}{2}
\psi'(a^*) &= \beta'_{k,a}\in \sHom_{\cC_w}(T_{s(a)}, T_k') &&\text{\quad for all } 
a\in\sGam_T^{(k,+)},\\
\psi'(b^*) &=\alpha'_{b,k}\in \sHom_{\cC_w}(T_k', T_{t(b)}) &&\text{\quad for all }
b\in\sGam_T^{(-,k)},\\
\psi'([ba])&=\alpha_{a,k}\beta_{k,b} &&\text{\quad for all } (b,a)\in\sGam_T^{(2,k)},\\
\psi'(c)   &=\psi(c) &&\text{\quad for all } c\in(\stGam_T)_1\cap(\sGam_T)_1.
\end{alignat*}
Moreover, $\Ker(\psi')=J(\tW_T)$.
\end{itemize}

\end{Lem}

\begin{proof}
By applying recursively~\cite[Theorem~5.3]{BIRSm} it follows 
from~\cite[Theorem~6.6]{BIRSm} and~\cite[Theorem~4.6]{BIRSm} that we can find
a \emph{liftable} (in the sense of~\cite[5.1]{BIRSm}) isomorphism
$$
\bar{\psi}\colon\cP(\sGam_T,W_T) \to \scE_T. 
$$
Thus, by~\cite[Lemma~5.7]{BIRSm} the
conditions (O)-(IV) described in~\cite[Section~5.2]{BIRSm} hold for our $\psi$.
This shows~(a). Part~(b) follows from~\cite[Theorem~5.6]{BIRSm}
and the construction of $\psi'$ (denoted by $\Phi'$ in \cite{BIRSm}), see 
also~\cite[Theorem~4.5]{BIRSm}.
\end{proof}

We now consider the following special case of the above: For an indecomposable module
$X\in\cC_w\setminus\add(T\oplus T_k')$ we consider the indecomposable
$\scE_T$-module $\Ext^1_\Lam(T,X)$ via $\bar{\psi}$ as a representation
$M$ of $\cP(\sGam_T,W_T)$. 
Similarly, we consider $\Ext_\LL^1(\mu_k(T),X)$ via 
$\bar{\psi}'$ as a representation $M'$ of
$\cP(\stGam_T,\tW_T)$.

\begin{Prop}\label{newprop}
With the above notation, the representations $M'$ and $\tilde{\mu}_k(M)$ of the Jacobian algebra
$\cP(\stGam_T,\tW_T)$ are isomorphic. Thus, we can consider
the $\scE_{\mu_k(T)}$-module 
$\Ext_\LL^1(\mu_k(T),X)$ as the mutation of 
the $\scE_T$-module $\Ext_\LL^1(T,X)$ in direction $k$.
\end{Prop}

\begin{proof}
Note that $M(\alpha_k)= \Ext^1_\Lam(\alpha_k,X)$ and
$M(\beta_k)=\Ext^1_\Lam(\beta_k,X)$ by the first two equations in 
Lemma~\ref{Lem:MutEnd}(a). 
Similarly, by Lemma~\ref{Lem:MutEnd}(b) 
we have $M'(\alpha_k)=\Ext^1_\Lam(\alpha'_k,X)$,
 $M'(\beta_k)=\Ext^1_\Lam(\beta'_k,X)$, and $M'(k)=\Ext^1_\Lam(T_k',X)$.
According to Remark~\ref{Rem:MutMod} it is sufficient to verify
\eqref{eq:MutMod1}--\eqref{eq:MutMod4} for this data.
Indeed, \eqref{eq:MutMod1} resp.~\eqref{eq:MutMod2} follows since
$\Ext^1_\Lam(-,X)$ is exact at the middle term of the short exact 
sequences~\eqref{eq:MutSES1} resp.~\eqref{eq:MutSES2}. Next,
\eqref{eq:MutMod3} is clear, since $\Ext^1_\Lam(\mu_k(T),X)$ is an
indecomposable $\scE_{\mu_k(T)}$-module. Finally, 
\[
M'(\beta_k)\circ M'(\alpha_k)=\Ext^1_\Lam(\alpha'_k\beta'_k,X)=
\Ext^1_\Lam(\psi(\partial_{b,a}W)_{(b,a)\in
\sGam_T^{(2,k)}},X)=M'(\gamma_k)
\]
by the last equation in Lemma~\ref{Lem:MutEnd}(a). 
Thus also~\eqref{eq:MutMod4} holds.
\end{proof}

\subsection{Transformation of g-vectors and F-polynomials}\label{Section5.1}
We fix a $V_\ii$-reachable cluster-tilting module $T$ in $\cC_w$, 
and define
for any  $X\in\cC_w$ the (extended) \emph{index} of $X$ with respect to $T$ as
\begin{equation}
\bg^T_X :=(\dimv\Hom_\Lam(T,X)) \cdot B^{(T)}\in\Z^r,
\end{equation}
and the \emph{F-polynomial} of $X$ with respect to $T$ as
\[
F_X^T((y_l)_{l\in\Rmi}):=
\sum_{\bd \in \N^\Rmi} \chi(\Gr_\bd^{\scE_T}(\Ext_\Lam^1(T,X)))\, y^\bd.
\]
Moreover, we will need the $\bh$-\emph{vector} of $X$ with respect to $T$
\[
\bh_X^T := (h_k)_{k \in \Rmi} 
\quad\text{ where }\quad 
h_k:=-\dim\Hom_{\scE_T}(S_k, \Ext_\Lam^1(T,X)).
\]
Fix $k\in\Rmi$ and
write $T' := \mu_k(T) = T_k' \oplus T/T_k$ for the cluster-tilting module obtained from $T$ by
mutation in direction $k$. 
Thus we have short exact sequences:
\begin{equation}\label{ses:mut-k}
0\ra T_k \ra \bigoplus_{l\in R} 
T_l^{[-B_{l,k}^{(T)}]_+}\ra T'_k\ra 0 
\quad\text{ and }\quad
0\ra T'_k \ra \bigoplus_{l\in R} T_l^{[-B_{k,l}^{(T)}]_+}\ra T_k\ra 0. 
\end{equation}

Let us recall the following observation 
from~\cite[Section~3]{FK}:
We have a short exact sequence 
\[
0\ra\bigoplus_{k\in\Rmi} T_k^{-\tilde{h}_k} \ra 
\bigoplus_{k\in R} T_k^{-\tilde{h}'_k}
\xrightarrow{\pi_X} X \ra 0
\]
such that $\pi_X$ is a minimal right $\add(T)$-approximation for certain
non-positive integers
$\tilde{h}_k$ and $\tilde{h}'_k$. 
It will be convenient to define
$\tilde{h}_{l} := 0$ for $l \in \Rma$. 
From this we obtain the exact 
sequences
\begin{equation} \label{eq:projres}
0\ra \Hom_\Lam\left(T,\bigoplus_{k\in\Rmi} T_k^{-\tilde{h}_k}\right)
 \ra \Hom_\Lam\left(T,\bigoplus_{k\in R} T_k^{-\tilde{h}'_k}\right)\ra \Hom_\Lam(T,X)\ra 0
\end{equation}
which can be viewed as a projective resolution of $\Hom_\Lam(T,X)$
over $\End_\Lam(T)^{\text{op}}$, and
\begin{equation} \label{eq:injres}
0\ra\Ext^1_\Lam(T,X)\ra
\Ext^1_\Lam\left(T,\bigoplus_{k\in\Rmi}\Omega_w^{-1}(T_k)^{-\tilde{h}_k}\right)
\ra \Ext^1_\Lam\left(T,\bigoplus_{k\in\Rmi}\Omega_w^{-1}(T_k)^{-\tilde{h}'_k}\right)
\end{equation}
which is a minimal injective copresentation
of $M := \Ext^1_\Lam(T,X)$ over the stable endomorphism ring $\scE_T$.
Thus, it follows from~\eqref{eq:projres} that
\[
\bg^T_X = (g_k)_{1 \le k \le r} = (\tilde{h}_k -\tilde{h}'_k)_{1 \le k \le r}.
\]
In particular, $g_k \ge 0$ for all $k \in \Rma$.
On the other hand, we conclude from~\eqref{eq:injres} that
\[
\tilde{h}_k = h_k := -\dim\Hom_{\scE_T}(S_k,M) 
\quad\text{ and }\quad 
\tilde{h}'_k=-\dim\Ext^1_{\scE_T}(S_k,M) 
\]
for  $k\in\Rmi$. 

We have the following easy application of deep results in~\cite{BIRSm}
and~\cite{DWZ2}:

\begin{Lem} \label{lem:transf}
Let 
\begin{align*}
\bg_X^T &= (g_k)_{k \in R}, & \bh_X^T 
&= (h_l)_{l \in \Rmi}, &
F &= F^T_X,\\
\bg_X^{T'} &= (g'_k)_{k \in R}, &
\bh_X^{T'} &= (h'_l)_{l \in \Rmi}, & 
F' &= F^{T'}_X.
\end{align*}
Moreover, let $B = (B^{(T)}_{l,m})_{l,m \in \Rmi}$, and let
$(B',(\hy'_l)_{l\in\Rmi})$ be
obtained from $(B,(\hy_l)_{l\in\Rmi})$ by Y-seed mutation in direction $k$
in $\Q_\text{sf}((y_l)_{l \in \Rmi})$.
Then for $M := \Ext^1_\Lam(T,X)$ and $k \in \Rmi$ we get
\begin{equation} \label{eq:hpk}
h'_k = -\dim\Ext^1_{\scE}(S_k,M)
\quad\text{ and therefore }\quad 
g_k = h_k-h'_k.
\end{equation}
Moreover, for $k \in \Rmi$ we have
\begin{equation}\label{eq:F-pol}
(\hy_k+1)^{h_k}F((\hy_l)_{l \in \Rmi})=(\hy'_k+1)^{h_k'}
F'((\hy'_l)_{l \in \Rmi})
\end{equation}
and for $l \in R$ we have
\begin{equation}\label{eq:g-vect}
g'_l = 
\begin{cases} g_l - h_k B^{(T)}_{l,k} + g_k[B_{l,k}^{(T)}]_+ & \text{ if } 
l\neq k,\\
-g_k & \text{ if } l=k.
\end{cases}
\end{equation}
\end{Lem}

\begin{Rem}
Observe that~\eqref{eq:g-vect} is just~\cite[(2.11)]{DWZ2} 
(proved in~\cite[Lemma 5.2]{DWZ2}) extended to our situation with coefficients.
Our independent proof for this situation is quite different. 
\end{Rem}

\begin{proof}
We know from~\cite[Theorems~5.3 and~6.4]{BIRSm} 
that the stable endomorphism algebra
$\scE_T$ is  given by a quiver with potential, and $\scE_{T'} $ is obtained
from $\scE_T$ by a mutation of quiver potentials in direction $k$. Then by Proposition~\ref{newprop}
the $\scE_{T'}$-module $M':=\Ext^1_\Lam(T',X)$ 
is obtained from
the $\scE_{T}$-module $M=\Ext^1_\Lam(T,X)$ by mutation in direction $k$ in
the sense of~\cite[(4.16) and~(4.17)]{DWZ2}. 

Now, equation~\eqref{eq:hpk} follows from 
the description of the minimal injective presentation of $M$ 
in~\cite[Remark~10.8]{DWZ2}.
In fact, with the notation used there, we have obviously
$\dim U^*_k=\dim\Ext^1_{\scE_T}(S_k,M)=\tilde{h}'_k$. 
On the other hand, by the definition
of the mutation procedure for $M$ in direction $k$ and the construction of
$U^*_k$ we have $\dim U^*_k=\dim\Hom_{\scE_{T'}}(S_k,M')=h'_k$.

Similarly,
equation~\eqref{eq:F-pol} follows now from
the ``Key-Lemma''~\cite[Lemma~5.2]{DWZ2}.

Next, let $Z_k$ be the $(r \times r)$-matrix defined by
\[
Z_k:=\left(\begin{array}{cccrccc}
1  &      &   0 &  0 &  0 &   &0\\
   &\ddots & && &\ddots\\
      0       &       &        1    &  0& 0            &      &    0\\
{[-B_{k,1}^{(T)}]_+}&\cdots &[-B_{k,k-1}^{(T)}]_+& -1
&[-B_{k,k+1}^{(T)}]_+&\cdots&[-B_{k,r}^{(T)}]_+\\
0            &       &    0         &  0&   1         &      &    0\\
     &    \ddots   &             &   &             &\ddots\\
0          &       &        0     & 0  &          0   &     & 1
\end{array}\right). 
\]
To prove equation~(\ref{eq:g-vect}), note that 
$B^{(T')} = Z_k^t B^{(T)} Z_k$, see~\cite[Proposition~7.5]{GLSRigid}. 
Moreover, if we apply $\Hom_\Lam(-,X)$ to the second short exact sequence
in~\eqref{ses:mut-k} we obtain with $M=\Ext^1_\Lam(T,X)$ the exact 
sequence
\begin{multline*}
0 \ra \Hom_\Lam(T_k,X) \ra \Hom_\Lam(\bigoplus_{l \in R} T_l^{[B^{(T)}_{l,k}]_+},X)
\ra \Hom_\Lam(T'_k,X)  \\
\ra M(k)\xrightarrow{M(\beta_k)} \bigoplus_{l \in \Rmi} M(l)^{[-B^{(T)}_{k,l}]_+}
\end{multline*}
where $\dim\Ker(M(\beta_k)) = -h_k$.
Thus 
\[
\dimv\Hom_\Lam(T',X)=(\dimv\Hom_\Lam(T,X)) \cdot Z_k^t -h_k \be_k
\]
where $\be_k$ is the $k$th standard coordinate vector of $\Z^r$.
Now, 
using $Z_k = Z_k^{-1}$ and $B^{(T)}_{k,k}=0$, we find
$\bg'= \bg Z_k + h_k \be_k B^{(T)}$.
This implies our claim since $-B^{(T)}_{k,l}=B^{(T)}_{l,k}$ for all
$(k,l)\in\Rmi\times R$.
\end{proof}

\subsection{Proof of Theorem~\ref{THM3}}\label{ssec:pf1}
Again, let $W := W_\ii$.
We show in a first step that
$$
\vph_X = \theta^W_X
$$ 
for all $X \in\cC_w$. 
It is well known that
the morphism $\sx_\bi\df(\C^*)^r \ra N^w$ is dominant. 
In fact,  by~\cite[Proposition~2.7]{L2} it is injective, and $N^w$ is 
irreducible and of dimension $r$. 
Since $\vph_X \in \C[N^w]$ and  $\theta^W_X$ is a rational function, 
it is sufficient to verify this equality on $\sx_\bi(\bt)$, 
see also~\cite[Corollary~15.7]{GLSUni2}.
Now, we have:
\begin{align*}
\vph_X(\sx_\bi(\bt))
&= \sum_{\ba \in \N^r} \chi(\F_{\bi,\ba,X})\, \bt^{\ba}\\
&= \bt^{\ba^-(X)} \sum_{\ba \in \N^r}
\chi(\Gr_{d_{\ii,X}(\ba)}^{\scE}(\Ext^1_\Lam(W,X)))\,\bt^{\ba -\ba^-(X)}\\
&= \vph_W^{(\dimv\Hom_\Lam(W,X))\cdot B^{(W)}}(\sx_\bi(\bt))
\sum_{\bd \in \N^\Rmi} \chi(\Gr_\bd^{\scE_W}(\Ext^1_\Lam(W,X)))\,
\hvph_W^\bd(\sx_\bi(\bt))\\
&= \theta^W_X(\sx_\bi(\bt)).
\end{align*}
The first equality is just the 
description of $\vph_X$ given in
\cite[Proposition~6.1]{GLSUni2}, and the second equality follows directly from
Theorem~\ref{THM1}.
To prove the third equality, we have to show that
$\bt^{\ba - \ba^-(X)} = \hvph_W^\bd(\sx_\bi(\bt))$ 
for all $\ba \in \N^r$ and $\bd = (d_k)_{k \in \Rmi} := d_{\bi,X}(\ba)$.
Proposition~\ref{prp:exp} yields
$$
\hvph_W^\bd(\sx_\bi(\bt)) = \prod_{k \in \Rmi} (t_{k^+}t_k^{-1})^{d_k},
$$
and by Theorem~\ref{variso} we have
$$ 
d_k = (a_k^- - a_k) + (a_{k^-}^- - a_{k^-}) +
\cdots + (a_{k_{\rm min}}^- - a_{k_{\rm min}})
$$
for all $k \in \Rmi$, where
$\ba^-(X) = (a_r^-,\ldots,a_1^-)$.
Observe that for $k \in \Rma$ we have $d_{k^-} = a_k-a_k^-$.
Now an easy calculation 
yields the result.
The last equality is just the definition of $\theta_X^W$.
Since the image of $\sx_\bi$ is dense
in $N^w$ and $\vph_X$ is regular (and thus continuous), we get
$\vph_X = \theta_X^W$.

Since $W$ is $V_\ii$-reachable, it remains to show
the following: 
If a new cluster-tilting module $T'$ is obtained from a
cluster-tilting module $T$ by mutation in direction $k$, then
$\theta^{T'}_X=\theta^T_X$ for all $X\in\cC_w$. This follows from
Lemma~\ref{lem:transf} with the notation used there and a calculation
inspired from the proof of~\cite[Proposition~6.8]{FZ2}:
\begin{align*}
\theta^T_X 
&= \vph_T^\bg F\left(\hvph_{T}\right)\\
&= \vph_T^\bg \left(\hvph_{T,k}+1\right)^{-h_k}
F'\left(\hvph_{T'}\right)\left(\hvph_{T',k}+1\right)^{h'_k}\\
&= \vph_T^\bg \left(\hvph_{T,k}+1\right)^{-h_k}
\left(\frac{\hvph_{T,k}+1}{\hvph_{T,k}}\right)^{h'_k} F'\left(\hvph_{T'}\right)\\
&= \vph_T^\bg \left(\hvph_{T,k}+1\right)^{-g_k}\hvph_{T,k}^{-h'_k} F'\left(\hvph_{T'}\right)\\
&= \vph_T^\bg\left({\vph_{T_k}^\pex}{\vph_{T_k'}}
\prod_{l \in R} \vph_{T_l}^{-[-B_{l,k}^{(T)}]_+}\right)^{-g_k}
\left(\prod_{l \in R} \vph_{T_l}^{B^{(T)}_{l,k}}\right)^{-h'_k} 
F'\left(\hvph_{T'}\right)\\
&= \left(\prod_{l \in R} {\vph_{T_l}}^{g_l+[-B^{(T)}_{l,k}]_+g_k - B^{(T)}_{l,k}h'_k}\right)
\left(\vph_{T_k}^\pex\vph_{T'_k}\right)^{-g_k} F'\left(\hvph_{T'}\right)\\
&=\left(\prod_{l\in R}\vph_{T_l}^{g_l+[B^{(T)}_{l,k}]_+g_k -B^{(T)}_{l,k}h_k}\right)
   \left(\vph_{T_k}^\pex\vph_{T'_k}\right)^{-g_k} F'\left(\hvph_{T'}\right)
   \quad\text{(since $g_k=h_k-h'_k$)}\\
&=\vph_{T'}^{\bg'} F'\left(\hvph_{T'}\right) 
\qquad\text{(by~\eqref{eq:g-vect})}\\
&=\theta^{T'}_X.\\
\end{align*}
(Here we set $\bg := \bg_X^T$ and $\bg' := \bg_X^{T'}$.)
At the beginning, we used that the $Y$-seed
$((B^{(T')}_{k,l})_{k,l\in\Rmi},\hvph_{T'})$ is
obtained from $((B^{(T)}_{k,l})_{k,l\in\Rmi},\hvph_{T})$ by $Y$-seed
mutation in direction~$k$ by~\cite[Proposition~3.9]{FZ2}.

\subsection{An example}
Let $Q$ be the Kronecker quiver
$$
\xymatrix{1 \ar@<0.5ex>[r]^{a}\ar@<-0.5ex>[r]_b & 2}
$$
and let $\LL = \C \ov{Q}/(c)$ be the corresponding preprojective algebra.
Then $\ii = (2,1,2,1)$ is a reduced expression for the Weyl group element
$w = s_2s_1s_2s_1$.
(The stable category $\scC_w$ is equivalent to the cluster category of 
$\C Q$, see \cite[Section~16]{GLSUni2}.)

The following picture describes the module $V := V_\ii = V_1 \oplus \cdots \oplus V_4$ and 
the quiver $\GG_V$ of $\cE_V := \End_\Lam(V)^\op$.
(The numbers 1 and 2 in the picture are basis vectors of the modules $V_k$.
The solid edges show how the arrows $a$ and $b$ of $\ov{Q}$ act on these
vectors, and the dotted edges illustrate the actions of $a^*$ and $b^*$.)
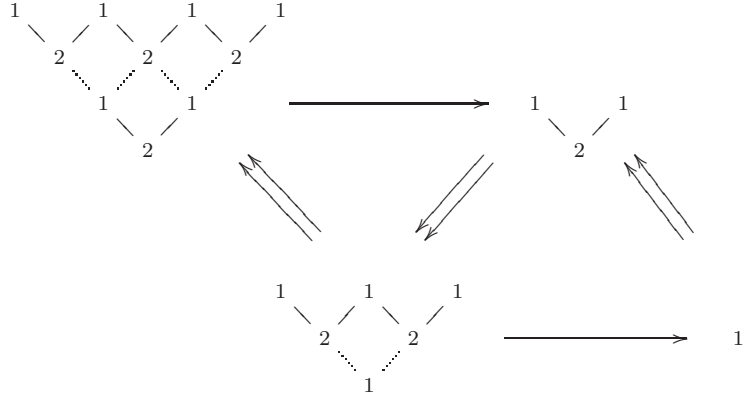
$$
\def\objectstyle{\scriptstyle}\def\labelstyle{\scriptstyle}
\xymatrix@-1.5pc{
1\ar@{-}[rd]&&\ar@{-}[ld]1\ar@{-}[rd]&&\ar@{-}[ld]1\ar@{-}[rd]&&\ar@{-}[ld]1\\
 &2\ar@{.}[rd]& &\ar@{.}[ld]2\ar@{.}[rd]& &\ar@{.}[ld]2\\
 & &1\ar@{-}[rd]& &\ar@{-}[ld]1
&&&&&&& \ar@{<-}[lllll] &1\ar@{-}[rd]& &\ar@{-}[ld]1\\
 & & &2& &\ar@<.5ex>@{<-}[rrdd]\ar@<-.5ex>@{<-}[rrdd] & & & & & & & &2&\ar@<.5ex>@{<-}[rrdd]\ar@<-.5ex>@{<-}[rrdd]\\
 & & & & & & & & & & & & & & & & &{\phantom{X}}\\
 & & & & & & & & &\ar@<.5ex>@{<-}[rruu]
\ar@<-.5ex>@{<-}[rruu] & & & & & & & &{\phantom{X}} \\
 & & & & & &1\ar@{-}[rd]& &\ar@{-}[ld]1\ar@{-}[rd]& &\ar@{-}[ld]1& & & & & &\\
 & & & & & & &2\ar@{.}[rd]& &\ar@{.}[ld]2&          & & & & & &\ar@{<-}[lllll] & 1\\
 & & & & & & & &1
}
$$
From the well-known description of $\cE_V$ by a quiver 
with relations, we obtain $B^{(V)}$, the matrix of the Ringel form of
$\cE_V$ and its inverse:
$$
B^{(V)} =
\left(\bbm
 0&-2& 1& 0\\
 2& 0&-2& 1\\
-1& 2& 1&-2\\
 0&-1& 0& 1
\ebm\right)
\quad\text{ and }\quad
(B^{(V)})^{-1} =
\left(\bbm
 1& 2& 3& 4\\
 0& 1& 2& 3\\
 1& 2& 4& 6\\
 0& 1& 2& 4
\ebm\right)
$$
Note that the entries of $(B^{(V)})^{-1}$ describe the dimensions
of  $\Hom_\Lam(V_j,V_i)$. 
Observe also that  
$\scE_V := \sEnd_{\cC_w}(V)^\op$ is isomorphic to $\C Q$.
Next, we describe
$$
W := W_\ii = I_w \oplus \Omega_w(V_\bi) = W_1 \oplus \cdots \oplus W_4.
$$ 
Note that $I_w = V_3 \oplus V_4$, $W_3 = V_3$ and $W_4 = V_4$.
We have short exact sequences
$$
0 \to W_1\to V_3^3 \to V_1 \to 0
\qquad\text{ and }\qquad
0 \to W_2 \to V_3^2 \to V_2 \to 0.
$$
Thus
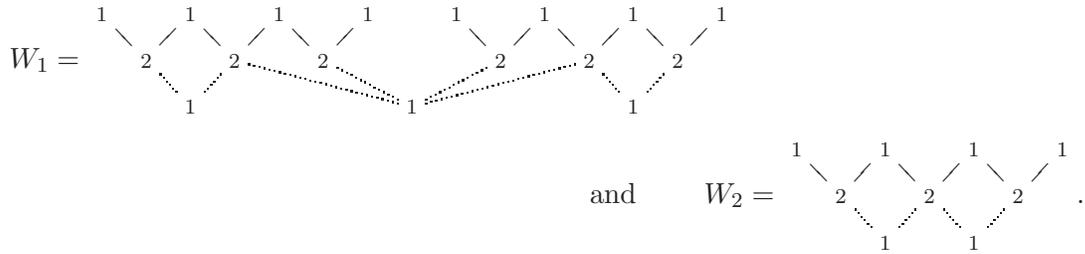
\begin{multline*}
W_1=\vcenter{\def\objectstyle{\scriptstyle}\def\labelstyle{\scriptstyle}
\xymatrix@-1.5pc{
1\ar@{-}[rd]&&\ar@{-}[ld]1\ar@{-}[rd]&&\ar@{-}[ld]1\ar@{-}[rd]&&\ar@{-}[ld]1&&1\ar@{-}[rd]&&\ar@{-}[ld]1\ar@{-}[rd]&&\ar@{-}[ld]1\ar@{-}[rd]&&\ar@{-}[ld]1\\
 &2\ar@{.}[rd]& &\ar@{.}[ld]2\ar@{.}[rrrrd]& &2\ar@{.}[rrd]& & & &\ar@{.}[lld]2& &\ar@{.}[lllld]2\ar@{.}[rd]& &\ar@{.}[ld]2\\
 & &1& & & & &1& & & & &1\\ 
}}\\
\text{and }\qquad W_2=\vcenter{\def\objectstyle{\scriptstyle}\def\labelstyle{\scriptstyle}
\xymatrix@-1.5pc{
1\ar@{-}[rd]& &\ar@{-}[ld]1\ar@{-}[rd]& &\ar@{-}[ld]1\ar@{-}[rd]& &\ar@{-}[ld]1\\  
 &2\ar@{.}[rd]& &\ar@{.}[ld]2\ar@{.}[rd]& &\ar@{.}[ld]2& \\
 & &1& &1& & \\
}}.
\end{multline*}
In our situation, we have $W = (\mu_1 \circ \mu_2)(V)$.
Using \cite[Section~7]{GLSRigid} we get
\[
B^{(W)}=\left(\bbm
 0&-2& 3& 0\\
 2& 0&-4&-1\\
-3& 4& 1& 0\\
 0& 1&-2& 1
\ebm\right)
\quad\text{ and }\quad
(B^{(W)})^{-1}=\left(\bbm
25&14& 9&14\\
16& 9& 6& 9\\
11& 6& 4& 6\\
 6& 3& 2& 4
\ebm\right).
\]
The quiver $\GG_W$ of 
$\cE_W := \End_\Lam(W)^\op$ looks as follows:
$$
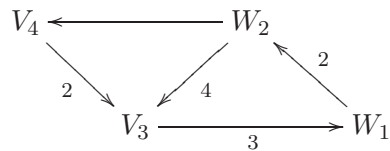
\xymatrix{
V_4  && W_2 \ar[ll]\ar@{<-}[rd]^2\\
& \ar@{<-}[lu]^2 V_3 \ar@{<-}[ru]_4 && \ar@{<-}[ll]^3 W_1
}
$$
(Here we write $\xymatrix{i \ar[r]^m &j}$ in case there are $m$ arrows from $i$ to $j$.)

For $\lam \in \C$ we consider the following $\LL$-module in $\cC_w$:
$$
X_\lam := \vcenter{\def\objectstyle{\scriptstyle}\def\labelstyle{\scriptstyle}
\xymatrix@-1.2pc{
1\ar@{-}[rd]&&\ar@{-}[ld]1\ar@{-}[rd]&&\ar@{-}[ld]1\ar@{-}[rd]&&\ar@{-}[ld]1\\
 &2\ar@{.}[rrd]&&\ar@<-.3ex>@{.}[d]2\ar@<.3ex>@{.}[d]^{\lam}&&\ar@{.}[lld]2\\
 &&& 1
}}
$$
Note that $X_\lam \cong \Omega_w(X_\lam)$.
A direct calculation shows that
$\Ext^1_\Lam(V,X_\lam)$ is an (indecomposable) regular $\C Q$-module
with dimension vector $(1,1)$ and that
$\dimv\Hom_\Lam(V, X_\lam)=(1,2,3,5)$.
This implies that
$$
(\dimv\Hom_\Lam(V, X_\lam)) \cdot B^{(V)}=(1,-1,0,1).
$$
Using the exact sequences
\[
0\ra\Hom_\Lam(V_k,X_\lam)\ra\Hom_\Lam(P(V_k),X_\lam)\ra\Hom_\Lam(W_k,X_\lam)\ra
\Ext^1_\Lam(V_k,X_\lam)\ra 0
\]
for $k=1,2$, we obtain
$\dimv\Hom_\Lam(W,X_\lam)=(9,5,3,5)$, and therefore
$$
(\dimv\Hom_\Lam(W,X_\lam)) \cdot B^{(W)} = (1,-1,0,0).
$$
Finally, since $\Omega_w^{-1}(X_\lam) \cong X_\lam$, we get that also
$\Ext^1_\Lam(W,X_\lam)$ is an indecomposable regular $\C Q$-module with
dimension vector $(1,1)$.

It is straightforward to calculate (using the Euler characteristics of
the corresponding varieties of partial composition series) the 
following:
\begin{align*}
\vph_{X_\lam}(\sx_\bi(\bt)) &= t_3^\pex t_2^3t_1^4+t_4^\pex t_3^\pex t_2^2t_1^4+
                           t_4^\pex t_3^2t_2^2t_1^3\\
\vph_{V_1}(\sx_\bi(\bt))   &= t_3 + t_1\\
\vph_{V_2}(\sx_\bi(\bt))   &= t_4(t_3^2+ 2t_3t_1+t_1^2)+ t_2^\pex t_1^2\\
\vph_{V_3}(\sx_\bi(\bt))   &= t_3^\pex t_2^2t_1^3\\
\vph_{V_4}(\sx_\bi(\bt))   &= t_4^\pex t_3^2t_2^3t_1^4\\
\vph_{W_1}(\sx_\bi(\bt))   &= t_3^3t_2^6t_1^8\\
\vph_{W_2}(\sx_\bi(\bt))   &= t_3^2t_2^3t_1^4
\end{align*}
Note that the $\vph$-functions of the direct summands of 
$W$ evaluate on $\sx_\bi(\bt)$ to monomials.

The $F$-polynomial of each indecomposable representation of
\[
\sEnd_{\cC_w}(V)^\op \cong \sEnd_{\cC_w}(W)^\op \cong \C Q
\]
with dimension vector $(1,1)$ is $1+y_2+ y_1y_2$.
Thus, from the definitions we get
\begin{align*}
\theta_{X_\lambda}^W &= 
\vph^\pexx_{W_1}\vph_{W_2}^{-1}
(1+\hvph^\pex_{W,2}+\hvph^\pex_{W,1}\hvph^\pex_{W,2}) ,\\
\theta_{X_\lambda}^V &= \vph^\pex_{V_1}\vph_{V_2}^{-1}\vph^\pexx_{V_4}(1+\hvph^\pex_{V,2}+\hvph^\pex_{V,1}\hvph^\pex_{V,2}).
\end{align*}
Thus, by Theorem~\ref{THM3} we should have
\begin{equation}\label{eq:ex}
\vph^\pex_{X_\lam} = \theta_{X_\lambda}^V = \theta_{X_\lambda}^W.
\end{equation}
From the matrices $B^{(V)}$ resp. $B^{(W)}$ we get
\begin{align*}
\hvph^\pex_{V,1} &= \vph_{V_2}^2\vph_{V_3}^{-1},&
\hvph^\pex_{W,1} &= \vph_{W_2}^2\vph_{V_3}^{-3},\\
\hvph^\pex_{V,2} &= \vph_{V_1}^{-2}\vph_{V_3}^2\vph_{V_4}^{-1}, &
\hvph^\pex_{W,2} &= \vph_{W_1}^{-2}\vph_{V_3}^4\vph^\pex_{V_4}.
\end{align*}
We verify equation~\eqref{eq:ex} by evaluation on $\sx_\bi(\bt)$. 
First, we observe
that
$$
\hvph^\pex_{W,1}(\sx_\bi(\bt))=t_3^\pex t_1^{-1}
\quad\text{ and }\quad
\hvph^\pex_{W,2}(\sx_\bi(\bt))=t_4^\pex t_2^{-1}.
$$
This implies that
$$
(\vph^\pex_{W_1}\vph_{W_2}^{-1}(1+\hvph_{W,2}+\hvph_{W,1}\hvph_{W,2}))(\sx_\bi(\bt))
=t_3^\pex t_2^3t_1^4(1+t_4^\pex t_2^{-1}+t_4^\pex t_3^\pex t_2^{-1}t_1^{-1})=\vph_{X_\lam}(\sx_\bi(\bt)).
$$
On the other hand, from the definitions we get
\[
\vph^\pex_{V_1}\vph_{V_2}^{-1}\vph^\pexx_{V_4}
(1+\hvph_{V,2}+\hvph_{V,1}\hvph_{V,2})=
\vph_{V_1}^{-1}\vph_{V_2}^{-1}
(\vph_{V_1}^2\vph^\pex_{V_4}+\vph_{V_3}^2+\vph_{V_2}^2\vph^\pex_{V_3}).
\]
Evaluation at $\sx_{\bi}(\bt)$ yields
\begin{align*}
\frac{t_3^\pex t_2^2t_1^3(t^\pex_4(t^\pex_3+t^\pex_1)^2+t_2^\pex t_1^2)^2+t_4^\pex t_3^2t_2^3t_1^4(t^\pex_3+t^\pex_1)^2+
t_3^2t_2^4t_1^6}{(t_3+t_1)(t_4(t_3+t_1)^2+t_2^\pex t_1^2)}
&= t_4^\pex (t_3^2t_2^2t_1^3+t_3^\pex t_2^2t_1^4)+t_3^\pex t_2^3t_1^4 \\
&= \vph_{X_\lambda}(\sx_\bi(\bt)).
\end{align*}


\section{E-invariant and Ext}\label{section6}


For a cluster-tilting module $T$ in $\cC_w$, and $\scE := \scE_T$,
let $E\df \cC_w \ra \md(\scE)$ be the functor defined by $X \mapsto \Ext^1_\Lam(T,X)$.
This functor is known to be dense 
(\ie up to isomorphism, all objects
in $\md(\scE)$ are in the image of $E$), 
see for example~\cite[Proposition~2.1(c)]{KR}.

The following proposition shows that, for an
$\scE$-module of the form $EX$, the E-invariant 
defined in
\cite{DWZ2} has a nice geometric description as the codimension
of the orbit of $X$ in the nilpotent variety.
This result plays a crucial role in the proof of
Proposition~\ref{prp:bas} and of Theorem~\ref{THM4}.

\begin{Prop}\label{prp:EinvExt}
Let $T\in\cC_w$ be a $V_\ii$-reachable cluster-tilting module. Then for
any $X,Y\in\cC_w$ there is a short exact sequence
\[
0\ra D{\Hom}_\scE(Y,\tau_\scE(X)) \ra \Ext^1_\LL(X,Y) \ra \Hom_\scE(X,\tau_\scE(Y)) \ra 0.
\]
In particular, if $\dimv_\Lam(X)=\bd$ we have
\[
\codim_{\Lam_\bd^w}(\GL_\bd . X)=\dim\Hom_\scE(EX,\tau_\scE(EX))=
\dim\Hom_\scE(\tau_\scE^{-1}(EX),EX).
\]
\end{Prop}

\begin{proof}
The stable endomorphism algebra $\scE$ of $T$ is the Jacobian algebra of a 
quiver with potential, see~\cite{BIRSm}. 
Denote by $G_T$ the corresponding Ginzburg dg-algebra
and by $\cDp(G_T)$ resp. $\cD_{\text{f.d.}}(G_T)$ the
corresponding subcategory of perfect complexes resp. of complexes with
total finite-dimensional cohomology of the derived category.
The shift in $\cDp(G_T)$ is denoted by $\Sigma$.
Following
Amiot~\cite{A} we have the generalized cluster category as the
triangulated quotient
\[
\cC_T:=\cDp(G_T)/\cD_{\text{f.d.}}(G_T).
\]
It follows from~\cite{ART} that $\cC_w\cong\cC_V$ as triangulated 
categories, and then from~\cite{BIRSm} and~\cite{KY} that
$\cC_V\cong\cC_T$.  
Next, denote by $\cF \subseteq \cDp(G_T)$
the subcategory which consists of the cones of maps in $\add(G_T)$.
Then the canonical projection 
$\cDp(G_T) \ra \cC_T$ induces an equivalence of additive
categories $\cF \ra \cC_T$~\cite[Proposition~2.9, Lemma~2.10]{A}, see 
also~\cite[Remark~4.1]{KY}. 
By~\cite[Proposition~2.12]{A} we have for
$X,Y\in\cF$ a short exact sequence
\[
0\ra\Ext^1_{\cDp(G_T)}(X,Y)\ra \Ext^1_{\cC_T}(X,Y)\ra
D{\Ext}^1_{\cDp(G_T)}(Y,X)\ra 0.
\]
We have to show that
$D{\Ext}^1_{\cDp(G_T)}(X,Y)\cong\Hom_{\scE}(H^0(Y),\tau_\scE(H^0(X)))$, 
since
$H^0(X)\cong\Hom_{\cDp(G_T)}(G_T,X)=\Hom_{\cC_T}(T,X)$ for $X\in\cF$.

To this end we choose a minimal presentation
\begin{equation}\label{eq:mpres}
P_1\xrightarrow{p} P_0\ra X\ra\Sig P_1 \text{ with } P_0, P_1\in\add(G_T).
\end{equation}
Now, for $P\in\add(G_T)$ and $Y\in\cDp(G_T)$ we have
\[
\Hom_{\cDp(G_T)}(P,Y)=\Hom_\scE(H^0(P),H^0(Y)) 
\]
and $\Hom_\cDp(P,\Sig Y)=0$ if $Y\in\cF$. Thus, 
\[
\Ext^1_{\cDp(G_T)}(X,Y)=\Hom_{\cDp(G_T)}(X,\Sig Y)\cong
\Coker(\Hom_\scE(H^0(p),H^0(Y))).
\]
Next, the sequence
\[
H^0(P_1)\xrightarrow{H^0(p)} H^0(P_0)\ra H^0(X)\ra 0
\]
is a minimal projective presentation in $\md(\scE)$ because of the
minimality of~\eqref{eq:mpres}. Since we have
$\Hom_\scE(H^0(P), L)\cong D{\Hom}_\scE(L,\nu_\scE(H^0(P)))$ for $L \in \md(\scE)$
and $P\in\add(G_T)$ we conclude that
\[
D{\Hom}_{\cDp(G_T)}(X,\Sig Y)\cong \Ker(\Hom_\scE(H^0(Y),\nu_\scE(H^0(p)))) =
\Hom_\scE(H^0(Y),\tau_\scE(H^0(X))).
\]
Here $\nu_\scE$ is the usual Nakayama functor.
Finally, by Lemma~\ref{lem:e2cdim} below, we have 
$$
\dim\Ext^1_\Lam(X,X)=2\codim_{\Lam^w_\bd}(\GL_\bd.X).
$$
This yields the result.
\end{proof}


\section{Generic bases for cluster algebras}\label{section7}


\subsection{Generically reduced components}
For a $V_\ii$-reachable cluster-tilting module $T$ in $\cC_w$,
the algebra $\scE := \scE_T$
is given by a quiver with 
potential~\cite{BIRSm}.
Then by~\cite[Corollary~10.8]{DWZ2},
$$
\dim \Hom_\scE(\tau_\scE^{-1}(Y),Y) = E^{\text{inj}}(Y)
$$ is the E-\emph{invariant} defined in \cite{DWZ2}.
Since this 
translates into a simple rank condition~\cite[Equation~(1.17)]{DWZ2},
for each irreducible component $\cZ \in \irr(\scE)$ the following hold:
\begin{itemize}

\item[(i)] 
There is a dense open subset $\cU' \subseteq \cZ$ and 
a unique $h(\cZ) \in \N$ such that 
$$
\dim \Hom_\scE(\tau_\scE^{-1}(U),U) = h(\cZ)
$$ 
for all $U \in \cU'$.

\item[(ii)]
There is a dense open subset $\cU'' \subseteq \cZ$ and a unique
$e(\cZ) \in \N$ such that 
$$
\dim \Ext^1_\scE(U,U) = e(\cZ)
$$ 
for all $U \in \cU''$.

\item[(iii)]
There is a dense open subset $\cU''' \subseteq \cZ$ and a unique 
$c(\cZ) \in \N$ 
such that 
$$
\codim_{\cZ}(U.\GL_{\bd}) = c(\cZ)
$$ 
for all $U \in \cU'''$.

\end{itemize}
It is well known that 
$$
c(\cZ) \le e(\cZ) \le h(\cZ).
$$ 
(For the second inequality, one uses the Auslander-Reiten formula
$$
\Ext_\scE^1(U,U) \cong D\sHom_\scE(\tau_\scE^{-1}(Y),Y).
$$
For an algebra $A$ and $A$-modules $M$ and $N$, 
$\sHom_A(M,N)$ denotes the homomorphism space
$\Hom_A(M,N)$ modulo the subspace of homomorphisms factoring through
projectives.)
It follows from Voigt's Lemma 
\cite[Proposition~1.1]{G} that $\cZ$ is (scheme-theoretically) 
generically reduced if and only if $c(\cZ) = e(\cZ)$.
Recall that $\cZ$ is \emph{strongly reduced} if $c(\cZ) = h(\cZ)$. 
So, strongly reduced components are in particular generically reduced.

\subsection{Open subsets of nilpotent varieties}
For $\bd\in\N^n$ let $\Lam_\bd$ be the affine variety of nilpotent 
representations with dimension vector $\bd$ of the preprojective 
algebra $\Lam$. Following Lusztig~\cite[Section~12]{L1}, 
$\Lam_\bd$ is equidimensional with
\begin{equation}\label{eq:dimnil}
\dim(\Lam_\bd) =\sum_{a\in Q_1} \bd(s(a))\bd(t(a)).
\end{equation}

On $\Lam_\bd$ acts the group 
$\GL_\bd=\prod_{i\in Q_0}\GL_{\bd(i)}(\C)$ from the left by conjugation.
We have the following surprising result, which we borrow 
from~\cite[Lemma~4.3]{GLSUni1}.

\begin{Lem}\label{lem:e2cdim}
Let $M$ be a nilpotent $\Lam$-module with $\dimv_\LL(M)=\bd$. Then
\[
2\codim_{\Lam_\bd}(\GL_\bd.M)=\dim\Ext^1_\Lam(M,M).
\]
\end{Lem}

\begin{proof}
We have
\[
2\codim_{\Lam_\bd}(\GL_\bd.M)=2(\dim(\Lam_\bd) -\dim(\GL_\bd) +
\dim\End_\Lam(M))=
\dim\Ext^1_\Lam(M,M),
\]
where the last equality holds by~\eqref{eq:dimnil} and~\cite[Lemma~1]{CB2}.
\end{proof}

Using the notation from Section~\ref{isosection}, let $J_w := J_{r,1}$, where $\LL = A$ and
$\ii = (i_r,\ldots,i_1)$ is a reduced expression of $w$.
This definition does not depend on the choice of $\ii$, see \cite[Proposition~III.1.8]{BIRS}.

\begin{Lem} \label{lem:lamwopen}
We have
\[
\cC_w = \{ X\in\md(\Lam)\mid \Ext^1_\Lam(D(\Lam/J_w),X)=0=\Hom_\Lam(X,D(J_w)) \}.
\]
Thus, the subset $\Lam^w_\bd:=\{X\in\Lam_\bd\mid X\in\cC_w\}$ is open in 
$\Lam_\bd$.
\end{Lem}

\begin{proof}
Since $\Lam/J_w$ is Gorenstein~\cite[III Proposition~2.2 and Corollary~3.6]{BIRS}
we have 
\[
\cC_w:=\Fac(\Lam/J_w)=\{X\in\md(\Lam/J_w)\mid\Ext^1_{\Lam/J_w}(D(\Lam/J_w),X)=0\}.
\]
We consider the short exact sequence
\[
0 \ra D(\Lam/J_w) \xrightarrow{i} D(\Lam) \xrightarrow{p} D(J_w) \ra 0.
\]
Since  $\Hom_\Lam(X,D(\Lam)) \cong D(X)$ and
$\Hom_\Lam(X,D(\Lam/J_w)) \cong D(X/J_wX)$ naturally, 
we conclude that $X\in\md(\Lam/J_w)$ if and only if $\Hom_\Lam(X,p)$ 
is an isomorphism.

For $X\in\Fac(\Lam/J_w)$ we have by~\cite[III.2.3]{BIRS},  
$$
0 = \Ext^1_\Lam(D(\Lam/J_w),X) \cong D{\Ext}^1_\Lam(X,D(\Lam/J_w)).
$$ 
It follows that
\[
0\ra\Hom_\Lam(X, D(\Lam/J_w)) \xrightarrow{\Hom_\LL(X,i)}\Hom_\LL(X, D(\Lam)) \ra
\Hom_\LL(X, D(J_w))\ra 0
\]
is exact. 
Now, $\Hom_\Lam(X,i)$ is an isomorphism since $X\in\md(\Lam/J_w)$.
This implies that $\Hom_\Lam(J_w,X)=0$.

Conversely, if $X\in\md(\Lam)$ fulfills the conditions of the lemma
we conclude  from $\Hom_\Lam(X,D(J_w))=0$ that $\Hom_\Lam(X,i)$ is an
isomorphism. Thus $X\in\md(\Lam/J_w)$, and we infer 
$\Ext^1_{\Lam/J_w}(D(\Lam/J_w),X)=0$ from $\Ext^1_\Lam(D(\Lam/J_w),X)=0$.

If $\Lam$ is of Dynkin type, \ie finite-dimensional, the conditions of
the lemma define obviously an open subset in $\Lam_\bd$.
Otherwise, $J_w$ is as a tilting module a finitely presented 
$\Lam$-module~\cite[Section~III.1]{BIRS}. Thus we get an injective resolution
\[
0\ra D(\Lam/J_w)\ra \bigoplus_{i\in Q_0} D(e_i\Lam)^{m_i}\ra 
\bigoplus_{i\in Q_0} D(e_i\Lam)^{n_i}\ra 0
\]
so that $\Hom_\Lam(X,D(J_w)) = 0$ represents also in this case an open condition.
\end{proof}

\subsection{The syzygy functor preserves irreducible components}
The results of this subsection will only be used in 
Section~\ref{ssec:pf2}, where we prove the second part
of Theorem~\ref{THM6}.
We consider for $\be \in \N^R$ the subset 
\[
\Lam^w_{\bd,\be} := \{X \in \Lam^w_\bd \mid \dimv\Hom_\Lam(V,X)= \be \}.
\]
By the upper semicontinuity of $\dim\Hom_\Lam(V_k,-)$ and
Lemma~\ref{lem:lamwopen}, $\LL^w_{\bd,\be}$
is a  locally closed (possibly empty) 
subset of $\Lam_\bd$. 
Note that for a given $\be$ there is at most one $\bd=\bd(\be)$
such that $\Lam^w_{\bd,\be} \not= \varnothing$. 
We showed in~\cite[Section~14]{GLSUni2}
that in case $\Lam^w_{\bd,\be} \not= \varnothing$, it is irreducible and
of the same dimension  as $\Lam_\bd$. 
In particular, if $\Lam^w_{\bd,\be} \not= \varnothing$ its Zariski closure is an irreducible 
component of $\Lam^w_\bd$ and each irreducible component of $\Lam^w_\bd$ is of
this form.

\begin{Prop}\label{loopcomp}
For $X \in \Lam^w_{\bd,\be}$ let 
$h\colon P \to X$ be a 
$\cC_w$-admissible epimorphism (\ie $P,X \in \cC_w$ and
$h$ is an epimorphism with $\Ker(h) \in \cC_w$), 
where $P$ is $\cC_w$-projective-injective. 
Then for $\bd' := \dimv_\Lam(P)-\bd$
and a unique $\be' \in \N^R$ there exists an irreducible variety
$\cE_{\be',\be}$ together with open morphisms 
$$
\xymatrix@-0.5pc{
& \cE_{\be',\be} \ar[dl]_{\pi'} \ar[dr]^{\pi} \\
\Lam^w_{\bd',\be'} && \Lam^w_{\bd,\be}
}
$$
such that for each $E \in \cE_{\be',\be}$ there exists a short exact sequence
$$
0 \ra \pi'(E) \xrightarrow{i} P \xrightarrow{p} \pi(E)\ra 0.
$$
\end{Prop} 

\begin{proof}
The proof consists of four steps.

\noindent{\rm (i)}\,
Let
$\tilde{\be}$ be the (componentwise) minimum
value of the map $\LL^w_{\bd,\be} \to \N^\Rmi$ defined
by 
$X \mapsto \dimv(\sHom_{\cC_w}(V,X)) = \dimv(D{\Ext}_\Lam^2(X,V))$.
Let $\Lam^{w,-}_{\bd,\be}$ be the open subset of $\LL^w_{\bd,\be}$ defined
by 
$$
\Lam^{w,-}_{\bd,\be} := \{ X \in \LL^w_{\bd,\be} \mid
 \dimv(\sHom_{\cC_w}(V,X)) = \tilde{\be} \}.
$$
Suppose that $P \in \Lam^w_{\bd'',\be''}$ and set $\be' := \be''-\be+\tilde{\be}$.
It is easy to see that for a short exact sequence 
\[
0 \ra X' \ra P \ra X \ra 0
\]
in $\cC_w$ with $X \in \Lam^w_{\bd,\be}$,  
we have $X' \in \Lam^w_{\bd',\be'}$ if and only if 
$X \in \Lam^{w,-}_{\bd,\be}$, since $\Ext^1_\Lam(V,X') \cong \sHom_{\cC_w}(V,X)$.

\noindent{\rm (ii)}\,
We claim that 
\begin{multline*}
\cE_{\be',\be} := \{ (X',i,p,X) \in
\Lam^w_{\bd',\be'} \times \Hom_{Q_0}(\C^{\bd'},\C^{\bd''}) \times
\Hom_{Q_0}(\C^{\bd''},\C^{\bd}) \times \Lam^w_{\bd,\be} \mid \\
i \in \Hom_\Lam(X',P) \text{ injective},
p \in \Hom_\Lam(P,X) \text{ surjective }, p \circ i = 0 \},
\end{multline*}
together with the obvious projections has the required properties.
By construction, we only have to show that $\pi$ and $\pi'$ are open.

\noindent{\rm (iii)}\,
As for the openness of $\pi$, consider the vector bundle  
\[
\cV_\be := \{ (p,X) \in \Hom_{Q_0}(\C^{\bd''},\C^\bd) \times \Lam^w_{\bd,\be} \mid 
p \in \Hom_\Lam(P,X) \}.
\]
This is a subbundle  of the trivial vector bundle 
$\Hom_{Q_0}(\C^{\bd''},\C^\bd) \times \Lam^w_{\bd,\be}$. 
In particular, the projection $\pi_2\df \cV_\be \ra \Lam^w_{\bd,\be}$, 
$(p,X)\mapsto X$ is an open morphism. 
The set
$$
\cV^\sur_\be := \{ (p,X)\in\cV_\be\mid p \text{ surjective} \}
$$ 
is a dense open subset of $\cV_\be$.

It is a standard argument to check that
\begin{multline*}
\cE_\be := \{ (X',i,p,X) \in
\Lam_{\bd'} \times \Hom_{Q_0}(\C^{\bd'},\C^{\bd''}) \times 
\Hom_{Q_0}(\C^{\bd''},\C^{\bd}) \times \Lam^w_{\bd,\be} \mid \\
i \in \Hom_\Lam(X',P) \text{ injective},
p \in \Hom_\Lam(P,X) \text{ surjective }, p \circ i = 0 \}
\end{multline*}
together with the obvious projection 
$$
\pi_{34}\df \cE_{\be} \ra \cV_\be^\sur,\, (X',i,p,X)\mapsto (p,X)
$$ 
is a $\GL_{\bd'}$-principal bundle. 
In particular, $\pi_{34}$ is open, and $\cE_\be$ is an irreducible variety.

Now,  by Lemma~\ref{lem:lamwopen}, and step~(i) of the proof,
$\cE_{\be',\be}$ is a dense open subset of $\cE_\be$. 
Thus $\pi$, as a composition of the open morphisms 
$\pi_2 \circ \pi_{34}\df \cE_\be \to \LL_{\bd,\be}^w$
and the inclusion $\cE_{\be',\be}\hookrightarrow \cE_\be$, is open.

\noindent{\rm (iv)}\, 
Let 
$$
\Lam^{w,+}_{\bd',\be'} := \{ X' \in \Lam^w_{\bd',\be'} \mid
\dimv\Ext^1_\Lam(V,X') = \tilde{\be} \}. 
$$
This is a locally closed subset of $\Lam^w_{\bd',\be'}$. 
A similar argument as in step~(iii) shows that the restriction 
$$
\cE_{\be',\be} \ra \Lam^{w,+}_{\bd',\be'}
$$
of $\pi'$ is open.
It remains to show that $\Lam^{w,+}_{\bd',\be'}$ is dense in $\Lam^w_{\bd',\be'}$.
To this end we note that there exist constants $f,f'\in\N$ such that
$$
\dim(\pi^{-1}(\pi(E))) = f
\quad\text{ and }\quad
\dim((\pi')^{-1}(\pi'(E))) = f'
$$ 
for all $E\in\cE_{\be',\be}$, see step~(iii) of the proof. 
Moreover, by Schanuel's Lemma, 
for $X\in\Bi(\pi)$ we have $\pi'(\pi^{-1}(\GL_\bd.X))=\GL_{\bd'}.X'$ for
some $X'\in\Bi(\pi')$. 
Thus, we have in this situation
\begin{multline} \label{eq:codim2}
\codim_{\Lam_\bd}(\GL_\bd.X)=\codim_{\Lam^w_{\bd,\be}}(\GL_\bd.X)\\
=\codim_{\Lam^{w,+}_{\bd',\be'}}(\GL_{\bd'}.X')\leq\codim_{\Lam_{\bd'}}(\GL_{\bd'}.X').
\end{multline}
On the other hand, since $\Ext^1_\Lam(X,X)\cong\Ext^1_\Lam(X',X')$, we have
by 
Lemma~\ref{lem:e2cdim} 
that 
$\codim_{\Lam_\bd}(\GL_\bd.X)=\codim_{\Lam_{\bd'}}(\GL_{\bd'}.X')$. 
This implies by~\eqref{eq:codim2} that
$$
\dim(\Lam^{w,+}_{\bd',\be'}) = \dim(\Lam^w_{\bd',\be'}) = \dim(\Lam_{\bd'}).
$$
\end{proof}

\subsection{Bongartz's bundle construction}\label{ssec:strat}
By Lemma~\ref{lem:lamwopen}, we know that
$\Lam^w_\bd$ is an open subset of  $\Lam_\bd$ for all $\dd$. 
The varieties
$\Lam_\bd$ and $\LL_\bd^w$ are equidimensional, and we know that
their irreducible components for all possible dimension vectors $\bd$ parametrize
the dual semicanonical bases $\cS^*$ of $\C[N]$, and
$\cS^*_w$ of $\C[N]^{N'(w)}$, respectively.

Given $\be \in \N^{\Rmi}$ and
$c \in \N$ we consider the locally closed subset
\begin{multline*}
\Lam_\bd^{(\be,c)} := \left\{ X \in \Lam_\bd^w \mid \dim\Ext^1_\Lam(T_k,X) = \be(k)
\text{ for } k \in \Rmi 
\text{ and } \dim \End_\Lam(X) = c \right\}.
\end{multline*}
We say that a subset $\cY \subseteq \Lam_\bd^w$ is a $T$-\emph{sheet of type} 
$\be$ if it is an irreducible component of some $\Lam_\bd^{(\be,c)}$.
In this case, we say that $\cY$ is \emph{dense} if the Zariski closure
$\overline{\cY}$ in $\Lam_\bd^w$ is an irreducible component of $\Lam_\bd^w$.

Clearly, $\Lam_\bd^w$ is a finite union of $T$-sheets, and each irreducible 
component contains a unique dense $T$-sheet.

\begin{Rem}
{\rm
(1) 
Let us recall some definitions and results from~\cite{CBS}. 
We write $\cZ = \cZ' \oplus \cZ''$ for irreducible components, say
$\cZ' \subseteq \Lam_{\bd'}^w$ and $\cZ'' \subseteq \Lam_{\bd''}^w$, if and only if
$\cZ\subseteq\Lam_{\bd'+\bd''}^w$ is an irreducible component which contains
a dense open subset $\cU$ such that for all $X\in\cU$ we have $X\cong X'\oplus X''$
for some $X'\in\cZ'$ and $X''\in\cZ''$. 
This is possible if and only if
$\Ext^1_\Lam(\cZ',\cZ'')=0$, \ie if there are dense open subsets 
$\cU'\subseteq\cZ'$ and $\cU''\subseteq\cZ''$ such that $\Ext^1_\Lam(X',X'')=0$
for all $X'\in\cU'$ and $X''\in\cU''$. 
(Note,  since $\scC_w$ is 
2-Calabi-Yau, the conditions $\Ext^1_\Lam(\cZ',\cZ'')=0$ and
$\Ext^1_\Lam(\cZ'',\cZ')=0$ are equivalent.) 
On the other hand, 
$\cZ\subseteq\Lam_\bd^w$ is by definition \emph{indecomposable}
if it contains a dense open subset $\cU$ such that all $X\in\cU$ are 
indecomposable.
For example,
since $\Ext^1_\Lam(T_k,T_k)=0$, one can apply Voigt's Lemma to show that
$$
\cT_k := \overline{\GL_\bd.T_k}
$$ 
is an indecomposable irreducible component.
(Here we set $\bd :=\dimv_\Lam(T_k)$.) 

With these definitions, each irreducible component
$\cZ$ admits an essentially unique decomposition 
into indecomposable irreducible components, see~\cite{CBS}.
 
(2) We say that an irreducible component $\cZ\subseteq\Lam_\bd^w$ is
\emph{generically} $T$-\emph{free} if in the decomposition into indecomposable
components, there is no summand of the form $\cT_k$ for any $1 \le k \le r$.
As a consequence of (1), each irreducible component $\cZ\subseteq\Lam_\bd^w$
can be written uniquely as
\[
\cZ = \cZ' \oplus \bigoplus_{k\in R} \cT_k^{m_k}
\]
with $\cZ'$ generically $T$-free.
If $\be\in\N^\Rmi$ is
the type of the unique dense $T$-sheet $\cY' \subseteq \cZ'$, then in the
above decomposition $m_k=0$ in case $\be(k)\neq 0$ by the definition of type. 
Moreover, the unique dense $T$-sheet $\cY\subseteq\cZ$ has the same type 
as $\cY'$.
}
\end{Rem}

We have the following
variant of a construction by Bongartz~\cite[Section~4.3]{Bo}:

\begin{Lem} 
Let $\cY \subseteq\Lam^w_\bd$ be a $T$-sheet of type $\be$.
Then there exists a
$\GL_\bd$-$\GL_\be$-variety $\cB_\cY$ together with morphisms
$$
\xymatrix@!@-3.2pc{
& \ar[dl]_{\pi_1} \cB_\cY \ar[dr]^{\pi_2} \\
\cY && \md(\scE,\be)
}
$$
such that
\begin{itemize}

\item
$\pi_1$ is a $\GL_\bd$-equivariant $\GL_\be$-principal bundle,

\item
$\pi_2$ is $\GL_\be$-equivariant and $\GL_\bd$-invariant,

\item
If $Y\in\cY$, then $\pi_2(B)\cong\Ext^1_\Lam(T,Y)$ for all $B\in\pi_1^{-1}(Y)$.

\end{itemize}
\end{Lem}

\begin{proof}
It is easy to derive 
from~\cite[Section~2.4]{Bo} that for any $k\in\Rmi$, the set
\[
\cE_{\cY,k}:= \left\{(Y,e)\mid Y\in\cY\text{ and } e\in\Ext^1_\Lam(T_k,Y)\right\}
\]
can be given the structure of an algebraic vector bundle of rank $\be(k)$
over $\cY$. It follows that
\begin{multline*}
\cB_{\cY} := \{ (Y,(v_{k,l})_{k\in\Rmi,1\leq l\leq\be(k)}) \mid Y \in \cY 
\text{ and } (v_{k,l})_{1 \le l \le \be(k)} \\
\text{ is a basis of } \Ext^1_\Lam(T_k,Y) \text{ for all } k\in\Rmi \}
\end{multline*}
is together with the obvious projection $\pi_1$ a $\GL_\bd$-equivariant 
$\GL_\be$-principal bundle over $\cY$.
Then one proceeds as in~\cite[Section~14]{GLSUni2}. 
\end{proof}

\begin{Prop} \label{prp:bas}
For $\cZ \in \irr(\md(\scE,\be))$ the following are equivalent:
\begin{itemize}

\item[(i)]
$\cZ$ is strongly reduced.

\item[(ii)]
$\cZ = \overline{\pi_2(\cB_\cY)}$ for some dense, generically
$T$-free $T$-sheet of type $\be$. 

\end{itemize}
In this case, $\cY$ is uniquely determined, 
and $\cZ = \overline{\pi_2(\cB_{\cY'})}$ precisely for the dense $T$-sheets
$$
\cY' \subseteq \overline{\cY} \oplus \bigoplus_{k\in\Cpl(\cZ)} \cT_k^{m_k}
$$ 
with $m_k\in\N$.
\end{Prop}

\begin{proof}
Since each $M \in \md(\scE)$ is of the form $M\cong\Ext^1_\Lam(T,X)$ for
some $X\in\cC_w$ we conclude that $\cZ$ is a countable union of constructible
sets of the form $\pi_2(\cB_\cY)$ for certain $T$-sheets $\cY$ of type $\be$.
Since $\C$ is not countable, the Baire category theorem implies that 
$\cZ = \overline{\pi_2(\cB_\cY)}$ for one of these $T$-sheets, say 
$\cY \subseteq \LL_\bd^w$.

There is some 
$c \in \N$ such that $\codim_\cY(\GL_\bd.Y)=c$ for all $Y\in\cY$ 
by the definition of the $T$-sheets. 
We claim that then $\pi_2(\cB_\cY)$ contains a dense open
subset $\cM$ such that $\codim_\cZ(M.\GL_\be)=c$ for all $M\in\cM$. Indeed, for
any $M\in\pi_2(\cB_\cY)$ we have $\codim_{\cB_\cY}(\pi^{-1}_2(M.\GL_\be))=c$ since
there are only finitely many orbits, say  $\GL_\bd.Y_s$ for $1 \le s \le t$ 
in $\cY$, such that $\Ext^1_\Lam(T,Y_s)\cong M$, so that
\[
\pi^{-1}_2(M.\GL_\be)=\bigcup_{s=1}^t \pi^{-1}_1(\GL_\bd.Y_s).
\]
Now, $\codim_{\cB_\cY}(\pi^{-1}_1(\GL_\bd.Y))=\codim_\cY(\GL_\bd.Y) = c$ 
for all $Y\in\cY$ since $\pi_1$ is a principal bundle. So our claim 
follows from Chevalley's theorem.

Finally, let $h:=\codim_{\Lam^w_\bd}(\cY)$. 
Then $\codim_{\Lam^w_\bd}(\GL_\bd.Y)=c+h$
for all $Y\in\cY$. Thus, since each $M\in\pi_2(\cB_\cY)$ is of the form
$M\cong\Ext^1_\Lam(T,Y)$ for some $Y\in\cY$,  we have by
Proposition~\ref{prp:EinvExt},
\[
c+h = \dim\Hom_\scE(M,\tau_\scE(M)).
\]
Thus, $\cZ$ is strongly reduced if and only if $h=0$.
The rest is clear by Remark~8.4(2).
\end{proof}

\subsection{Proof of Theorem~\ref{THM4}}
Consider the
epimorphism
$$
\Pi_T \df \C[N^w]\to \cA(\sGG_T)
$$
defined in Section~\ref{section1.5}.
Then, by 
Theorem~\ref{THM3} and our results from Section~\ref{Section5.1},  
for $X \in \cC_w$ we get 
$$
\Pi_T(\vph_X) = 
\Pi_T(\theta_X^T) =
x^{\mm} \cdot \psi_{\Ext^1_\Lam(T,X')}
$$
where we write 
$$
X = X' \oplus \bigoplus_{k=1}^r T_k^{m_k}
$$ 
in such a way that $X'$ has no
direct summand from $\add(T)$, and we set
$\mm := (m_k)_{k \in \Rmi}$.
Thus, by Proposition~\ref{prp:bas}, the
image of the dual semicanonical basis of $\C[N^w]$ under $\Pi_T$ is
just the generic basis $\cG_w^T$.
(This is indeed a basis by~\cite[Section~15]{GLSUni2}.)
This finishes the proof.

\subsection{An example}\label{genericex}
We continue with the example from Section~\ref{chamberex}.
Let $W := W_\ii = W_1 \oplus \cdots \oplus W_6$.
The quiver $\GG_W$ of $\cE_W$ looks as follows:
$$
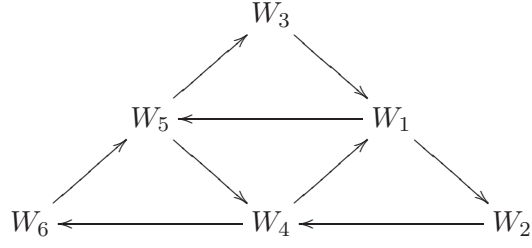
\xymatrix{
&& W_3 \ar[dr]\\
& W_5 \ar[ur]\ar[dr] && 
W_1 \ar[ll]\ar[dr]\\
W_6 \ar[ur] && 
W_4 \ar[ll]\ar[ur] && 
W_2 \ar[ll]
}
$$
An easy computation yields
$$
B^{(W)} = \left( \bbm 0&-1&1&1&-1&0\\1&0&0&-1&0&0\\
-1&0&1&0&0&0\\-1&1&0&0&1&-1\\1&0&-1&-1&1&0\\0&0&0&1&-1&1 \ebm\right)
\text{\;\;\; and \;\;\;}
(B^{(W)})^{-1} = \left( \bbm 1&1&0&1&1&1\\0&1&0&1&0&1\\
1&1&1&1&1&1\\1&0&0&1&1&1\\1&0&1&1&2&1\\0&0&1&0&1&1 \ebm\right).
$$
Note that $\scE := \scE_W$ is the path algebra of the quiver $\sGG_W$
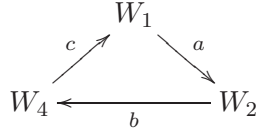
$$
\xymatrix@-0.5pc{
& W_1 \ar[dr]^a \\ 
W_4 \ar[ur]^c && W_2 \ar[ll]^b 
}
$$
modulo the ideal generated by $\{ ba, cb, ac\}$.
We have 
$$
\sB^{(W)} = \left( \bbm 0&-1&1\\1&0&-1\\
-1&1&0 \ebm\right).
$$
Let $S_1,S_2,S_4$ be the simple $\scE$-modules corresponding to the vertices of
$\sGG_W$, and let
$P_1,P_2,P_4$ and $I_1,I_2,I_4$ be their projective covers
and injective envelopes, respectively.
One easily checks that $I_1 = P_4$, $I_2 = P_1$ and $I_4 = P_2$,
and that $S_1,S_2,S_4,I_1,I_2,I_4$ are the only indecomposable $\scE$-modules
up to isomorphism.
Let us write these modules as representations of $\sGG_W$:
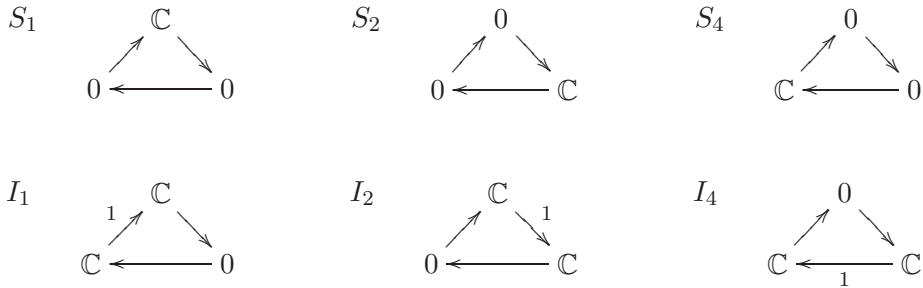
\begin{align*}
{\xymatrix@-1pc{S_1 && \C \ar[dr] \\ &0 \ar[ur] && 0 \ar[ll]}} &&
{\xymatrix@-1pc{S_2 && 0 \ar[dr] \\ &0 \ar[ur] && \C \ar[ll]}} &&
{\xymatrix@-1pc{S_4 && 0\ar[dr] \\ &\C \ar[ur] && 0 \ar[ll]}} \\[20pt]
{\xymatrix@-1pc{I_1 && \C \ar[dr] \\ &\C \ar[ur]^1 && 0 \ar[ll]}} &&
{\xymatrix@-1pc{I_2 && \C \ar[dr]^1 \\ &0 \ar[ur] && \C \ar[ll]}} &&
{\xymatrix@-1pc{I_4 && 0 \ar[dr] \\ &\C \ar[ur] && \C \ar[ll]^1}}
\end{align*}
Next, we determine the $\scE$-modules $\Ext_\LL^1(W,X)$, where $X$ runs
through all 12 indecomposable $\LL$-modules. 
\begin{align*}
\Ext_\LL^1(W,V_1) &= I_1 &
\Ext_\LL^1(W,V_2) &= I_2 &
\Ext_\LL^1(W,V_4) &= I_3 \\
\Ext_\LL^1(W,L_1) &= S_1 &
\Ext_\LL^1(W,L_2) &= S_2 &
\Ext_\LL^1(W,L_4) &= S_4
\end{align*}
and we have $\Ext_\LL^1(W,W_k)  = 0$ for all $1 \le k \le 6$.
Since $\scE$ is a representation-finite algebra (it has only 6 indecomposable modules),
each irreducible component in $\irr(\scE,\be)$ is the closure of some $\GL_\be$-orbit.
For an $\scE$-module $Y$ in $\md(\scE,\be)$ let $\orb_Y := Y.\GL_\be$ be its
$\GL_\be$-orbit, and let $\ov{\orb_Y}$ be the Zariski closure of $\orb_Y$.
We get
$$
\irr(\scE) = \left\{ \ov{\orb_Y} \mid Y = I_1^{a_1} \oplus I_2^{a_2} \oplus I_4^{a_4} \oplus
S_k^{s_k} \mid a_1,a_2,a_4,s_k \ge 0,\, k = 1,2,4 \right\}.
$$
An easy calculation shows that
\begin{multline*}
\srirr(\scE) = \{ \ov{\orb_Y} \mid Y = I_1^{a_1} \oplus I_2^{a_2} \oplus I_4^{a_4} \oplus
S_k^{s_k} \mid a_1,a_2,a_4,s_k \ge 0,\, k = 1,2,4,  \\
s_1a_4 = s_2a_1 = s_4a_2 = 0 \}. 
\end{multline*}
Next, we compute the
functions $\psi_Y$, where $Y$ runs through the
6 indecomposable $\scE$-modules.
Observe that $\hx_{W,1} = x_2x_4^{-1}$, 
$\hx_{W,2} = x_1^{-1}x_4$ and 
$\hx_{W,4} = x_1x_2^{-1}$.
We obtain
\begin{align*}
\psi_{S_1} 
&= x_1^{-1}x_4(1 \cdot \hx_W^{(0,0,0)}) + 1 \cdot \hx_W^{(1,0,0)})  
= x_1^{-1}x_4(1 + x_2x_4^{-1}) \\
\psi_{S_2} 
&= x_1x_2^{-1}(1 \cdot \hx_W^{(0,0,0)}) + 1 \cdot \hx_W^{(0,1,0)})  
= x_1x_2^{-1}(1 + x_1^{-1}x_4) \\
\psi_{S_4} 
&= x_2x_4^{-1}(1 \cdot \hx_W^{(0,0,0)}) + 1 \cdot \hx_W^{(0,0,1)})  
= x_2x_4^{-1}(1 + x_1x_2^{-1}) \\
\psi_{I_1} 
&= x_1^{-1}(1 \cdot \hx_W^{(0,0,0)}) + 1 \cdot \hx_W^{(1,0,0)}  + 1 \cdot \hx_W^{(1,0,1)})  
= x_1^{-1}(1 + x_2x_4^{-1} + x_2x_4^{-1}x_1x_2^{-1}) \\
\psi_{I_2} 
&= x_2^{-1}(1 \cdot \hx_W^{(0,0,0)}) + 1 \cdot \hx_W^{(0,1,0)}  + 1 \cdot \hx_W^{(1,1,0)})  
= x_1^{-1}(1 + x_1^{-1}x_4 + x_2x_4^{-1}x_1^{-1}x_4) \\
\psi_{I_4} 
&= x_4^{-1}(1 \cdot \hx_W^{(0,0,0)}) + 1 \cdot \hx_W^{(0,0,1)}  + 1 \cdot \hx_W^{(0,1,1)})  
= x_1^{-1}(1 + x_1x_2^{-1} + x_1^{-1}x_4x_1x_2^{-1}) 
\end{align*}
The basis $\cG_w^W$ consists then of the following 14 sets of monomials:
\begin{align*}
x_1^ax_2^bx_3^c && 
\psi_{I_1}^a\psi_{I_2}^b\psi_{I_4}^c \\[10pt]
\psi_{S_1}^a\psi_{I_1}^b\psi_{I_2}^c &&
x_2^a\psi_{S_1}^b\psi_{I_1}^c &&
x_4^a\psi_{S_1}^b\psi_{I_2}^c &&
x_2^ax_4^b\psi_{S_1}^c \\[10pt]
\psi_{S_2}^a\psi_{I_2}^b\psi_{I_4}^c &&
x_4^a\psi_{S_2}^b\psi_{I_2}^c &&
x_1^a\psi_{S_2}^b\psi_{I_4}^c &&
x_1^ax_4^b\psi_{S_2}^c \\[10pt]
\psi_{S_4}^a\psi_{I_4}^b\psi_{I_1}^c &&
x_1^a\psi_{S_4}^b\psi_{I_4}^c &&
x_2^a\psi_{S_4}^b\psi_{I_1}^c &&
x_1^ax_2^b\psi_{S_4}^c
\end{align*}
where $a,b,c \ge 0$.

For example, from
our calculations in Section~\ref{chamberex} we get
$$
\vph_{V_1} = \vph_{W_6}\vph_{W_4}^{-1} + 
\vph_{W_4}^{-1}\vph_{W_5}\vph_{W_1}^{-1}\vph_{W_2} +
\vph_{W_3}\vph_{W_1}^{-1}.
$$
As predicted by Theorem~\ref{THM4} we get
$$
(\Pi_T \circ \Phi_T)(\vph_{V_1}) = x_4^{-1} + x_4^{-1}x_1^{-1}x_2 +
x_1^{-1} = \psi_{I_1}.
$$


\section{Categorification of the twist automorphism}\label{section8}


\subsection{}\label{ssec:pf2}
We define an isomorphism
$$
\kappa_\ii\df \C[\vph_{V_{\ii,1}}^{\pm 1},\ldots,\vph_{V_{\ii,r}}^{\pm 1}] \to 
\C[\vph_{W_{\ii,1}}^{\pm 1},\ldots,\vph_{W_{\ii,r}}^{\pm 1}] 
$$
of Laurent polynomial rings
by
$$
\vph_{V_{\ii,k}} \mapsto \vph_{V_{\ii,k}}'
$$
for $1 \le k \le r$.
By the \emph{Laurent phenomenon} \cite{FZ1}, each cluster variable of 
a cluster algebra is a Laurent polynomial in the 
cluster variables of any given cluster.
It follows that $\C[N^w]$ is a subalgebra of both 
$\C[\vph_{V_{\ii,1}}^{\pm 1},\ldots,\vph_{V_{\ii,r}}^{\pm 1}]$ and
$\C[\vph_{W_{\ii,1}}^{\pm 1},\ldots,\vph_{W_{\ii,r}}^{\pm 1}]$. 
Here we use that $W_\ii$ is $V_\ii$-reachable.

By Theorem~\ref{THM3} and the definition of $\kappa_\ii$ we have for $X\in\cC_w$,
$$
\kappa_\ii(\vph_X) = 
\kappa_\ii(\theta^{V_\ii}_X)= 
(\vph'_\bul)^{(\dimv\Hom_\Lam(V_\ii,X)) \cdot B^{(V_\ii)}}
F^{W_\ii}_{\Omega_w(X)}(\hvph'_\bul).
$$
Here we also used that
$\Ext^1_\Lam(V_\ii,X)$ and $\Ext^1_\Lam(W_\ii,\Omega_w(X))$ are isomorphic as
modules over
$\sEnd_{\cC_w}(V_\ii)^\op \cong \sEnd_{\cC_w}(W_\ii)^\op$. 
Now, using the short exact sequence
\[
0\ra \Omega_w(X) \ra P(X) \ra X \ra 0
\]
where $P(X)$ is $\cC_w$-projective-injective,
we get
\[
(\vph'_\bul)^{(\dimv\Hom_\Lam(V_\ii,X))\cdot B^{(V_\ii)}}=
(\vph'_\bul)^{(\dimv\Ext^1_\Lam(V_\ii,\Omega_w(X)) - 
\dimv\Hom_\Lam(V_\ii,\Omega_w(X)))\cdot B^{(V_\ii)}}
\vph_{P(X)}^{-1}.
\]
Here we used that
\[
(\vph'_\bul)^{(\dimv\Hom_\Lam(V_\ii,P(X)))\cdot B^{(V_\ii)}}=\vph_{P(X)}^{-1}.
\]
Thus, we can continue with the help of Proposition~\ref{prp:exp} 
and Theorem~\ref{THM3}:
\begin{align*}
(\vph'_\bul)^{(\dimv\Hom_\Lam(V_\ii,X))\cdot B^{(V_\ii)}} F^{W_\ii}_{\Omega_w(X)}
(\hvph'_\bul) &= \vph_{W_\ii}^{(\dimv\Hom_\Lam(W_\ii,\Omega_w(X)))\cdot B^{(W_\ii)}} F^{W_\ii}_X(\hvph_{W_\ii}^\pex)\vph_{P(X)}^{-1}\\
&=
\vph^\pex_X\vph_{P(X)}^{-1}.
\end{align*}
This proves that $\kappa_\ii$ does not depend on the choice of 
the reduced word $\ii$. Thus we can denote by $\kappa_w$ the 
restriction of $\kappa_\ii$ to $\C[N^w] \subset \C[\vph_{V_{\ii,1}}^{\pm 1},\ldots,\vph_{V_{\ii,r}}^{\pm 1}]$.
Moreover, the fact that $\kappa_w$ permutes the elements of the
dual semicanonical basis follows
directly from Proposition~\ref{loopcomp} and
our results in \cite[Section~15]{GLSUni2}.
It now remains to prove that $\kappa_w = (\eta_w^*)^{-1}$.

\subsection{Definitions and known results}
Before we proceed, we recall some
definitions and known results.
For more details we refer to \cite{GLSUni2}.

Let $u\in W$.
We denote by $D_{\varpi_i,u(\varpi_i)}$ the restriction to $N$
of the generalized minor $\Delta_{\varpi_i,u(\varpi_i)}$.
Let $O_w$ be the open subset of $N$ defined by
\[
 O_w := \left\{ x\in N \mid D_{\varpi_j,w^{-1}(\varpi_j)}(x) \not = 0 \text{ for all } 
1 \le j \le n \right\}.
\]
Following \cite{BZ} and \cite{GLSUni2}, we introduce a map
$\widetilde{\eta}_w : N\cap O_w \to N^w$ by 
\[
\widetilde{\eta}_w(x) = [wz^T]_+. 
\]
Here, 
for $g$ in the Kac-Moody group of $\g$ admitting
a Birkhoff decomposition, $[g]_+$ stands for the factor of this
decomposition belonging to $N$.
Let
$N(w) = N \cap (w^{-1}N_-w)$ and  $N'(w) = N \cap (w^{-1}Nw)$ be the 
unipotent groups  associated to $w$ \cite[Sections~5.2 and~8.2]{GLSUni2}. 
Multiplication in $N$ induces a bijective map
$N(w) \times N'(w) \to N$. 
The ring of $N'(w)$-invariant functions on $N$, denoted by $\C[N]^{N'(w)}$, is thus
isomorphic to $\C[N(w)]$.

The restriction of $\widetilde{\eta}_w$ to $N(w)\cap O_w$ is an isomorphism
from $N(w)\cap O_w$ to $N^w$. 
Also, $N^w \subset O_w$, and the restriction
of $\widetilde{\eta}_w$ to $N^w$ is precisely the automorphism $\eta_w$ of $N^w$ mentioned 
in Section~\ref{sect_twist}.

Fix $x\in N^w$, and set $z = \eta_w^{-1}(x) \in N^w$.
Also let $y$ be the unique element of $N(w)\cap O_w$ such that 
$\widetilde{\eta}_w(y) = x$. It is known that $z^{-1}y\in N'(w)$.
Hence, for every $\varphi\in\C[N]$ invariant by right translation by $N'(w)$,
we have $\varphi(z)=\varphi(y)$.

Finally, let $\ii=(i_r,\ldots,i_1)$ be a reduced word for $w$.
We know that when ${\bf t}$ varies over $(\C^*)^r$, 
$\sx_\ii({\bf t})$ goes over a dense subset of $N^w$.
For $1 \le l \le k \le r$, we set $w_k := s_{i_k}\cdots s_{i_1}$ and
$$
\beta_\ii(k) := w_{k-1}^{-1}(\alpha_{i_k}) = s_{i_1} \cdots s_{i_{k-1}}(\alpha_{i_k}).
$$
As before, let
$b_\ii(l,k) := -(s_{i_l}\cdots s_{i_k}(\varpi_{i_k}),\alpha_{i_l})$.
Note that $b_\ii(l,k) \ge 0$.

\subsection{End of the proof of Theorem~\ref{THM6}}
To prove that $\kappa_w = (\eta_w^*)^{-1}$,
we have to show that
\[
 (\eta_w^{-1})^*(\varphi_{V_k}) = \frac{\varphi_{\Omega_w(V_k)}}{\varphi_{P(V_k)}} = \varphi_{V_k}'
\]
for all $1 \le k \le r$.

We know that $\varphi_{V_k} = D_{\varpi_{i_k},w_k^{-1}(\varpi_{i_k})}$.
Let $x := \sx_\ii({\bf t})$ where $\bt \in (\C^*)^r$,
and attach to $x$ the elements $y$ and $z$ as above.
By Proposition~\ref{ch4} we have
\[
\varphi_{V_k}'(x) = \prod_{1\le l\le k} t_l^{-b_{\ii}(l,k)}. 
\]
Hence, it is enough to show that
\begin{equation}\label{2prove}
D_{\varpi_{i_k},w_k^{-1}(\varpi_{i_k})}(z) = D_{\varpi_{i_k},w_k^{-1}(\varpi_{i_k})}(y) 
=\prod_{1\le l\le k} t_l^{-b_\ii(l,k)},
\end{equation}
where the first equality follows from the fact that
$D_{\varpi_{i_k},w_k^{-1}(\varpi_{i_k})}$ is $N'(w)$-invariant for $1 \le k\le r$.

The proof is very similar to that of \cite{BZ}, so we just recall the
main ideas, referring to appropriate places in \cite{BZ} for some simple calculations.
There are two steps.

\noindent(a)\, 
We first show (\ref{2prove}) in the particular case when 
$k = k_j := \max\{ s \in R \mid i_s = j \}$ for a given $1 \le j \le n$.
In this case, (\ref{2prove}) can be written as
\begin{equation}
D_{\varpi_j,w^{-1}(\varpi_j)}(y)=\prod_{1\le l \le k} t_l^{-b_\ii(l,k)}. 
\end{equation}
Since $\varphi'_{I_{\ii,j}} = 1/\varphi_{I_{\ii,j}}$, already know
by Proposition~\ref{ch4} that
\[
D_{\varpi_j,w^{-1}(\varpi_j)}(x) = \varphi_{I_{\ii,j}}(x) = \prod_{1\le l \le k} t_l^{b_\ii(l,k)}.
\]
Hence it is enough to check that 
\begin{equation}
 D_{\varpi_j,w^{-1}(\varpi_j)}(y)D_{\varpi_j,w^{-1}(\varpi_j)}(x) = 1.
\end{equation}
This is proved in exactly the same way as in \cite[Lemma 6.4 (a)]{BZ}, using some
basic properties of generalized minors.

\noindent(b)\, 
We then show (\ref{2prove}) for any $1 \le k\le r$. 
We write $x = x''x'$, where
\[
x' := x_{i_k}(t_k)\cdots x_{i_1}(t_1)
\qquad\text{ and }\qquad
x'' := x_{i_r}(t_r)\cdots x_{i_{k+1}}(t_{k+1}).  
\]
Let $N(\beta_{\ii}(k)) = \exp(\n_{\beta_{\ii}(k)})$ denote the root subgroup of $N(w)$
associated to the root $\beta_{\ii}(k)$.
The product map gives an isomorphism of affine varieties 
$$
N(\beta_{\ii}(1)) \times \cdots \times N(\beta_{\ii}(r)) \to N(w).
$$
Therefore we can write $y = y^{(1)} \cdots y^{(r)}$ with $y^{(k)}\in N(\beta_{\ii}(k))$.
Arguing as in \cite[Propositions~5.3 and~5.4]{BZ}, one shows that
\[
 \widetilde{\eta}_{w_k}(y^{(1)} \cdots y^{(k)}) = x_{i_k}(t_k) \cdots x_{i_1}(t_1) = x'.
\]
Moreover, $y':= y^{(1)} \cdots y^{(k)} \in N(w_k)$ 
and $y^{(k+1)},\ldots,y^{(r)} \in N'(w_k)$.
Therefore, since $D_{\varpi_{i_k},w_k^{-1}(\varpi_{i_k})}$ is $N'(w_k)$-invariant, 
we have 
\[
D_{\varpi_{i_k},w_k^{-1}(\varpi_{i_k})}(y) = D_{\varpi_{i_k},w_k^{-1}(\varpi_{i_k})}(y'). 
\]
Now, using (a) with $w$ replaced by $w_k$, we obtain that
\[
 D_{\varpi_{i_k},w_k^{-1}(\varpi_{i_k})}(y') =  \prod_{1 \le l \le k} t_l^{-b_\ii(l,k)}.
\]
This finishes the proof.

\subsection*{Acknowledgments} 
The authors thank Bernhard Keller for his help with
Proposition~\ref{prp:EinvExt}.
The authors thank the Sonderforschungsbereich/Transregio SFB 45 
and the Hausdorff Center for Mathematics in Bonn
for financial support.
The first author was partially supported by PAPIIT grant IN103507-2.


\end{document}